\documentclass[final]{siamonline190516}

\usepackage{amsmath}
\usepackage{mathtools}	
\usepackage{amssymb}

\usepackage{graphicx}

\usepackage{blindtext}

\usepackage{xcolor}

\usepackage{subcaption}

\usepackage{url}

\usepackage{cleveref}

\usepackage{accents}

\usepackage{leftidx}

\usepackage{algorithm}
\usepackage{algorithmic}

\usepackage{thmtools}
\usepackage{thm-restate}

\usepackage{tikz}
\usetikzlibrary{shapes}
\usetikzlibrary{calc}
\usetikzlibrary{fit}

\newtheorem{remark}{Remark}
\newtheorem{exmp}{Example}[section]

\newtheorem{defi}{Definition}[section]

\crefname{exmp}{Example}{Examples}
\newcommand{\xset}{\mathbb{X}}

\newcommand{\yset}{\mathbb{Y}}

\newcommand{\permut}{\sigma}

\newcommand{\type}[1]{\mathcal{T}_\mathcal{#1}}

\newcommand{\tset}{T}

\newcommand{\stirst}[2]{\mathcal{S}_{1}(#1,#2)}

\newcommand{\stirstg}[3]{\mathcal{S}^{#1}_{1}(#2,#3)}

\newcommand{\stirnd}[2]{\mathcal{S}_{2}(#1,#2)}

\title{Decomposition of admissible functions in weighted coupled cell networks\thanks{First submitted to the editors at 12 of January of 2022. Revised version submited at 24 of August of 2022. Accepted at 9 of November of 2022.
		\funding{This work was supported in part by the FCT Project RELIABLE (PTDC/EEI-AUT/3522/2020), funded by FCT/MCTES and in part by the National Science Foundation under Grant No.~ECCS-2029985. The work of P. Sequeira was supported by a Ph.D. Scholarship, grant SFRH/BD/119835/2016 from Fundação para a Ciência e a Tecnologia (FCT), Portugal (POCH program).}}}

\author{Pedro Sequeira\thanks{Faculdade de Engenharia, Universidade do Porto, Portugal (\email{pedro.sequeira@fe.up.pt}, \email{pedro.aguiar@fe.up.pt}).}
	\and João P. Hespanha\thanks{Department of Electrical and Computer Engineering, University of California, Santa Barbara, CA (\email{hespanha@ece.ucsb.edu}).}
	\and A. Pedro Aguiar\footnotemark[2]
}

\headers{Decomposition of admissible functions in weighted coupled cell networks}{Pedro Sequeira, João P. Hespanha, and A. Pedro Aguiar}

\begin{document}

\maketitle

\begin{abstract}
This work makes explicit the degrees of freedom involved in modeling the dynamics of a network, or some other first-order property of a network, such as a measurement function. 
In previous work, an admissible function in a network was constructed through the evaluation of what we called oracle components. These oracle components are defined through some minimal properties that they are expected to obey. This is a high-level description in the sense that it is not clear how one could design such an object. The goal is to obtain a low-level representation of these objects by unwrapping them into their degrees of freedom. To achieve this, we introduce two decompositions. The first one is the more intuitive one and allows us to define the important concept of coupling order. The second decomposition is built on top of the first one and is valid for the class of coupling components that have finite coupling order. Despite this requirement, we show that this is still a very useful tool for designing coupling components with infinite coupling orders, through a limit approach.
\end{abstract}
\begin{keywords}
	Coupled cell networks, Admissibility, Function decomposition, Stirling numbers
\end{keywords}
\begin{AMS}
	34A34, 41A63, 11B73
\end{AMS}
\section{Introduction}
Networks are ubiquitous structures, be it either in the natural world or in engineering. In order to study dynamical systems associated with such structures, the groupoid formalism of \textit{coupled cell networks} (CCN) was introduced in 
\cite{stewart2003symmetry,golubitsky2005patterns,golubitsky2006nonlinear}. Here, the concept of \textit{admissibility} was defined through the minimal properties that a function must obey in order for it to be a plausible modeling of the dynamics (or some other first-order property) of a network. Here, ``first-order" means that we are modeling something that, when evaluated at cell, depends on the state of that cell and its in-neighbors. This does not mean that everything on a network has to (or can) be defined by such a function. For instance, the second derivative or the two-step evolution of the mentioned dynamical systems will not be of this form. Those functions will be ``second-order" in the sense that they are dependent on their first and second in-neighborhoods (neighbor of neighbor). They are, however, fully defined from the original first-order functions.
This concept of order should not be confused with ``coupling order", which is the focus of this work and is something entirely different.\\
Although the groupoid formalism is general enough to cover admissible functions with more complex structure, many important models are usually given by very simple dynamical functions, such as being ``additive in the edges/weights" or being ``weakly coupled".
The Kuramoto model \cite{arenas2008synchronization}, \cite{dorfler2014synchronization}, \cite{rodrigues2016kuramoto}, is one of the most predominant models in the study of synchrony of oscillators and, although it has many variants, it is often assumed to have this simple structure.\\
Nevertheless, the importance of studying systems with higher order couplings has been recognized \cite{battiston2021physics}, as reviewed in \cite{bick2021higher,battiston2020networks}.\\
Refer for instance to \cite{memmesheimer2012non}, in which it is experimentally observed that changing additive coupling dynamics to non-additive can enable persistent synchrony, with this phenomenon appearing even in random networks, with no pre-constructed graph structure that would justify the existence of synchronism. The inability of the additive coupling system to exhibit such a feature might mean that such a system is, in some sense, degenerate.\\
This has lead to many works that extend the concept of network, such as hypernetworks \cite{aguiar2022network} and simplicial complexes
\cite{nijholt2022dynamical}.\\
The generalization of networks into more complex, higher dimensional structures is justified through the objective of constructing admissible functions that have non-pairwise terms. While it is true that these structures allow for this type of terms, this is not a requirement. That is, standard networks are perfectly capable of having higher order, non-pairwise terms, although they are constrained in a very particular way, as we show in this work.\\
There have been some extensions of the Kuramoto model to higher coupling orders \cite{aguiar2018synchronization,ashwin2016hopf,bick2016chaos}. While these functions certainty are invariant to permutations (\cref{thm:oracle_permut} of \cref{defi:oracle}) and dependent only on the cell of reference and  its in-neighbors (\cref{thm:oracle_0_equal} of \cref{defi:oracle}), it is not clear whenever they follow the edge-merging principle (\cref{thm:oracle_merge} of \cref{defi:oracle}), which is a very strong constraint. Note that this last condition was already present in the original groupoid formalism and it is essential in order to properly define quotient networks. The weighted formalism used here only makes it more explicit.\\
The groupoid formalism has previously been  applied to the case of networks with weighted edges \cite{aguiar2017patterns}, where the weights could assume real values. However, here the admissible functions were considered to only admit a very simple additive in the weights structure, which is far too restrictive for our purposes.
For this reason, we use the formalism for general weighted CCNs recently introduced in \cite{sequeira2021commutative}, which is a proper generalization of the groupoid formalism. This formalism uses the algebraic structure of the commutative monoid to deal with arbitrary edge sets. This is the minimal structure, with the necessary symmetry properties, that is able to encode finite edges in parallel.\\
Much more important than the extension to general weights, is the development of the concept of oracle components.
An oracle component is a mathematical object that describes how cells of a given type respond to arbitrary finite in-neighborhoods. It completely separates the modeling of the behavior of cells from the particular network on which the cells of interest are inserted.\\
Then, to specify an admissible function on a CCN, which models its dynamics (or an output function in general), we just need  to choose a tuple of oracle components (one for each cell type), which is called an oracle function. The admissible function is then obtained by simply evaluating on each cell, together with its corresponding in-neighborhood, the appropriate oracle component.\\
Note that the oracle component is a much preferable mathematical object to work with than the admissible function. That is, in order to study a function, we would rather know it completely than just knowing its value when evaluated at some points. In particular, despite the fact that in most applications we might not have to deal with cells that have arbitrarily large in-neighborhoods, it proves essential for the oracle components to be properly defined in such cases.\\
The oracle components are defined through a set of equality constraints to itself. Although in this work we define them in a simpler way than in \cite{sequeira2021commutative}, it is still a high-level description in the sense that it is not clear how one could construct such an object. In this work, we build on top of this formalism with the intent of dissecting the mathematical objects that are the oracle components by making explicit their degrees of freedom.\\
To the given set of cell types $ \tset $, we associate the state sets $ \{\xset_i\}_{i\in\tset} $, which are used to construct the domains of oracle components and the output sets $ \{\yset_i\}_{i\in\tset} $ which are the corresponding codomains. This work requires $ \{\yset_i\}_{i\in\tset} $ to be vector spaces in order for the decompositions to be well defined. Note that this is still fairly general. In particular, we might consider the state sets $ \{\xset_i\}_{i\in\tset} $ to be manifolds and the output sets to be their tangents spaces, that is, $ \yset_i = T_p\xset_{i} $, which are indeed vector spaces. Therefore, we can apply these results to spaces other than $ \mathbb{R}^n $, such as spaces involving angles (torus) in the study of oscillators.\\
Furthermore, we always assume that the scalar fields associated with the vector spaces contain at least the rational numbers. This covers, for instance, real and complex vector spaces, which are the most commonly used ones. Some results involve convergence and require output sets to have a topology defined on them. These topologies need very tame assumptions, such as being Hausdorff. Note that in this work some assumptions could be weakened. For instance, \cref{sec:decomp_coupl_comp} only requires additions and subtractions, which means that we could use commutative (often called abelian) groups instead of vector spaces. However, we prefer to have these results in their vector space form in order to apply them easily in \cref{sec:decomp_basis_comp}.\\
In \cref{sec:new_formalism} we provide the necessary background for understanding the rest of the work by summarizing the commutative monoid formalism for general weighted CCNs, which we present in a simpler form than in the original paper \cite{sequeira2021commutative}.\\
In \cref{sec:decomp_coupl_comp} we study what we call a coupling decomposition. We decompose an oracle component $ \hat{f}_i $ by establishing a bijection between it and an infinite family of functions $ \{f_i^{\mathbf{k}}\}_{\mathbf{k} \geq\mathbf{0}_{\vert\tset\vert}} $, called coupling components. That is, they are equivalent representations of the same object. This does not yet solve the problem of understanding the degrees of freedom involved in modeling oracle components $ \hat{f}_i $, since the resulting coupling components $ \{f_i^{\mathbf{k}}\}_{\mathbf{k} \geq\mathbf{0}_{\vert\tset\vert}}  $ are all related to each other. It is, however, an essential first step. In particular, this decomposition allows us to define the important concept of coupling order.\\ 
In \cref{sec:decomp_basis_comp} we show that for oracle components with (arbitrary) finite coupling order, we can establish a bijection between $ \{f_i^{\mathbf{k}}\}_{\mathbf{k} \geq\mathbf{0}_{\vert\tset\vert}} $ and a family of simpler functions $ \{\leftidx{^b}{}f_i^{\mathbf{k}}\}_{\mathbf{k} \geq \mathbf{0}_{\vert\tset\vert}} $, with finite support, called basis components. More importantly, the functions in $ \{\leftidx{^b}{}f_i^{\mathbf{k}}\}_{\mathbf{k} \geq \mathbf{0}_{\vert\tset\vert}} $ are all decoupled from one another. From composition of bijections, this implies that every oracle component $ \hat{f}_i $ (with finite coupling order) corresponds to exactly one set of basis components $ \{\leftidx{^b}{}f_i^{\mathbf{k}}\}_{\mathbf{k} \geq \mathbf{0}_{\vert\tset\vert}} $ (with finite support), which have a simple structure and are decoupled. That is, we have successfully exposed the degrees of freedom involved in modeling oracle components (with finite coupling order).\\
Although this decomposition only applies to (arbitrary) finite coupling order, we show that it is a very useful tool for designing oracle components with infinite coupling order. In particular, we can build a valid $ \hat{f}_i $ with infinite coupling order by taking the limit of a sequence of oracle components with finite order. The relationships between the oracle components and the two decompositions is illustrated, with the relevant equations, in the diagram of \cref{fig:decomposition_scheme}.
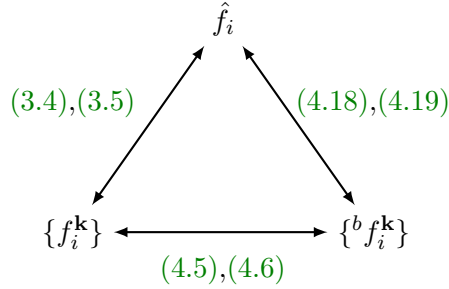
\begin{figure}[h]
	\centering
	\begin{tikzpicture}[
	node1/.style = {},
	edge1/.style = {>=latex,thick},
	]
	\def\scale{2.0}
	\node[node1] at (0,0)(n1){$ \{f_i^{\mathbf{k}}\} $};
	\node[node1] at (2*\scale,0)(n2){$ \{\leftidx{^b}{}f_i^{\mathbf{k}}\} $};
	\node[node1] at (1*\scale,{sqrt(2)*\scale})(n3){$ \hat{f}_i $};

	\draw [<->,edge1](n1) -- (n2);
	\draw [<->,edge1](n1) -- (n3);
	\draw [<->,edge1](n2) -- (n3);
	
	\node (f2fk) at ($(n1)!0.5!(n3) + (-1.0,0.3)$) {\cref{eq:decomp_explicit},\cref{eq:composition_f}};

	\node (fk2bfk) at ($(n1)!0.5!(n2) + (0.0,-0.5)$) {\cref{eq:component_from_basis1_inf},\cref{eq:basis_from_component1_inf}};

	\node (f2fk) at ($(n2)!0.5!(n3) + (+1.0,0.3)$) {\cref{eq:oracle_from_basis1_inf},\cref{eq:basis1_inf_from_oracle}};
	\end{tikzpicture} 
	\caption{Decomposition diagram.}
	\label{fig:decomposition_scheme}
\end{figure}
\section{Weighted multi-edge CCNs} \label{sec:new_formalism}
In this section we briefly introduce the necessary concepts developed in \cite{sequeira2021commutative} on top of which this work is built.
\subsection{Commutative monoids}
The commutative monoid is a set equipped with a binary operation (usually denoted $ + $) such that it is commutative and associative. Furthermore, it has one identity element (usualy denoted $ 0 $). This is the simplest algebraic structure that can be used to describe arbitrary finite parallels of edges. Note that associativity and commutativity, together, are equivalent to the invariance to permutations property.\\
In this work, the ``sum'' operation is denoted by $ \| $, with the meaning of ``adding in parallel''. In this context, the zero element of a monoid should be interpreted as ``no edge''.
Note that we do not require the existence of inverse elements. That is, given an edge, there does not need to exist another one such that the two in parallel act as ``no edge". This is the reason for the use of monoids instead of the algebraic structure of groups.
\subsection{Multi-indexes}
In this work we require the use of multi-index notation. A multi-index is an ordered $ n $-tuple of non-negative integers (indexes). That is, an element of $ \mathbb{N}_0^n $. Two particularly important multi-indexes are $ \mathbf{0}_n $ and $ \mathbf{1}_n $, which represent the tuple of $ n $ zeros and the tuple of $ n $ ones, respectively. Furthermore, we denote by $ 1_j $ the tuple such that its $ j^{th} $ entry is $ 1 $ and all the others are zero.\\
We will denote the multi-indexes with the same notation we use for vectors, using bold, as in $ \mathbf{k} = \left[k_1, \ldots. k_n\right]^{\top}$. Their norm is defined as $ \vert \mathbf{k} \vert := \sum_{i=1}^{n} k_i $.\\
The elements (of the same tupleness $ n $) can be multiplied by non-negative integers and added together freely, although subtraction and division are not always well-defined. For instance,
\begin{align*}
\mathbf{k} = 2 \, \mathbf{1}_3 + 3 \, 1_2 
=
\begin{bmatrix}
2\\
5\\
2
\end{bmatrix} 
\end{align*}
(note the difference between bold and non-bold).\\
The multi-indexes (of the same tupleness $ n $) form a partial order in the straightforward way, that is, $ \mathbf{k^1}\geq \mathbf{k^2} $ if and only if $ k^1_i \geq k^2_i $ for every entry $ 1 \leq i \leq n $. Note that for $ n>1 $ the order is partial since neither $ \mathbf{k^1}\geq \mathbf{k^2} $ nor $ \mathbf{k^1}\leq \mathbf{k^2} $ are required. This happens when there are $ 1 \leq i,j \leq n $ such that $ k^1_i > k^2_i $ and $ k^1_j < k^2_j $. In this case we say that the pair $(\mathbf{k}_1, \mathbf{k}_2 )$ is non-comparable.\\
We will often specify the tupleness $ n $ of a multi-index $ \mathbf{k} $ indirectly, by using $ \mathbf{k}\geq \mathbf{0}_n  $ in order to denote $ \mathbf{k}\in\mathbb{N}_0^n $, or $ \mathbf{k}\geq \mathbf{1}_n  $ to denote $ \mathbf{k}\in\mathbb{N}^n $.
\subsection{CCN formalism} \label{subsec:CCN_formalism}
According to \cite{sequeira2021commutative}, a general weighted coupled cell network is given by the following definition.
\begin{defi}
	A network $ \mathcal{G} $ consists of a set of cells $ \mathcal{C}_{\mathcal{G}} $, where each cell has a type, given by an index set $ \tset=\{1, \ldots, \vert\tset\vert\} $ according to $ \type{G}\colon\ \mathcal{C}_{\mathcal{G}}\to\tset $ and has an \mbox{$ \vert\mathcal{C}_{\mathcal{G}}\vert\times\vert\mathcal{C}_{\mathcal{G}}\vert $} \mbox{in-adjacency} matrix $ M_{\mathcal{G}} $. The entries of $ M_{\mathcal{G}} $ are elements of a family of commutative monoids $ \{\mathcal{M}_{ij}\}_{i,j\in\tset} $ such that $ \left[M_{\mathcal{G}}\right]_{cd} = m_{cd}\in\mathcal{M}_{ij} $, for any cells $ c,d \in \mathcal{C}_{\mathcal{G}} $ with types $ i =\type{G}(c) $, $ j =\type{G}(d) $. 
	\hfill$ \square $
\end{defi}
For each commutative monoid $ \mathcal{M}_{ij} $ we denote its ``zero'' element as $ 0_{ij} $. 
\begin{remark}
	The subscripts $ _\mathcal{G} $ are omitted when the network of interest is clear from context. 
	\hfill$ \square $
\end{remark}
\subsection{Admissibility}\label{sec:Gadmissibility}
A function used to model some first-order property of a network, such as its dynamics, is admissible if it respects the minimal properties that we expect from it. Consider the simple network of \cref{fig:edge_merginga}, (which could be part of a larger network) consisting of cell $ c $ and its in-neighborhood. 
\begin{figure}[h]
	\centering
	\begin{subfigure}[t]{0.23\textwidth}
		\centering
		\begin{tikzpicture}[
node1/.style = {circle,minimum size=23,draw},
node2/.style = {circle,minimum size=23,draw,fill=white!75!black},
node3/.style = {circle,minimum size=23,draw,fill=white!50!black},
noderect/.style = {rectangle,minimum size=20,draw},
edge1/.style = {>=latex,thick},
edgedash/.style = {>=latex,thick,dashed},
edge2/.style = {>=latex,thick,blue},
edge3/.style = {>=latex,thick,red}
]
\node[noderect] at (-0.75,2)(n1){$ x_a $};
\node[noderect] at (+0.75,2)(n2){$ x_b $};
\node[node1] at (0,0)(n3){$ x_c $};

\draw [->,edge1](n1) -- (n3);
\draw [->,edge1](n2) -- (n3);

\node (w13) at ($(n1)!0.5!(n3) + (-0.4,-0.0)$) {$ w_a $};
\node (w23) at ($(n2)!0.5!(n3) + (+0.4,-0.0)$) {$ w_b $};

%\DoubleLine{n1}{n2}{<-,edge1}{}{->,edge1}{}
%\DoubleLine{n1}{n3}{<-,edge1}{}{->,edge1}{}
%\DoubleLine{n2}{n3}{<-,edge1}{}{->,edge1}{}

\end{tikzpicture} 
		\caption{Original.}
		\label{fig:edge_merginga}
	\end{subfigure}
	\begin{subfigure}[t]{0.23\textwidth}
		\centering
		\begin{tikzpicture}[
node1/.style = {circle,minimum size=23,draw},
node2/.style = {circle,minimum size=23,draw,fill=white!75!black},
node3/.style = {circle,minimum size=23,draw,fill=white!50!black},
noderect/.style = {rectangle,minimum size=20,draw},
edge1/.style = {>=latex,thick},
edgedash/.style = {>=latex,thick,dashed},
edge2/.style = {>=latex,thick,blue},
edge3/.style = {>=latex,thick,red}
]

\node[noderect] at (0,2)(n3){$ x_a=x_b $};
\node[node1] at (0,0)(n12){$ x_c $};

\draw [->,edge1](n3) -- (n12);
\node (w13) at ($(n3)!0.5!(n12) + (+0.8,-0.0)$) {$ w_a \| w_b $};

%\DoubleLine{n1}{n2}{<-,edge1}{}{->,edge1}{}
%\DoubleLine{n1}{n3}{<-,edge1}{}{->,edge1}{}
%\DoubleLine{n2}{n3}{<-,edge1}{}{->,edge1}{}

\end{tikzpicture} 
		\caption{Merged.}
		\label{fig:edge_mergingb}
	\end{subfigure}
	\centering
	\caption{Edge merging.}
	\label{fig:edge_merging}
\end{figure}
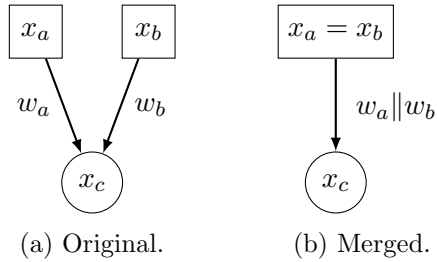
We have cell types $ \tset =\{1,2\} $ which represent ``circle" and ``square" cells, respectively. In order to define functions on the cells we associate with them the state sets $ \xset_{1},\xset_{2} $ and the output sets $ \yset_{1},\yset_{2} $ according to their respective type.\\
We consider that the input received by a cell is independent of how we draw the network, that is, from the point of view of cell $ c $, there would be no difference if cell $ b $ was at the left of cell $ a $. Then, for a function $ \hat{f}_1 $ acting on cells of type $ 1 $, we would expect that
\begin{align*}
\hat{f}_1\left(x_c; \begin{bmatrix}w_a \\ w_b\end{bmatrix},
\begin{bmatrix}x_a \\ x_b\end{bmatrix}\right) 
=
\hat{f}_1\left(x_c; \begin{bmatrix}w_b \\ w_a\end{bmatrix},
\begin{bmatrix}x_b \\ x_a\end{bmatrix}\right), 
\end{align*}
for $ x_c\in\xset_{1} $, $ x_a,x_b\in\xset_{2} $ and $ w_a,w_b\in \mathcal{M}_{12}$. 
Moreover, since cells $ a $ and $ b $ are of the same cell type (square) ($ \type{}(a) = \type{}(b) = 2 $), we expect that when they are in the same state ($ x_a=x_b =x_{ab}$), the total input received by cell $ c $ at that instant, is the same as if both edges originated from a single ``square" cell with that state, as in \cref{fig:edge_mergingb}. That is,
\begin{align*}
\hat{f}_1\left(x_c; \begin{bmatrix}w_a \\ w_b\end{bmatrix},
\begin{bmatrix}x_{ab} \\ x_{ab}\end{bmatrix}\right) 
=
\hat{f}_1\left(x_c;w_a \| w_b, x_{ab} \right). 
\end{align*}
Although this might look inconsistent since the domains look mismatched, the following definition formalizes it in a rigorous way. 
Finally, when $ \hat{f}_1 $ is evaluated at a cell it should only depend on the in-neighborhood of that cell. Therefore, if $ w_a = 0_{12} $, cell $ c $ should not be directly influenced by cell $ a $. That is,
\begin{align*}
\hat{f}_1\left(x_c; \begin{bmatrix}0_{12} \\ w_b\end{bmatrix},
\begin{bmatrix}x_a \\ x_b\end{bmatrix}\right) 
=
\hat{f}_1\left(x_c; w_b,
x_b\right). 
\end{align*}
These ideas are now formalized in the following definition, which is a simpler but equivalent version of the one in \cite{sequeira2021commutative}.
\begin{defi} \label{defi:oracle}
	Consider the set of cell types $ \tset $, and some related sets $ \{\xset_j, \yset_j\}_{j\in\tset} $ together with a family of commutative monoids $ \{\mathcal{M}_{ij}\}_{j\in\tset} $, for a given fixed $ i\in\tset $. Take a function $ \hat{f}_i $ defined on
	\begin{align}
	\hat{f}_i\colon \xset_i\times\bigcup^\circ_{\mathbf{k} 
		\geq \mathbf{0}_{\vert\tset\vert}} \left( \mathcal{M}_i^{\mathbf{k}} \times \xset^{\mathbf{k}} \right)  \to \yset_i,
	\label{eq:oracle_domain_func}
	\end{align}
	where $ \bigcup\limits^\circ $ denotes the disjoint union and for multi-index $ \mathbf{k} $ we define $ \xset^{\mathbf{k}} := \xset_1^{k_1} \times \ldots \times \xset_{\vert\tset\vert}^{k_{\vert\tset\vert}} $ and $ \mathcal{M}_i^{\mathbf{k}} := \mathcal{M}_{i1}^{k_1} \times \ldots \times \mathcal{M}_{i\vert\tset\vert}^{k_{\vert\tset\vert}} $.\\
	The function $ \hat{f}_i $ is called an \textbf{oracle component of type i}, if it has the following properties:
	\begin{enumerate}
		\item \label{thm:oracle_permut}
		If $ \permut $ is a \textbf{permutation} matrix (of appropriate dimension), then
		\begin{align}
		\label{eq:oracle_permut}
		\hat{f}_i(x;\mathbf{w},\mathbf{x})
		=
		\hat{f}_i(x;\permut{\mathbf{w}},\permut{\mathbf{x}}),
		\end{align}
		where we assume, without loss of generality, that one can keep track of the cell types of each element of $ \permut{\mathbf{w}} $ and $ \permut{\mathbf{x}} $.
		\item
		\label{thm:oracle_merge}
		If the indexes $ j_1 $, $j_2 $ and $ j_{12} $ denote cells of type $ j \in\tset$, then
		\begin{align}
		\label{eq:oracle_merge}
		\hat{f}_i
		\left(x; 
		\begin{bmatrix}
		w_{j_1} \| w_{j_2}\\ 
		\mathbf{w}
		\end{bmatrix},
		\begin{bmatrix}
		x_{j_{12}} \\ 
		\mathbf{x}
		\end{bmatrix}
		\right)
		=
		\hat{f}_i
		\left(x; 
		\begin{bmatrix}
		w_{j_1} \\ 
		w_{j_2} \\
		\mathbf{w}
		\end{bmatrix},
		\begin{bmatrix}
		x_{j_{12}} \\ 
		x_{j_{12}} \\ 
		\mathbf{x}
		\end{bmatrix}
		\right).
		\end{align}
		\item
		\label{thm:oracle_0_equal}
		If $ \mathbf{w} $ has its $ k^{th} $ element (corresponding to cell $ c_k $) equal to $ 0_{ij} $, with $ j = \type{}(c_k) $, then 
		\begin{align}
		\hat{f}_i(x;\mathbf{w},\mathbf{x}) = 
		\hat{f}_i(x;\mathbf{w}_{-k},\mathbf{x}_{-k}),
		\label{eq:oracle_0_equal}
		\end{align}
		where $ \mathbf{w}_{-k} $, $ \mathbf{x}_{-k} $ denotes the result of removing the $ k^{th} $ element of the original vectors $ \mathbf{w} $, $ \mathbf{x} $.
	\end{enumerate}
	\hfill$ \square $
\end{defi}
The disjoint union allows us to distinguish neighborhoods of different types, that is, the set $ \xset_1\times\xset_1 $ is always taken as a different set from $ \xset_1\times\xset_2 $ even in the particular case of $ \xset_1 = \xset_2 $.	
\begin{remark}
	As stated in \cref{thm:oracle_permut} of \cref{defi:oracle}, it is always assumed that given any weight $ w_c $ or state $ x_c $, we always know the cell type of the corresponding cell $ c $. Note that one can always do enough bookkeeping in order to ensure this. For instance, one can extend $ \hat{f}_i(x;\mathbf{w},\mathbf{x}) $ into $ \hat{f}_i(x;\mathbf{t},\mathbf{w},\mathbf{x}) $, where $ \mathbf{t} $ would be a vector that encodes the cell types associated with $ \mathbf{w},\mathbf{x} $. Then, we would have $ \hat{f}_i(x;\mathbf{t},\mathbf{w},\mathbf{x})=\hat{f}_i(x;\permut{\mathbf{t}},\permut{\mathbf{w}},\permut{\mathbf{x}}) $ instead.\\
	Our implicit bookkeeping means that we do not have to constrain $ \permut $ to preserve cell typing. That is, if we assume some canonical order of the cell types in the part of the domain $ \mathcal{M}_i^{\mathbf{k}} \times \xset^{\mathbf{k}} $ in \cref{eq:oracle_domain_func}, then we know the correct $ \mathbf{k}	\geq \mathbf{0}_{\vert\tset\vert} $ and can reorder the rows of $ \mathbf{w} $ and $ \mathbf{x} $ in $ \hat{f}_i(x;\mathbf{w},\mathbf{x}) $ appropriately.\\
	Note that by considering invariance under general permutations, and not having to worry about preserving cell types or respecting some canonical ordering of cell types, we are always able to shift the cells of major interest to the top of the vectors, as in \cref{eq:oracle_merge}, regardless of the types of other cells. This is used throughout the paper and it allows us to make our statements and proofs more manageable.
	\hfill$ \square $
\end{remark}
In \cref{defi:oracle} we slightly changed the original notation by interpreting $ \mathbf{w} \in \mathcal{M}_i^{\mathbf{k}} $ as a column vector instead of a row. This is merely cosmetic but it makes this work more clear.\\
The \textbf{oracle set} is the set of all $ \vert \tset\vert $-tuples of oracle components, such that each element of the tuple represents one of the types in $ \tset $. It is denoted as
\begin{align*}
\hat{\mathcal{F}}_{\tset} = \prod_{i\in \tset}\hat{\mathcal{F}}_{i} ,
\end{align*}
where $ \hat{\mathcal{F}}_{i} $ is the set of all oracle components of type $ i $. We are always implicitly assuming sets $ \{\xset_i, \yset_i\}_{i\in\tset} $ and commutative monoids $ \{\mathcal{M}_{ij}\}_{i,j\in\tset} $.
Note that modeling some aspect of a network that follows our assumptions is effectively choosing one of the elements of $ \hat{\mathcal{F}}_{\tset} $, which we call \textbf{oracle functions}. In this work we will make use of the following topological result.
\begin{lemma}
	Consider $ \hat{\mathcal{F}}_{i} $ and $ \hat{\mathcal{F}}_{\tset} $ such that the related sets $ \{\yset_j\}_{j\in\tset} $ are Hausdorff spaces. Then, $ \hat{\mathcal{F}}_{i} $ and $ \hat{\mathcal{F}}_{\tset} $ are sequentially closed in the topology of pointwise convergence (product topology).
	\label{lemma:F_closed}
	\hfill$ \square $
\end{lemma}
\begin{proof}
	Consider a sequence of functions $ (\leftidx{^N}{}\hat{f}_i)_{N\in\mathbb{N}} $, with $ \leftidx{^N}{}\hat{f}_i \in \hat{\mathcal{F}}_{i} $ for all $ N\in\mathbb{N} $, such that it converges pointwise to some function $ \hat{f}_i $. That is,
	\begin{align*}
	\lim\limits_{N\to\infty}
	\leftidx{^N}{}\hat{f}_i(x;\mathbf{w},\mathbf{x})
	= 
	\hat{f}_i(x;\mathbf{w},\mathbf{x}),
	\end{align*}
	for all $ x\in\xset_i $, $ \mathbf{x} \in \xset^{\mathbf{k}} $, $ \mathbf{w} \in \mathcal{M}_i^{\mathbf{k}} $, for any given $ \mathbf{k} 
	\geq \mathbf{0}_{\vert\tset\vert} $. Given a permutation matrix $ \permut $ of appropriate dimension, then
	\begin{align*}
	\lim\limits_{N\to\infty}
	\leftidx{^N}{}\hat{f}_i(x;\permut\mathbf{w},\permut\mathbf{x})
	= 
	\hat{f}_i(x;\permut\mathbf{w},\permut\mathbf{x}).
	\end{align*}
	Note that from assumption, \cref{eq:oracle_permut} is satisfied for every $ \leftidx{^N}{}\hat{f}_i $. Therefore, these two sequences are the same. Since $ \yset_i $ is Hausdorff, we know that the limit of a convergent sequence is unique, which implies $ \hat{f}_i(x;\mathbf{w},\mathbf{x})	=	\hat{f}_i(x;\permut{\mathbf{w}},\permut{\mathbf{x}}) $. That is, $ \hat{f}_i $ also satisfies \cref{eq:oracle_permut}. The same reasoning applies with respect to
	\cref{eq:oracle_merge,eq:oracle_0_equal}.\\
	Therefore, $ \hat{f}_i \in \hat{\mathcal{F}}_{i} $, which means that $ \hat{\mathcal{F}}_{i} $ is sequentially closed. Since the product of sequentially closed sets is sequentially closed, $ \hat{\mathcal{F}}_{\tset} $ is also sequentially closed.
\end{proof}
\section{Decomposition into coupling components} \label{sec:decomp_coupl_comp}
We are now ready to develop our first decomposition scheme. In this section, we present a decomposition scheme for oracle components in which the output sets $ \{\yset_i\}_{i\in\tset} $ are vector spaces. We start by illustrating the main ideas with an example.
\begin{exmp}
	Consider cell types $ \tset=\{1,2\} $ which denote the cell types ``circle" and ``square" respectively. \Cref{fig:decomp_exmp} presents a cell of type $ 1 $ with different types of inputs sets, denoted by the multi-indexes $ \left[0 0\right]$, $ \left[0 1\right]$ and $ \left[0 2\right]$ respectively. 
	\begin{figure}[h]
		\centering
		\begin{subfigure}[t]{0.30\textwidth}
			\centering
			\begin{tikzpicture}[
node1/.style = {circle,minimum size=23,draw},
node2/.style = {circle,minimum size=23,draw,fill=white!75!black},
node3/.style = {circle,minimum size=23,draw,fill=white!50!black},
noderect/.style = {rectangle,minimum size=20,draw},
edge1/.style = {>=latex,thick},
edgedash/.style = {>=latex,thick,dashed},
edge2/.style = {>=latex,thick,blue},
edge3/.style = {>=latex,thick,red}
]

\node[node1] at (0,0)(n3){$ x $};

%\DoubleLine{n1}{n2}{<-,edge1}{}{->,edge1}{}
%\DoubleLine{n1}{n3}{<-,edge1}{}{->,edge1}{}
%\DoubleLine{n2}{n3}{<-,edge1}{}{->,edge1}{}

\end{tikzpicture} 
			\caption{No in-neighbors.}
			\label{fig:decomp_exmpa}
		\end{subfigure}
		\begin{subfigure}[t]{0.30\textwidth}
			\centering
			\begin{tikzpicture}[
node1/.style = {circle,minimum size=23,draw},
node2/.style = {circle,minimum size=23,draw,fill=white!75!black},
node3/.style = {circle,minimum size=23,draw,fill=white!50!black},
noderect/.style = {rectangle,minimum size=20,draw},
edge1/.style = {>=latex,thick},
edgedash/.style = {>=latex,thick,dashed},
edge2/.style = {>=latex,thick,blue},
edge3/.style = {>=latex,thick,red}
]
\node[noderect] at (0.0,2)(n1){$ x_a $};
\node[node1] at (0,0)(n3){$ x $};

\draw [->,edge1](n1) -- (n3);

\node (w13) at ($(n1)!0.5!(n3) + (-0.4,-0.0)$) {$ w_a $};

%\DoubleLine{n1}{n2}{<-,edge1}{}{->,edge1}{}
%\DoubleLine{n1}{n3}{<-,edge1}{}{->,edge1}{}
%\DoubleLine{n2}{n3}{<-,edge1}{}{->,edge1}{}

\end{tikzpicture} 
			\caption{One in-neighbor.}
			\label{fig:decomp_exmpb}
		\end{subfigure}
		\begin{subfigure}[t]{0.30\textwidth}
			\centering
			\begin{tikzpicture}[
node1/.style = {circle,minimum size=23,draw},
node2/.style = {circle,minimum size=23,draw,fill=white!75!black},
node3/.style = {circle,minimum size=23,draw,fill=white!50!black},
noderect/.style = {rectangle,minimum size=20,draw},
edge1/.style = {>=latex,thick},
edgedash/.style = {>=latex,thick,dashed},
edge2/.style = {>=latex,thick,blue},
edge3/.style = {>=latex,thick,red}
]
\node[noderect] at (-0.75,2)(n1){$ x_a $};
\node[noderect] at (+0.75,2)(n2){$ x_b $};
\node[node1] at (0,0)(n3){$ x $};

\draw [->,edge1](n1) -- (n3);
\draw [->,edge1](n2) -- (n3);

\node (w13) at ($(n1)!0.5!(n3) + (-0.4,-0.0)$) {$ w_a $};
\node (w23) at ($(n2)!0.5!(n3) + (+0.4,-0.0)$) {$ w_b $};

%\DoubleLine{n1}{n2}{<-,edge1}{}{->,edge1}{}
%\DoubleLine{n1}{n3}{<-,edge1}{}{->,edge1}{}
%\DoubleLine{n2}{n3}{<-,edge1}{}{->,edge1}{}

\end{tikzpicture} 
			\caption{Two in-neighbors.}
			\label{fig:decomp_exmpc}
		\end{subfigure}
		\centering
		\caption{Simple input sets.}
		\label{fig:decomp_exmp}
	\end{figure}
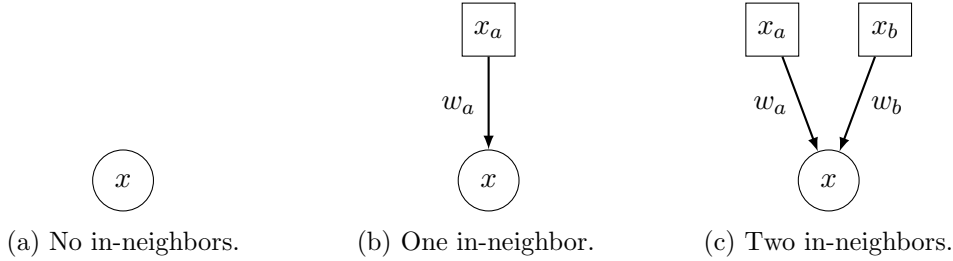
	We assume a particular oracle component $ \hat{f}_1 \in \hat{\mathcal{F}}_{1} $ has been chosen.
	Consider the input set in \cref{fig:decomp_exmpa}. This cell does not depend on anything else in the network, it evolves only according to its own internal dynamics. We define the function $ f_1^{\left[ 0 0 \right]}\colon
	\xset_1\to \yset_1 $ as
	\begin{align*}
	f_1^{\left[00\right]}(x) := \hat{f}_1(x).
	\end{align*}
	We use this to rewrite the function evaluation of the input set in \cref{fig:decomp_exmpb} as
	\begin{align*}
	\hat{f}_1(x;w_a,x_a) 
	= 
	f_1^{\left[00\right]}(x) 
	+ 
	f_1^{\left[01\right]}(x;w_a,x_a),
	\end{align*}
	where $ f_1^{\left[01\right]}
	\colon
	\xset_1\times 
	\mathcal{M}_{12}
	\times 
	\xset_{2}
	\to \yset_1 $ is defined as
	\begin{align*}
	f_1^{\left[01\right]}(x;w_a,x_a) 
	:= 
	\hat{f}_1(x;w_a,x_a) 
	-
	f_1^{\left[00\right]}(x).
	\end{align*}
	That is, we decompose the evaluation of the oracle component $ \hat{f}_1 $ into the internal dynamics of the cell $ \left(f_1^{\left[00\right]}\right) $ and the influence from its single in-neighbor of cell type $ 2 $ $ \left(f_1^{\left[01\right]}\right) $. Note that if the weight value is $ 0_{12}$, this case reduces to the one in \cref{fig:decomp_exmpa}, which implies that $ f_1^{\left[01\right]}(x;0_{12},x_a) = 0_{\yset_1} $.\\
	Consider now the input set in \cref{fig:decomp_exmpc}. We can write its evaluation of the oracle component as
	\begin{align*}
	\hat{f}_1\left(x; \begin{bmatrix}w_a \\ w_b\end{bmatrix},
	\begin{bmatrix}x_a \\ x_b\end{bmatrix}\right) 
	=
	f_1^{\left[00\right]}(x)
	+ 
	f_1^{\left[01\right]}(x;w_a,x_a) 
	+ 
	f_1^{\left[01\right]}(x;w_b,x_b) 
	+ 
	f_1^{\left[02\right]}\left(x; 
	\begin{bmatrix}w_a \\ w_b\end{bmatrix},
	\begin{bmatrix}x_a \\ x_b\end{bmatrix}
	\right), 
	\end{align*}
	where $ f_1^{\left[02\right]}\colon
	\xset_1\times
	\mathcal{M}_{12}^2 \times \xset_{2}^2
	\to \yset_1 $ is defined as
	\begin{align*}
	f_1^{\left[02\right]}
	\left(x; 
	\begin{bmatrix}w_a \\ w_b\end{bmatrix},
	\begin{bmatrix}x_a \\ x_b\end{bmatrix}
	\right) 
	&:=
	\hat{f}_1
	\left(x; 
	\begin{bmatrix}w_a \\ w_b
	\end{bmatrix},
	\begin{bmatrix}x_a \\ x_b
	\end{bmatrix}
	\right) 
	-
	f_1^{\left[00\right]}(x)
	-
	f_1^{\left[01\right]}(x;w_a,x_a) 
	-
	f_1^{\left[01\right]}(x;w_b,x_b).
	\end{align*}
	The term $ f_1^{\left[02\right]} $ describes a $ 2 $-order coupling effect of cells of type ``square" onto cells of type ``circle". By definition, it corresponds to what cannot be explained by the internal dynamics $ \left(f_1^{\left[00\right]}\right) $ ($ 0 $-order coupling) and the $ 1 $-order coupling contributions from each ``square" in-neighbor $ \left(f_1^{\left[01\right]}\right) $.
	Note that if \textbf{any} of its weight parameters $ w_a,w_b $ is $ 0_{12} $, this reduces to the previous case and similarly we conclude that $ f_1^{\left[02\right]}\left(x; 
	\begin{bmatrix}w_a \\ w_b\end{bmatrix},
	\begin{bmatrix}x_a \\ x_b\end{bmatrix}
	\right) = 0_{\yset_1} $. Moreover, note that
	\begin{align*}
	f_1^{\left[02\right]}
	\left(x; 
	\begin{bmatrix}w_a \\ w_b\end{bmatrix},
	\begin{bmatrix}x_a \\ x_b\end{bmatrix}
	\right)
	=
	f_1^{\left[02\right]}
	\left(x; 
	\begin{bmatrix}w_b \\ w_a\end{bmatrix},
	\begin{bmatrix}x_b \\ x_a\end{bmatrix}
	\right). 
	\end{align*}
	Consider now the case where $ x_a = x_b = x_{ab}$. This is equivalent to having an edge weight of $ w_a \| w_b $ in \cref{fig:decomp_exmpb}. This implies
	\begin{align}
	f_1^{\left[01\right]}(x;w_a \| w_b,x_{ab}) 
	= 
	f_1^{\left[01\right]}(x;w_a,x_{ab})
	+ 
	f_1^{\left[01\right]}(x;w_b,x_{ab})
	+ 
	f_1^{\left[02\right]}
	\left(x; 
	\begin{bmatrix}w_a \\ w_b\end{bmatrix},
	\begin{bmatrix}x_{ab} \\ x_{ab}\end{bmatrix}
	\right)
	\label{eq:f1_f2_dependence},
	\end{align}
	which means that $ f_1^{\left[01\right]} $ and $ f_1^{\left[02\right]} $ are related to one another and cannot be chosen independently.
	\hfill$ \square $
\end{exmp}
The following definition is the generalization of this approach to arbitrary finite cell types and in-neighborhoods.
\begin{defi}
	Consider the set of cell types $ \tset $ and the related sets $ \{\xset_i, \yset_i\}_{i\in\tset} $ where $ \{\yset_i\}_{i\in\tset} $ are vector spaces. Given an oracle component $ \hat{f}_i \in \hat{\mathcal{F}}_{i} $, $ i\in\tset $, we define the family of \textbf{coupling components} $ \{f_i^{\mathbf{k}}\}_{\mathbf{k} \geq\mathbf{0}_{\vert\tset\vert}} $, with
	\begin{align}
	f_i^{\mathbf{k}}\colon
	\xset_i\times
	\mathcal{M}_i^{\mathbf{k}} \times \xset^{\mathbf{k}}
	\to \yset_i, 
	\end{align}
	defined recursively by
\begin{equation}
f_i^{\mathcal{K}(\mathbf{s})}\left(x;\mathbf{w}_{\mathbf{s}},\mathbf{x}_{\mathbf{s}}\right)
:=
\hat{f}_i\left(x; \mathbf{w}_{\mathbf{s}},\mathbf{x}_{\mathbf{s}}\right) 
-
\sum_{\overline{\mathbf{s}}\subset\mathbf{s}} f_i^{\mathcal{K}(\overline{\mathbf{s}})}\left(x;\mathbf{w}_{\overline{\mathbf{s}}},\mathbf{x}_{\overline{\mathbf{s}}}\right),
\label{eq:decomp_defi_recursive}
\end{equation}
where $ \mathbf{k} =  \mathcal{K}(\mathbf{s}) $ gives the corresponding multi-index of the types of cells $ \mathbf{s} $ and $ x\in\xset_{i} $, $ \mathbf{x}_{\mathbf{s}} \in \xset^{\mathbf{k}} $, $ \mathbf{w}_{\mathbf{s}} \in 
\mathcal{M}_i^{\mathbf{k}} $.
\hfill$ \square $
\label{defi:func_decomp}
\end{defi}
The following result expands the recursive formula in \cref{eq:decomp_defi_recursive} and writes $ \{f_i^{\mathbf{k}}\}_{\mathbf{k} \geq\mathbf{0}_{\vert\tset\vert}} $ explicitly in terms of $ \hat{f}_i $.
\begin{lemma}\label{lemma:decomp_explicit}
The coupling components $ \{f_i^{\mathbf{k}}\}_{\mathbf{k} \geq\mathbf{0}_{\vert\tset\vert}} $ of an oracle component $ \hat{f}_i \in \hat{\mathcal{F}}_{i} $, $ i\in\tset $, are given by
\begin{align}
f_i^{\mathcal{K}(\mathbf{s})}\left(x;\mathbf{w}_{\mathbf{s}},\mathbf{x}_{\mathbf{s}}\right)
=
\sum_{\overline{\mathbf{s}}\subseteq\mathbf{s}}
(-1)^{\vert\mathbf{s}\vert-\vert\overline{\mathbf{s}}\vert}
\hat{f}_i
\left(x;\mathbf{w}_{\overline{\mathbf{s}}},\mathbf{x}_{\overline{\mathbf{s}}}\right).
\label{eq:decomp_explicit}
\end{align}
\hfill$ \square $
\end{lemma}
\begin{proof}	
The proof is by strong induction. Assume the statement to be true for all $ \overline{\mathbf{s}}\subset\mathbf{s} $. Then, by assumption we can plug the explicit formula \cref{eq:decomp_explicit} into the recursive definition \cref{eq:decomp_defi_recursive} in order to obtain
\begin{align*}
f_i^{\mathcal{K}(\mathbf{s})}\left(x;\mathbf{w}_{\mathbf{s}},\mathbf{x}_{\mathbf{s}}\right)
&=
\hat{f}_i\left(x; \mathbf{w}_{\mathbf{s}},\mathbf{x}_{\mathbf{s}}\right) 
-
\sum_{\overline{\mathbf{s}}\subset\mathbf{s}} f_i^{\mathcal{K}(\overline{\mathbf{s}})}\left(x;\mathbf{w}_{\overline{\mathbf{s}}},\mathbf{x}_{\overline{\mathbf{s}}}\right)
\\
&=
\hat{f}_i\left(x; \mathbf{w}_{\mathbf{s}},\mathbf{x}_{\mathbf{s}}\right) 
-
\sum_{\overline{\mathbf{s}}\subset\mathbf{s}}
\sum_{\overline{\mathbf{r}}\subseteq\overline{\mathbf{s}}}
(-1)^{\vert\overline{\mathbf{s}}\vert-\vert\overline{\mathbf{r}}\vert}
\hat{f}_i
\left(x;\mathbf{w}_{\overline{\mathbf{r}}},\mathbf{x}_{\overline{\mathbf{r}}}\right).
\end{align*}
We reorder this such that the outer sum is indexed over $ \overline{\mathbf{r}} $, which yields
\begin{align*}
\hat{f}_i\left(x; \mathbf{w}_{\mathbf{s}},\mathbf{x}_{\mathbf{s}}\right) 
-
\sum_{\overline{\mathbf{r}}\subset\mathbf{s}}
\left[
\sum_{\substack{
	\overline{\mathbf{s}} \subset \mathbf{s}\\
	\overline{\mathbf{s}} \supseteq \overline{\mathbf{r}}
}}
(-1)^{\vert\overline{\mathbf{s}}\vert- \vert\overline{\mathbf{r}}\vert}
\right]
\hat{f}_i
\left(x;\mathbf{w}_{\overline{\mathbf{r}}},\mathbf{x}_{\overline{\mathbf{r}}}\right).
\end{align*}
Note that
\begin{align*}
\sum_{\substack{
		\overline{\mathbf{s}} \subset \mathbf{s}\\
		\overline{\mathbf{s}} \supseteq \overline{\mathbf{r}}
}}
(-1)^{\vert\overline{\mathbf{s}}\vert- \vert\overline{\mathbf{r}}\vert}
=
\sum_{
		(\overline{\mathbf{s}}\setminus \overline{\mathbf{r}}) \subset (\mathbf{s} \setminus\overline{\mathbf{r}})
}
(-1)^{\vert\overline{\mathbf{s}}\setminus\overline{\mathbf{r}}\vert}
=
\sum_{
	(\overline{\mathbf{s}}\setminus \overline{\mathbf{r}}) \subseteq (\mathbf{s} \setminus\overline{\mathbf{r}})
}
(-1)^{\vert\overline{\mathbf{s}}\setminus\overline{\mathbf{r}}\vert}
-(-1)^{\vert\mathbf{s}\setminus\overline{\mathbf{r}}\vert} .
\end{align*}
In the power set of a non-empty finite set, half of the subsets have an even size and the other half has odd size. Therefore, if $ \overline{\mathbf{r}}\subset\mathbf{s} $, we are in this situation and the sum cancels, and we get
\begin{align*}
f_i^{\mathcal{K}(\mathbf{s})}\left(x;\mathbf{w}_{\mathbf{s}},\mathbf{x}_{\mathbf{s}}\right)
&=
\hat{f}_i\left(x; \mathbf{w}_{\mathbf{s}},\mathbf{x}_{\mathbf{s}}\right) 
-
\sum_{\overline{\mathbf{r}}\subset\mathbf{s}}
\left[
-(-1)^{\vert\mathbf{s}\setminus\overline{\mathbf{r}}\vert}
\right]
\hat{f}_i
\left(x;\mathbf{w}_{\overline{\mathbf{r}}},\mathbf{x}_{\overline{\mathbf{r}}}\right)\\
&=
\hat{f}_i\left(x; \mathbf{w}_{\mathbf{s}},\mathbf{x}_{\mathbf{s}}\right) 
+
\sum_{\overline{\mathbf{r}}\subset\mathbf{s}}
(-1)^{\vert\mathbf{s} \vert - \vert \overline{\mathbf{r}}\vert}
\hat{f}_i
\left(x;\mathbf{w}_{\overline{\mathbf{r}}},\mathbf{x}_{\overline{\mathbf{r}}}\right)\\
&=
\sum_{\overline{\mathbf{r}}\subseteq\mathbf{s}}
(-1)^{\vert\mathbf{s} \vert - \vert \overline{\mathbf{r}}\vert}
\hat{f}_i
\left(x;\mathbf{w}_{\overline{\mathbf{r}}},\mathbf{x}_{\overline{\mathbf{r}}}\right),
\end{align*}
which proves the result for $ \mathbf{s} $. Note that the strong induction immediately satisfies the case $ \mathbf{s} = \emptyset $ since its hypothesis is vacuously true.
\end{proof}
Similarly, we can also write $ \hat{f}_i $ explicitly in terms of $ \{f_i^{\mathbf{k}}\}_{\mathbf{k} \geq\mathbf{0}_{\vert\tset\vert}} $.
\begin{lemma}
An oracle component $ \hat{f}_i \in \hat{\mathcal{F}}_{i} $, $ i\in\tset $ is given by its coupling components $ \{f_i^{\mathbf{k}}\}_{\mathbf{k} \geq\mathbf{0}_{\vert\tset\vert}} $, according to
	\begin{equation}
	\hat{f}_i\left(x; \mathbf{w}_{\mathbf{s}},\mathbf{x}_{\mathbf{s}}\right) 
	=
	\sum_{\overline{\mathbf{s}}\subseteq\mathbf{s}} f_i^{\mathcal{K}(\overline{\mathbf{s}})}
	\left(x;
	\mathbf{w}_{\overline{\mathbf{s}}},
	\mathbf{x}_{\overline{\mathbf{s}}}
	\right).
	\label{eq:composition_f}
	\end{equation}
\hfill$ \square $
\end{lemma}
\begin{proof}
This is immediate from  \cref{eq:decomp_defi_recursive} by simple rearrangement.
\end{proof}
Note that \cref{eq:composition_f} can also be written as
\begin{equation}
\hat{f}_i\left(x; \mathbf{w}_{\mathbf{s}},\mathbf{x}_{\mathbf{s}}\right) 
=
\sum_{\substack{\overline{\mathbf{k}} \leq \mathbf{k}
		\\
		\mathcal{K}(\mathbf{s})=\mathbf{k}
}}
\sum_{\substack{\overline{\mathbf{s}}\subseteq\mathbf{s}
		\\ 
		\mathcal{K}(\overline{\mathbf{s}})=\overline{\mathbf{k}}
}}
f_i^{\overline{\mathbf{k}}}\left(x;\mathbf{w}_{\overline{\mathbf{s}}},\mathbf{x}_{\overline{\mathbf{s}}}\right).
\label{eq:decomp_order_sum}
\end{equation}
\begin{remark}
	The number of multi-indexes smaller or equal to $ \mathbf{k} $ is $ \prod_{i\in\tset}(k_i+1) $ and for each particular $ \overline{\mathbf{k}} $ the number of terms in the sum is $ \prod_{i\in\tset}\binom{k_i}{\overline{k}_i} $.
	\hfill$ \square $	
\end{remark}
\begin{remark}
	These functions operate on an arbitrary (but finite) set of cells $ \mathbf{s} $. Even though there is no upper bound for the amount of terms in the sums, for any particular chosen $ \mathbf{s} $ the sum is always finite. Therefore, everything is well-defined and there are no convergence issues.
	\hfill$ \square $	
\end{remark}
This is exactly the anchored decomposition~\cite{kuo2010decompositions} applied to an arbitrary finite set of variables. The decomposition is done with respect to the weights of $ \mathbf{w}_{\mathbf{s}} $, anchoring them at $ 0_{ij} $, for the appropriate $ j\in\tset $. From the properties of the anchored decomposition we know immediately that if any of the entries of $ \mathbf{w} $ is $ 0_{ij} $, then
$f_i^{\mathbf{k}}\left(x;\mathbf{w},\mathbf{x}\right)=0_{\yset_i} $. From \cref{thm:oracle_0_equal} of \cref{defi:oracle}, we note that when we anchor some entry of $ \mathbf{w}_{\mathbf{s}} $ to $ 0_{ij} $ we are also removing the functional dependence on the corresponding entry of $ \mathbf{x}_{\mathbf{s}} $.\\
Moreover, note that for subsets of cells $ \mathbf{s}_1, \mathbf{s}_2 \subset \mathbf{s} $ such that $ \mathcal{K}(\mathbf{s}_1) = \mathcal{K}(\mathbf{s}_2) = \mathbf{k}$, we indexed their associated function by $ \mathbf{k} $ instead of by $  \mathbf{s}_1 $ and $ \mathbf{s}_2 $ as is traditional in the anchored decomposition. This is proper since the functions $ \{f_i^{\mathbf{k}}\}_{\mathbf{k} \geq \mathbf{0}_{\vert\tset\vert}} $ inherit from $ \hat{f}_i $ the property of being invariant to permutations.\\
In summary, the decomposition according to \cref{defi:func_decomp} gives us a family of functions $ \{f_i^{\mathbf{k}}\}_{\mathbf{k} \geq \mathbf{0}_{\vert\tset\vert}} $, which is an equivalent representation of a given oracle component function $ \hat{f}_i $.\\
The following result presents the necessary and sufficient conditions for $ \{f_i^{\mathbf{k}}\}_{\mathbf{k} \geq \mathbf{0}_{\vert\tset\vert}} $ to be such that it corresponds to a valid $ \hat{f}_i $. That is, for the corresponding $ \hat{f}_i $ to follow \cref{defi:oracle}.
\begin{theorem}
	\label{thm:decomp_f}
	The family of functions $ \{f_i^{\mathbf{k}}\}_{\mathbf{k} \geq\mathbf{0}_{\vert\tset\vert}} $, represents some valid oracle component $ \hat{f}_i \in \hat{\mathcal{F}}_{i} $, and is related to it according to \cref{defi:func_decomp}, if and only if for every $ \mathbf{k} \geq \mathbf{0}_{\vert\tset\vert} $, $ f_i^{\mathbf{k}} $ has the following properties:
	\begin{enumerate}
		\item \label{thm:decomp_f_permutation_invariance}If $ \permut $ is any permutation matrix (of appropriate dimension), then
		\begin{align}
		f_i^{\mathbf{k}}\left(x;\mathbf{w},\mathbf{x}\right)
		=
		f_i^{\mathbf{k}}\left(x;\permut\mathbf{w},\permut\mathbf{x}\right).
		\label{eq:decomp_perm_inv}
		\end{align}
		\item \label{thm:decomp_f_dependence} If $ k_j > 0 $, then $ f_i^{\mathbf{k}} $ and $ f_i^{\mathbf{k}+1_{j}} $ are related by
		\begin{align}
		f_i^{\mathbf{k}}
		\left(x; 
		\begin{bmatrix}
		w_{j_1} \| w_{j_2}\\ 
		\mathbf{w}_{\mathbf{s}}
		\end{bmatrix},
		\begin{bmatrix}
		x_{j_{12}} \\ 
		\mathbf{x}_{\mathbf{s}}
		\end{bmatrix}
		\right)
		&=
		f_i^{\mathbf{k}}
		\left(x; 
		\begin{bmatrix}
		w_{j_1} \\ 
		\mathbf{w}_{\mathbf{s}}
		\end{bmatrix},
		\begin{bmatrix}
		x_{j_{12}} \\ 
		\mathbf{x}_{\mathbf{s}}
		\end{bmatrix}
		\right)
		+
		f_i^{\mathbf{k}}
		\left(x; 
		\begin{bmatrix}
		w_{j_2}\\ 
		\mathbf{w}_{\mathbf{s}}
		\end{bmatrix},
		\begin{bmatrix}
		x_{j_{12}} \\ 
		\mathbf{x}_{\mathbf{s}}
		\end{bmatrix}
		\right)
		\\
		&\quad+
		f_i^{\mathbf{k}+1_{j}}
		\left(x; 
		\begin{bmatrix}
		w_{j_1} \\ 
		w_{j_2} \\
		\mathbf{w}_{\mathbf{s}}
		\end{bmatrix},
		\begin{bmatrix}
		x_{j_{12}} \\ 
		x_{j_{12}} \\ 
		\mathbf{x}_{\mathbf{s}}
		\end{bmatrix}
		\right),
		\notag
		\end{align}
		where $ \mathbf{s} $ is a set of cells such that $ \mathcal{K}(\mathbf{s}) = \mathbf{k} - 1_j$, and the indexes $ j_1 $, $j_2 $ and $ j_{12} $ denote cells of type $ j $.
		\item \label{thm:decomp_f_zero_kill}If \textbf{any} of the entries of $ \mathbf{w} $ is $ 0_{ij} $ for some $ j\in\tset $, then
		\begin{align}
		f_i^{\mathbf{k}}\left(x;\mathbf{w},\mathbf{x}\right)
		=
		0_{\yset_i}.
		\end{align}
	\end{enumerate}
	\hfill$ \square $
\end{theorem}
\begin{proof}
	We begin by proving the $ \implies $ direction. That is, for a given $ \hat{f}_i $, the family of functions
	$ \{f_i^{\mathbf{k}}\}_{\mathbf{k} \geq \mathbf{0}_{\vert\tset\vert}} $ will have the properties in \cref{thm:decomp_f_permutation_invariance,thm:decomp_f_zero_kill,thm:decomp_f_dependence}.\\
	\Cref{thm:decomp_f_permutation_invariance} is immediate from \cref{eq:decomp_explicit}, which writes $ f_i^{\mathcal{K}(\mathbf{s})} $ explicitly as a function of $ \hat{f}_i $, together with $ \cref{eq:oracle_permut} $. Note that the sum \cref{eq:decomp_explicit} being indexed over all subsets $ \overline{\mathbf{s}}\subseteq \mathbf{s} $ is crucial to keep the whole sum invariant under permutations.  
	This means that unlike what is traditional in the general anchored decomposition, we do not require to index the coupling components according to cells subsets (e.g., $ f_i^{\mathbf{s}_1} $, $ f_i^{\mathbf{s}_2} $) since they are functionally the same whenever $ \mathcal{K}(\mathbf{s}_1) = \mathcal{K}(\mathbf{s}_2) $. Instead we can freely index them according to their respective type multi-index.
	That is, our definition is self-consistent.\\
	The proof of \cref{thm:decomp_f_dependence} is by strong induction. Assume the statement to be true for all $ \overline{\mathbf{s}}\subset\mathbf{s} $. We apply \cref{eq:composition_f} to \cref{eq:oracle_merge}. Then, the left hand side becomes
	\begin{align*}
	\sum_{
		\overline{\mathbf{s}}
		\subseteq
		\mathbf{s}
	}
	\left(
	f_i^{\mathcal{K}(\overline{\mathbf{s}})}\left(x;
	\mathbf{w}_{\overline{\mathbf{s}}},
	\mathbf{x}_{\overline{\mathbf{s}}}\right)
	+
	f_i^{\mathcal{K}(\overline{\mathbf{s}})+1_j}
	\left(x;
	\begin{bmatrix}
	w_{j_1} \| w_{j_2}\\
	\mathbf{w}_{\overline{\mathbf{s}}}
	\end{bmatrix}
	,
	\begin{bmatrix}
	x_{j_{12}}\\
	\mathbf{x}_{\overline{\mathbf{s}}}
	\end{bmatrix}
	\right)
	\right),
	\end{align*}
	and the right hand expands into
	\begin{align*}
	\sum_{
		\overline{\mathbf{s}}
		\subseteq
		\mathbf{s} 
	}
	&\left[
	f_i^{\mathcal{K}(\overline{\mathbf{s}})}\left(x;
	\mathbf{w}_{\overline{\mathbf{s}}},
	\mathbf{x}_{\overline{\mathbf{s}}}\right)
	+
	f_i^{\mathcal{K}(\overline{\mathbf{s}})+1_j}
	\left(x;
	\begin{bmatrix}
	w_{j_1}\\
	\mathbf{w}_{\overline{\mathbf{s}}}
	\end{bmatrix}
	,
	\begin{bmatrix}
	x_{j_{12}}\\
	\mathbf{x}_{\overline{\mathbf{s}}}
	\end{bmatrix}
	\right)
	+
	f_i^{\mathcal{K}(\overline{\mathbf{s}})+1_j}
	\left(x;
	\begin{bmatrix}
	w_{j_2}\\
	\mathbf{w}_{\overline{\mathbf{s}}}
	\end{bmatrix}
	,
	\begin{bmatrix}
	x_{j_{12}}\\
	\mathbf{x}_{\overline{\mathbf{s}}}
	\end{bmatrix}
	\right)
	\right.
	\\
	&\quad+\left.
	f_i^{\mathcal{K}(\overline{\mathbf{s}})+2_j}
	\left(x;
	\begin{bmatrix}
	w_{j_1}\\
	w_{j_2}\\
	\mathbf{w}_{\overline{\mathbf{s}}}
	\end{bmatrix}
	,
	\begin{bmatrix}
	x_{j_{12}}\\
	x_{j_{12}}\\
	\mathbf{x}_{\overline{\mathbf{s}}}
	\end{bmatrix}
	\right)
	\right].
	\end{align*}
	The first terms of the sum in both sides cancel each other. Using the assumption that \cref{thm:decomp_f_dependence} holds for every index $ \overline{\mathbf{s}}\subset\mathbf{s}$, the last term of the left hand side cancels with the last three terms of the right hand side. Thus, what remains is those terms indexed with $ \overline{\mathbf{s}}=\mathbf{s} $, that is,
	\begin{align*}
	f_i^{\mathcal{K}(\mathbf{s})+1_j}
	\left(x;
	\begin{bmatrix}
	w_{j_1} \| w_{j_2}\\
	\mathbf{w}_{\mathbf{s}}
	\end{bmatrix}
	,
	\begin{bmatrix}
	x_{j_{12}}\\
	\mathbf{x}_{\mathbf{s}}
	\end{bmatrix}
	\right)
	&=
	f_i^{\mathcal{K}(\mathbf{s})+1_j}
	\left(x;
	\begin{bmatrix}
	w_{j_1}\\
	\mathbf{w}_{\mathbf{s}}
	\end{bmatrix}
	,
	\begin{bmatrix}
	x_{j_{12}}\\
	\mathbf{x}_{\mathbf{s}}
	\end{bmatrix}
	\right)
	+
	f_i^{\mathcal{K}(\mathbf{s})+1_j}
	\left(x;
	\begin{bmatrix}
	w_{j_2}\\
	\mathbf{w}_{\mathbf{s}}
	\end{bmatrix}
	,
	\begin{bmatrix}
	x_{j_{12}}\\
	\mathbf{x}_{\mathbf{s}}
	\end{bmatrix}
	\right)
	\\
	&\quad+
	f_i^{\mathcal{K}(\mathbf{s})+2_j}
	\left(x;
	\begin{bmatrix}
	w_{j_1}\\
	w_{j_2}\\
	\mathbf{w}_{\mathbf{s}}
	\end{bmatrix}
	,
	\begin{bmatrix}
	x_{j_{12}}\\
	x_{j_{12}}\\
	\mathbf{x}_{\mathbf{s}}
	\end{bmatrix}
	\right).
	\end{align*}
	This means that \cref{thm:decomp_f_dependence} applies for every $ \mathbf{s} $.\\
	\Cref{thm:decomp_f_zero_kill} comes directly from the known properties of the anchored decomposition. We prove it explicitly for completeness sake. Split the sum in \cref{eq:decomp_explicit} into two sums according to whenever the indexed subset contains a given cell $ c $ or not. That is,
	\begin{align*}
	f_i^{\mathcal{K}(\mathbf{s})}\left(x;\mathbf{w}_{\mathbf{s}},\mathbf{x}_{\mathbf{s}}\right)
	=
	\sum_{\overline{\mathbf{s}}\subseteq\mathbf{s}\setminus \{c\}} 
	(-1)^{\vert\mathbf{s}\vert-\vert\overline{\mathbf{s}}\vert}
	\hat{f}_i
	\left(x;\mathbf{w}_{\overline{\mathbf{s}}},\mathbf{x}_{\overline{\mathbf{s}}}\right)
	+
	\sum_{\overline{\mathbf{s}}\subseteq\mathbf{s}\setminus \{c\}} 
	(-1)^{\vert\mathbf{s}\vert-\vert\overline{\mathbf{s}}\cup \{c\}\vert}
	\hat{f}_i
	\left(x;
	\begin{bmatrix}
	w_c\\
	\mathbf{w}_{\overline{\mathbf{s}}}
	\end{bmatrix}
	,
	\begin{bmatrix}
	x_c\\
	\mathbf{x}_{\overline{\mathbf{s}}}
	\end{bmatrix}
	\right).
	\end{align*}
	If $ w_c = 0_{ij} $ for some $ j\in\tset $, then, we can apply \cref{eq:oracle_0_equal} on the right sum, which results in
	\begin{align*}
	f_i^{\mathcal{K}(\mathbf{s})}\left(x;\mathbf{w}_{\mathbf{s}},\mathbf{x}_{\mathbf{s}}\right)
	=
	\sum_{\overline{\mathbf{s}}\subseteq\mathbf{s}\setminus \{c\}} 
	(-1)^{\vert\mathbf{s}\vert-\vert\overline{\mathbf{s}}\vert}
	\hat{f}_i
	\left(x;\mathbf{w}_{\overline{\mathbf{s}}},\mathbf{x}_{\overline{\mathbf{s}}}\right)
	-
	\sum_{\overline{\mathbf{s}}\subseteq\mathbf{s}\setminus \{c\}} 
	(-1)^{\vert\mathbf{s}\vert-\vert\overline{\mathbf{s}}\vert}
	\hat{f}_i
	\left(x;
	\mathbf{w}_{\overline{\mathbf{s}}}
	,
	\mathbf{x}_{\overline{\mathbf{s}}}
	\right)
	=
	0_{\yset_i}.
	\end{align*}
	We now prove the $ \impliedby $ direction. That is, any given family of functions $ \{f_i^{\mathbf{k}}\}_{\mathbf{k} \geq \mathbf{0}_{\vert\tset\vert}} $ with the properties in \cref{thm:decomp_f_permutation_invariance,thm:decomp_f_zero_kill,thm:decomp_f_dependence}, defines a valid oracle component $ \hat{f}_i $. To prove that, we show that \cref{defi:func_decomp} will always be respected for any input.\\
	The proof of \cref{eq:oracle_permut} is immediate from \cref{eq:composition_f}, which writes $ \hat{f}_i $ explicitly as a function of $ f_i^{\mathcal{K}(\mathbf{s})} $, together with \cref{thm:decomp_f_permutation_invariance}. Note that the sum \cref{eq:composition_f} being indexed over all subsets $ \overline{\mathbf{s}}\subseteq \mathbf{s} $ is crucial to keep the whole sum invariant under permutations.\\
	We now prove that \cref{eq:oracle_merge} is satisfied. Using \cref{eq:composition_f} on its left hand side gives us
	\begin{align*}
	\sum_{
		\overline{\mathbf{s}}
		\subseteq
		\mathbf{s}
	}
	\left(
	f_i^{\mathcal{K}(\overline{\mathbf{s}})}\left(x;
	\mathbf{w}_{\overline{\mathbf{s}}},
	\mathbf{x}_{\overline{\mathbf{s}}}\right)
	+
	f_i^{\mathcal{K}(\overline{\mathbf{s}})+1_j}
	\left(x;
	\begin{bmatrix}
	w_{j_1} \| w_{j_2}\\
	\mathbf{w}_{\overline{\mathbf{s}}}
	\end{bmatrix}
	,
	\begin{bmatrix}
	x_{j_{12}}\\
	\mathbf{x}_{\overline{\mathbf{s}}}
	\end{bmatrix}
	\right)
	\right).
	\end{align*}
	We now apply \cref{thm:decomp_f_dependence} to the second term of the sum and we obtain
	\begin{align*}
	\sum_{
		\overline{\mathbf{s}}
		\subseteq
		\mathbf{s} 
	}
	&\left[
	f_i^{\mathcal{K}(\overline{\mathbf{s}})}\left(x;
	\mathbf{w}_{\overline{\mathbf{s}}},
	\mathbf{x}_{\overline{\mathbf{s}}}\right)
	+
	f_i^{\mathcal{K}(\overline{\mathbf{s}})+1_j}
	\left(x;
	\begin{bmatrix}
	w_{j_1}\\
	\mathbf{w}_{\overline{\mathbf{s}}}
	\end{bmatrix}
	,
	\begin{bmatrix}
	x_{j_{12}}\\
	\mathbf{x}_{\overline{\mathbf{s}}}
	\end{bmatrix}
	\right)
	+
	f_i^{\mathcal{K}(\overline{\mathbf{s}})+1_j}
	\left(x;
	\begin{bmatrix}
	w_{j_2}\\
	\mathbf{w}_{\overline{\mathbf{s}}}
	\end{bmatrix}
	,
	\begin{bmatrix}
	x_{j_{12}}\\
	\mathbf{x}_{\overline{\mathbf{s}}}
	\end{bmatrix}
	\right)
	\right.
	\\
	&\quad+\left.
	f_i^{\mathcal{K}(\overline{\mathbf{s}})+2_j}
	\left(x;
	\begin{bmatrix}
	w_{j_1}\\
	w_{j_2}\\
	\mathbf{w}_{\overline{\mathbf{s}}}
	\end{bmatrix}
	,
	\begin{bmatrix}
	x_{j_{12}}\\
	x_{j_{12}}\\
	\mathbf{x}_{\overline{\mathbf{s}}}
	\end{bmatrix}
	\right)
	\right],
	\end{align*}
	Using \cref{eq:composition_f} again gives us the right hand side of \cref{eq:oracle_merge}.\\
	We now prove \cref{eq:oracle_0_equal}. Assume $ \mathbf{w} $ has its $ c^{th} $ element equal to $ 0_{ij} $ for some $ j\in\tset $. Then,
	\begin{align*}
	\hat{f}_i(x;\mathbf{w},\mathbf{x})
	=
	\sum_{\overline{\mathbf{s}}\subseteq\mathbf{s}} 
	f_i^{\mathcal{K}(\overline{\mathbf{s}})}
	\left(x;
	\mathbf{w}_{\overline{\mathbf{s}}}
	,
	\mathbf{x}_{\overline{\mathbf{s}}}
	\right)
	=
	\sum_{\overline{\mathbf{s}}\subseteq\mathbf{s}\setminus \{c\}} 
	f_i^{\mathcal{K}(\overline{\mathbf{s}})}
	\left(x;
	\mathbf{w}_{\overline{\mathbf{s}}}
	,
	\mathbf{x}_{\overline{\mathbf{s}}}
	\right)
	=
	\hat{f}_i(x;\mathbf{w}_{-c},\mathbf{x}_{-c}),
	\end{align*}
	where the first equality comes from \cref{eq:composition_f}, the second from applying \cref{thm:decomp_f_zero_kill} and the last one from using \cref{eq:composition_f} again.
\end{proof}
At this point, we have started with the definition of oracle components $ \hat{f}_i $ in \cref{defi:oracle}. Then, we established a bijective correspondence between $ \hat{f}_i $ and a family of coupling components $ \{f_i^{\mathbf{k}}\}_{\mathbf{k} \geq\mathbf{0}_{\vert\tset\vert}} $ in \cref{defi:func_decomp}. Finally, \cref{thm:decomp_f} completed the cycle by making it so that we can also start by first constructing a valid $ \{f_i^{\mathbf{k}}\}_{\mathbf{k} \geq\mathbf{0}_{\vert\tset\vert}} $ and then obtaining the corresponding $ \hat{f}_i $ afterwards.\\
We are now interested in knowing how to manipulate this mathematical object through this new representation. Subsequently, we will provide some examples that illustrate this decomposition and its properties.
\begin{lemma}
	Consider the oracle components $ \hat{f}_i $ and the ones in the sequence $ (\leftidx{^N}{}\hat{f}_i)_{N\in\mathbb{N}} $ such that their corresponding coupling components are, respectively, $ \{f_i^{\mathbf{k}}\}_{\mathbf{k} \geq \mathbf{0}_{\vert\tset\vert}} $ and\\$ \left( \{\leftidx{^N}{}f_i^{\mathbf{k}}\}_{\mathbf{k} \geq \mathbf{0}_{\vert\tset\vert}}\right)_{N\in\mathbb{N}} $. If the output set $ \yset_i $ is a Hausdorff topological vector space, then,
	\begin{align*}
	\lim\limits_{N\to\infty} 
	\leftidx{^N}{}\hat{f}_i
	= 
	\hat{f}_i
	\Longleftrightarrow
	\lim\limits_{N\to\infty} 
	\{\leftidx{^N}{}f_i^{\mathbf{k}}\}_{\mathbf{k} \geq \mathbf{0}_{\vert\tset\vert}}
	=
	\{f_i^{\mathbf{k}}\}_{\mathbf{k} \geq \mathbf{0}_{\vert\tset\vert}}
	\end{align*}
	in the topology of pointwise convergence.
	\label{lemma:convergence}
	\hfill$ \square $
\end{lemma}
\begin{proof}
	We begin by proving the $ \implies $ direction. That is, assume $ \lim\limits_{N\to\infty} 
	\leftidx{^N}{}\hat{f}_i = 
	\hat{f}_i $.\\
	For any $ N\in\mathbb{N} $, we know from \cref{eq:decomp_explicit}, that for any set of cells $ \mathbf{s} $
	\begin{align*}
	\lim\limits_{N\to\infty} 
	\leftidx{^N}{}f_i^{\mathcal{K}(\mathbf{s})}
	\left(x;
	\mathbf{w}_{\mathbf{s}},
	\mathbf{x}_{\mathbf{s}}
	\right)
	&=
	\lim\limits_{N\to\infty} 
	\sum_{\overline{\mathbf{s}}\subseteq\mathbf{s}}
	(-1)^{\vert\mathbf{s}\vert-\vert\overline{\mathbf{s}}\vert}
	\leftidx{^N}{}\hat{f}_i
	\left(x;\mathbf{w}_{\overline{\mathbf{s}}},\mathbf{x}_{\overline{\mathbf{s}}}\right)
	\\
	&=
	\sum_{\overline{\mathbf{s}}\subseteq\mathbf{s}}
	(-1)^{\vert\mathbf{s}\vert-\vert\overline{\mathbf{s}}\vert}
	\lim\limits_{N\to\infty} 
	\leftidx{^N}{}\hat{f}_i
	\left(x;\mathbf{w}_{\overline{\mathbf{s}}},\mathbf{x}_{\overline{\mathbf{s}}}\right)
	\\
	&=
	\sum_{\overline{\mathbf{s}}\subseteq\mathbf{s}}
	(-1)^{\vert\mathbf{s}\vert-\vert\overline{\mathbf{s}}\vert}
	\hat{f}_i
	\left(x;\mathbf{w}_{\overline{\mathbf{s}}},\mathbf{x}_{\overline{\mathbf{s}}}\right)
	\\
	&=
	f_i^{\mathcal{K}(\mathbf{s})}
	\left(x;
	\mathbf{w}_{\mathbf{s}},
	\mathbf{x}_{\mathbf{s}}
	\right).
	\end{align*}
	We now prove the $ \impliedby $ direction. That is, assume $ \lim\limits_{N\to\infty} 
	\{\leftidx{^N}{}f_i^{\mathbf{k}}\}_{\mathbf{k} \geq \mathbf{0}_{\vert\tset\vert}}
	=
	\{f_i^{\mathbf{k}}\}_{\mathbf{k} \geq \mathbf{0}_{\vert\tset\vert}} $.\\
	For any $ N\in\mathbb{N} $, we know from \cref{eq:composition_f}, that for any set of cells $ \mathbf{s} $
	\begin{align*}
	\lim\limits_{N\to\infty} 
	\leftidx{^N}{}\hat{f}_i
	\left(x; \mathbf{w}_{\mathbf{s}},\mathbf{x}_{\mathbf{s}}\right) 
	&=
	\lim\limits_{N\to\infty} 
	\sum_{\overline{\mathbf{s}}\subseteq\mathbf{s}} 
	\leftidx{^N}{}f_i^{\mathcal{K}(\overline{\mathbf{s}})}
	\left(x;
	\mathbf{w}_{\overline{\mathbf{s}}},
	\mathbf{x}_{\overline{\mathbf{s}}}
	\right)
	\\
	&=
	\sum_{\overline{\mathbf{s}}\subseteq\mathbf{s}} 
	\lim\limits_{N\to\infty} 
	\leftidx{^N}{}f_i^{\mathcal{K}(\overline{\mathbf{s}})}
	\left(x;
	\mathbf{w}_{\overline{\mathbf{s}}},
	\mathbf{x}_{\overline{\mathbf{s}}}
	\right)
	\\
	&=
	\sum_{\overline{\mathbf{s}}\subseteq\mathbf{s}} 
	f_i^{\mathcal{K}(\overline{\mathbf{s}})}
	\left(x;
	\mathbf{w}_{\overline{\mathbf{s}}},
	\mathbf{x}_{\overline{\mathbf{s}}}
	\right)
	\\
	&=
	\hat{f}_i
	\left(x; \mathbf{w}_{\mathbf{s}},\mathbf{x}_{\mathbf{s}}\right). 
	\end{align*}
	Note that in a topological vector space the addition operation $ +(\cdot,\cdot) $ is (jointly) continuous. This is what allowed us to convert limits of (finite) sums into (finite) sums of the limits. The Hausdorff property is required to ensure that the limits are always as stated due to uniqueness.
\end{proof}
\begin{lemma}
	For two oracle components $ \hat{f}_i, \hat{g}_i \in \hat{\mathcal{F}}_i $ with coupling components $ \{f_i^{\mathbf{k}}\}_{\mathbf{k} \geq \mathbf{0}_{\vert\tset\vert}} $ and $ \{g_i^{\mathbf{k}}\}_{\mathbf{k} \geq \mathbf{0}_{\vert\tset\vert}} $ respectively, the coupling components of $ \hat{h}_i = \alpha\hat{f}_i  + \hat{g}_i  $ are given by $ \{\alpha f_i^{\mathbf{k}} + g_i^{\mathbf{k}} \}_{\mathbf{k} \geq \mathbf{0}_{\vert\tset\vert}} $, for any scalar $ \alpha $.
	\label{lemma:coupling_additivity}
	\hfill$ \square $
\end{lemma}
\begin{proof}
This comes directly from writing the coupling components explicitly in terms of the oracle components as in \cref{eq:decomp_explicit}. That is,
\begin{align*}
h_i^{\mathcal{K}(\mathbf{s})}\left(x;\mathbf{w}_{\mathbf{s}},\mathbf{x}_{\mathbf{s}}\right)
&=
\sum_{\overline{\mathbf{s}}\subseteq\mathbf{s}}
(-1)^{\vert\mathbf{s}\vert-\vert\overline{\mathbf{s}}\vert}
\left(\alpha \hat{f}_i+\hat{g}_i\right)
\left(x;\mathbf{w}_{\overline{\mathbf{s}}},\mathbf{x}_{\overline{\mathbf{s}}}\right)
\\
&=
\alpha
\left(
\sum_{\overline{\mathbf{s}}\subseteq\mathbf{s}}
(-1)^{\vert\mathbf{s}\vert-\vert\overline{\mathbf{s}}\vert}
\hat{f}_i
\left(x;\mathbf{w}_{\overline{\mathbf{s}}},\mathbf{x}_{\overline{\mathbf{s}}}\right)
\right)
+
\sum_{\overline{\mathbf{s}}\subseteq\mathbf{s}}
(-1)^{\vert\mathbf{s}\vert-\vert\overline{\mathbf{s}}\vert}
\hat{g}_i
\left(x;\mathbf{w}_{\overline{\mathbf{s}}},\mathbf{x}_{\overline{\mathbf{s}}}\right)
\\
&=
\alpha
f_i^{\mathcal{K}(\mathbf{s})}\left(x;\mathbf{w}_{\mathbf{s}},\mathbf{x}_{\mathbf{s}}\right)
+
g_i^{\mathcal{K}(\mathbf{s})}\left(x;\mathbf{w}_{\mathbf{s}},\mathbf{x}_{\mathbf{s}}\right)
\\
&=
\left(
\alpha
f_i^{\mathcal{K}(\mathbf{s})}+g_i^{\mathcal{K}(\mathbf{s})}\right)\left(x;\mathbf{w}_{\mathbf{s}},\mathbf{x}_{\mathbf{s}}\right).
\end{align*}
\end{proof}
We have shown that operating linearly on $ \hat{\mathcal{F}}_{i} $ is completely straightforward, with the coupling components $ \{f_i^{\mathbf{k}}\}_{\mathbf{k} \geq \mathbf{0}_{\vert\tset\vert}} $ being affected component-wise according to the respective linear combination.
\begin{corollary}
	The coupling components of order $ 0 $, that is, $ f_i^{\mathbf{0}} $, which describes the inner dynamics of a cell, are completely free and independent of the remaining coupling components $ \{f_i^{\mathbf{k}} \}_{\mathbf{k} > \mathbf{0} }$.
	\hfill$ \square $
\end{corollary}
\begin{corollary}
	Consider $ f_i^{\mathbf{k}+1_j} = 0_{\yset_i} $ for some $ \mathbf{k} \geq \mathbf{0}_{\vert\tset\vert} $, with $ k_j\geq 1 $, $ j\in\tset $.\\
	Then, $ f_i^{\mathbf{k}} $ is \textbf{additive in the weights} with respect to type $ j $. That is,
	\begin{align}
	f_i^{\mathbf{k}}
	\left(x; 
	\begin{bmatrix}
	w_{j_1} \| w_{j_2}\\ 
	\mathbf{w}_{\mathbf{s}}
	\end{bmatrix},
	\begin{bmatrix}
	x_{j_{12}} \\ 
	\mathbf{x}_{\mathbf{s}}
	\end{bmatrix}
	\right)
	=
	f_i^{\mathbf{k}}
	\left(x; 
	\begin{bmatrix}
	w_{j_1} \\ 
	\mathbf{w}_{\mathbf{s}}
	\end{bmatrix},
	\begin{bmatrix}
	x_{j_{12}} \\ 
	\mathbf{x}_{\mathbf{s}}
	\end{bmatrix}
	\right)
	+
	f_i^{\mathbf{k}}
	\left(x; 
	\begin{bmatrix}
	w_{j_2}\\ 
	\mathbf{w}_{\mathbf{s}}
	\end{bmatrix},
	\begin{bmatrix}
	x_{j_{12}} \\ 
	\mathbf{x}_{\mathbf{s}}
	\end{bmatrix}
	\right).
	\end{align}
	\label{cor:zero_makes_additive}
	\hfill$ \square $
\end{corollary}
The coupling decomposition allows us to define very important concepts that will prove essential in \cref{sec:decomp_basis_comp}.
\begin{defi}
	We say that an oracle component $ \hat{f}_i \in \hat{\mathcal{F}}_i $ with coupling components \\$ \{f_i^{\mathbf{k}}\}_{\mathbf{k} \geq \mathbf{0}_{\vert\tset\vert}} $ has (finite) \textbf{coupling order} $ \gamma_j \in \mathbb{N}_0 $ with respect to the cell type $ j\in\tset $, if there is some $ \mathbf{k} \geq \mathbf{0}_{\vert\tset\vert} $, with $ k_j = \gamma_j $ such that $ f_i^{\mathbf{k}} \neq 0_{\yset_i} $ and there is no such $ \overline{\mathbf{k}} \geq \mathbf{0}_{\vert\tset\vert} $ with $ \overline{k_j} > \gamma_j $.\\
	We say that $ \hat{f}_i \in \hat{\mathcal{F}}_i $ has infinite coupling order $( \gamma_j = \infty) $ with respect to the cell type $ j\in\tset $, if for every $ k_j \in \mathbb{N}_0 $ there is some $ \overline{\mathbf{k}} \geq \mathbf{0}_{\vert\tset\vert}$ such that $ f_i^{\overline{\mathbf{k}}} \neq 0_{\yset_i} $, with $ \overline{k_j}\geq k_j $. \\
	In particular, if $ \gamma_j = 1 $ or $ \gamma_j = 0 $, we say that it is \textbf{additive} or \textbf{uncoupled}, respectively, with regard to $ j\in\tset $.
\hfill$ \square $
\end{defi}
\begin{corollary}
	Consider an oracle component $ \hat{f}_i \in \hat{\mathcal{F}}_i $ with finite coupling order $ \gamma_j \geq 1 $ for some $ j\in\tset $. Then, for any $ \mathbf{k} \geq \mathbf{0}_{\vert\tset\vert} $ such that $ k_j = \gamma_j $, $ f_i^{\mathbf{k}} $ is additive in the weights with respect to type $ j $.
	\label{cor:top_is_additive}
	\hfill$ \square $
\end{corollary}
\begin{proof}
	If is it of order $ k_j = \gamma_j $, then, $ f_i^{\mathbf{k}+1_{j}} = 0_{\yset_i} $. The rest follows from \cref{cor:zero_makes_additive}.
\end{proof}
\begin{lemma}
	Consider an oracle component $ \hat{f}_i \in \hat{\mathcal{F}}_i $ such that for a particular $ j\in\tset $ the associated commutative monoid $ \mathcal{M}_{ij} $ has an annihilator $ a_{ij} $. Then the coupling order of $ \hat{f}_i $ with respect to cell type $ j\in\tset $, is either infinite or $ 0 $ (uncoupled).
	\hfill$ \square $
\end{lemma}
\begin{proof}
	The proof is by contradiction. Assume $ \hat{f}_i $ has finite order $ \gamma_j\geq 1 $. Then, for every $ \mathbf{k} \geq \mathbf{0}_{\vert\tset\vert} $, with $ k_j = \gamma_j $, we have that $ f_i^{\mathbf{k}+1_{j}} = 0_{\yset_i} $.	From \cref{cor:top_is_additive}, $ f_i^{\mathbf{k}}  $ is additive, which implies
	\begin{align}
	f_i^{\mathbf{k}}
	\left(x; 
	\begin{bmatrix}
	w_{j_1} \| a_{ij}\\ 
	\mathbf{w}_{\mathbf{s}}
	\end{bmatrix},
	\begin{bmatrix}
	x_{j_{12}} \\ 
	\mathbf{x}_{\mathbf{s}}
	\end{bmatrix}
	\right)
	=
	f_i^{\mathbf{k}}
	\left(x; 
	\begin{bmatrix}
	w_{j_1} \\ 
	\mathbf{w}_{\mathbf{s}}
	\end{bmatrix},
	\begin{bmatrix}
	x_{j_{12}} \\ 
	\mathbf{x}_{\mathbf{s}}
	\end{bmatrix}
	\right)
	+
	f_i^{\mathbf{k}}
	\left(x; 
	\begin{bmatrix}
	a_{ij}\\ 
	\mathbf{w}_{\mathbf{s}}
	\end{bmatrix},
	\begin{bmatrix}
	x_{j_{12}} \\ 
	\mathbf{x}_{\mathbf{s}}
	\end{bmatrix}
	\right).
	\end{align}
	Since $ w_{j_1} \| a_{ij} = a_{ij}$, this means that $ f_i^{\mathbf{k}} = 0_{\yset_i} $, which contradicts the assumption that $ \hat{f}_i $ is of order $ \gamma_j $.
\end{proof}
We illustrate the decomposition into coupling components scheme with the following examples.
\begin{exmp}
	Consider a single-type network such that 
	\begin{align*}
	\hat{f}_i
	\left(x; \mathbf{w}_{\mathbf{s}},\mathbf{x}_{\mathbf{s}}
	\right) 
	=
	f_i^0(x)
	+
	\left(
	\sum_{c\in\mathbf{s}} w_c x_c
	\right)^{2}.
	\end{align*}	
	We can derive the commutative monoid that defines the edge merging. It has to obey
	\begin{align*}
	f_i^0(x) + \left(\left(w_1\|w_2\right)x_{12}\right)^2 
	&=
	f_i^0(x) + \left(w_1 x_{12} + w_2 x_{12}\right)^2
	\\
	\left(w_1\|w_2\right)^2 x_{12}^2
	&=
	\left(w_1 + w_2\right)^2 x_{12}^2,
	\end{align*}
	from which we conclude that $ w_1\|w_2 $ is either $ w_1+w_2 $ or $ -\left(w_1+w_2\right) $. Note that for either case $ 0\|0 = 0 $. Assume the second option to be true. From the properties of the commutative monoid
	\begin{align*}
	w\|(0\|0) &= (w\|0)\|0
	\\
	w\|0 &= -w\|0
	\\
	-w &= w.
	\end{align*}
	 That is, the second option will only allow the trivial situation in which all edges are $ 0 $. Therefore, we choose $ w_1\|w_2 = w_1 + w_2 $.
	Considering only one in-neighbor, we conclude that
	\begin{align*}
	f_i^{1}(x;w_1,x_1) = \left(w_1 x_1\right)^{2}.
	\end{align*}
	Similarly,
	\begin{align*}
	f_i^2
	\left(x; 
	\begin{bmatrix}w_1 \\ w_2\end{bmatrix},
	\begin{bmatrix}x_{1} \\ x_{2}\end{bmatrix}
	\right)
	=
	2(w_1 x_1)(w_2 x_2).
	\end{align*}
	We can verify that \cref{thm:decomp_f_dependence} of \cref{thm:decomp_f} is satisfied, that is
	\begin{align*}
	f_i^1(x;w_1 \| w_2,x_{12}) 
	&=
	f_i^1(x;w_1,x_{12})
	+
	f_i^1(x;w_2,x_{12})
	+
	f_i^{2}
	\left(x; 
	\begin{bmatrix}w_1 \\ w_2\end{bmatrix},
	\begin{bmatrix}x_{12} \\ x_{12}\end{bmatrix}
	\right)	
	\\
	\left((w_1+w_2) x_{12}\right)^{2}
	&=
	\left(w_1 x_{12}\right)^{2}
	+
	\left(w_2 x_{12}\right)^{2}
	+
	2 w_1 w_2 x_{12}^2,
	\end{align*}
	which is indeed true.
	It can be seen that higher orders will all be $ 0 $.
	That is,
	\begin{align*}
	\hat{f}_i
	\left(x; \mathbf{w}_{\mathbf{s}},\mathbf{x}_{\mathbf{s}}
	\right) 
	&=
	f_i^0(x)
	+
	\sum_{c\in\mathbf{s}} (w_c x_c)^2 
	+
	\sum_{
		\substack{
			c,d\in\mathbf{s}\\
			c\neq d
		}
	} 
	2(w_c x_c)(w_d x_d)
	\\
	&=
	f_i^0(x)
	+
	\sum_{c\in\mathbf{s}} 
	f_i^{1}(x;w_c,x_c)
	+
	\sum_{
		\substack{
			c,d\in\mathbf{s}\\
			c\neq d
		}
	} 
	f_i^2
	\left(x; 
	\begin{bmatrix}w_c \\ w_d\end{bmatrix},
	\begin{bmatrix}x_{c} \\ x_{d}\end{bmatrix}
	\right).
	\end{align*}
	\hfill$ \square $
	\label{exmp:power_2}
\end{exmp}
We now extend the previous example to a general integer power. This requires the following generalization of the binomial coefficient.
\begin{defi}
	Consider $ n \geq 0 $ and $ \mathbf{m} \in \mathbb{Z}^{k} $ such that $ k>1 $ and $ \vert \mathbf{m} \vert = n $. The \textbf{multinomial coefficient} $ \binom{n}{\mathbf{m}} $ is defined as
	\begin{align}
	\binom{n}{\mathbf{m}}
	:=
	\begin{cases}
	\frac{n!}{\prod_{i=1}^{k} m_i !} & \text{if $ \mathbf{m}\geq \mathbf{0}_{k} $},\\
	0 & \text{otherwise}.
	\end{cases}
	\end{align}
	\hfill$ \square $
	\label{defi:multinomia_coef}
\end{defi}
\begin{remark}
	The reason for considering the cases $ \mathbf{m} \in \mathbb{Z}^{k}$ that are outside $\mathbb{N}_{0}^{k}$ and defining them as $ 0 $ is because it greatly simplifies the use of the recurrence relation
	\begin{align}
	\binom{n}{\mathbf{m}} = \sum_{i=1}^{k} \binom{n-1}{\mathbf{m}-1_i}, \quad n>0.
	\end{align}
	This avoids having to treat many corner cases as special. For instance, in the binomial case, defined as $ \binom{n}{m} = \binom{n}{m,n-m} $, this corresponds to $ \binom{n}{m} = \binom{n-1}{m-1} + \binom{n-1}{m}$, for $ n>0 $. The cases $ m =0 $ and $ m =n $ give us $ \binom{n}{0} = \binom{n-1}{-1} + \binom{n-1}{0} = \binom{n-1}{0}$ and $ \binom{n}{n} = \binom{n-1}{n-1} + \binom{n-1}{n} = \binom{n-1}{n-1}$, respectively.
	\hfill$ \square $
\end{remark}
\begin{exmp}
	Consider a single-type network such that 
	\begin{align*}
	\hat{f}_i
	\left(x; \mathbf{w}_{\mathbf{s}},\mathbf{x}_{\mathbf{s}}
	\right) 
	=
	f_i^0(x)
	+
	\left(
	\sum_{c\in\mathbf{s}} w_c x_c
	\right)^{n},
	\end{align*}
	with $ n \in\mathbb{N} $. Then, the coupling components $ f^k $ for $ k > 0 $ are given according to
	\begin{align*}
	f_i^{\vert \mathbf{s} \vert}
	\left(x;
	\mathbf{w}_{\mathbf{s}},
	\mathbf{x}_{\mathbf{s}}
	\right) 
	=
	\sum_{
		\substack{
			\mathbf{m} \geq \mathbf{1}_{\vert \mathbf{s} \vert}\\
			\vert \mathbf{m} \vert = n
		}
	}
	\binom{n}{\mathbf{m}}
	\prod_{c\in\mathbf{s}}
	(w_c x_c)^{m_c}
	=
	n!
	\sum_{
		\substack{
			\mathbf{m} \geq \mathbf{1}_{\vert \mathbf{s} \vert}\\
			\vert \mathbf{m} \vert = n
		}
	}
	\prod_{c\in\mathbf{s}}
	\frac{(w_c x_c)^{m_c}}{m_c!}.
	\end{align*}
	\label{lemma:power_n}
	The proof is by strong induction. Assume this to be true for $ k \in \{1,\ldots,a-1\} $, with $ a>0 $. Choose any set of cells $ \mathbf{s} $ such that $ \vert \mathbf{s} \vert = a$. From the recursive definition we have
	\begin{align*}
	f_i^{a}
	\left(x; 
	\mathbf{w}_{\mathbf{s}},
	\mathbf{x}_{\mathbf{s}}
	\right)
	&=
	\hat{f}_i
	\left(x; \mathbf{w}_{\mathbf{s}},\mathbf{x}_{\mathbf{s}}\right) 
	-
	\sum_{\overline{\mathbf{s}}\subset\mathbf{s}} 
	f_i^{ \vert \overline{\mathbf{s}} \vert} 
	\left(x;
	\mathbf{w}_{\overline{\mathbf{s}}},
	\mathbf{x}_{\overline{\mathbf{s}}}
	\right)\\
	&=
	f_i^0(x)
	+
	\sum_{
		\substack{
			\mathbf{m} \geq \mathbf{0}_{\vert \mathbf{s} \vert}\\
			\vert \mathbf{m} \vert = n
		}
	}
	\binom{n}{\mathbf{m}}
	\prod_{c\in\mathbf{s}}
	(w_c x_c)^{m_c}
	-
	\left[
	f_i^0(x)
	+
	\sum_{
		\substack{
			\overline{\mathbf{s}}\subset\mathbf{s}\\
			\overline{\mathbf{s}} \neq \emptyset
		}
	}
	\sum_{
		\substack{
			\mathbf{m} \geq \mathbf{1}_{\vert \overline{\mathbf{s}} \vert}\\
			\vert \mathbf{m} \vert = n
		}
	}
	\binom{n}{\mathbf{m}}
	\prod_{c\in\overline{\mathbf{s}}}
	(w_c x_c)^{m_c}
	\right]
	\\
	&=
	\sum_{
		\substack{
			\mathbf{m} \geq \mathbf{1}_{\vert \mathbf{s} \vert}\\
			\vert \mathbf{m} \vert = n
		}
	}
	\binom{n}{\mathbf{m}}
	\prod_{c\in\mathbf{s}}
	(w_c x_c)^{m_c}
	\\
	&=
	n!
	\sum_{
		\substack{
			\mathbf{m} \geq \mathbf{1}_{\vert \mathbf{s} \vert}\\
			\vert \mathbf{m} \vert = n
		}
	}
	\prod_{c\in\mathbf{s}}
	\frac{(w_c x_c)^{m_c}}{m_c!}.
	\end{align*}
	That is, the case $ k=a $ is also satisfied, which concludes the proof. Note that the case $ k=1 $ comes for free due to using strong induction (its hypothesis is vacuously true), although it is trivial to verify.
	\hfill$ \square $
\end{exmp}
\begin{remark}
	Note that for $ k > n$ there are no multi-indexes that satisfy simultaneously $ \mathbf{m} \geq \mathbf{1}_{k} $ and $ \vert \mathbf{m} \vert = n $. Therefore, $ f_i^k = 0 $ for such $ k $. The coupling order is then $ \gamma = n $.
	\hfill$ \square $
\end{remark}
\begin{remark}
	As a sanity check we verify that \cref{thm:decomp_f_dependence} of \cref{thm:decomp_f} is satisfied.\\
	First we can derive the commutative monoid that defines the edge merging. It has to obey
	\begin{align*}
	f_i^0(x) + \left(\left(w_1\|w_2\right)x_{12}\right)^n 
	&=
	f_i^0(x) + \left(w_1 x_{12} + w_2 x_{12}\right)^n
	\\
	\left(w_1\|w_2\right)^n x_{12}^n
	&=
	\left(w_1 + w_2\right)^n x_{12}^n.
	\end{align*}
	If $ n $ is odd, then $ w_1\|w_2 = w_1+w_2 $. If $ n $ is even, we are in the same situation as in \cref{exmp:power_2} and $ w_1\|w_2 = w_1+w_2 $ for us to be in a non-trivial setting.
	Now, to verify
	\begin{align*}
	f_i^{\vert \overline{\mathbf{s}} \vert + 2}
	\left(x; 
	\begin{bmatrix}
	w_{1} \\ 
	w_{2} \\
	\mathbf{w}_{\overline{\mathbf{s}}}
	\end{bmatrix},
	\begin{bmatrix}
	x_{12} \\ 
	x_{12} \\ 
	\mathbf{x}_{\overline{\mathbf{s}}}
	\end{bmatrix}
	\right)
	&=
	f_i^{\vert \overline{\mathbf{s}} \vert + 1}
	\left(x; 
	\begin{bmatrix}
	w_{1} \| w_{2}\\ 
	\mathbf{w}_{\overline{\mathbf{s}}}
	\end{bmatrix},
	\begin{bmatrix}
	x_{12} \\ 
	\mathbf{x}_{\overline{\mathbf{s}}}
	\end{bmatrix}
	\right)
	\\
	&\quad
	-
	f_i^{\vert \overline{\mathbf{s}} \vert + 1}
	\left(x; 
	\begin{bmatrix}
	w_{1} \\ 
	\mathbf{w}_{\overline{\mathbf{s}}}
	\end{bmatrix},
	\begin{bmatrix}
	x_{12} \\ 
	\mathbf{x}_{\overline{\mathbf{s}}}
	\end{bmatrix}
	\right)
	-
	f_i^{\vert \overline{\mathbf{s}} \vert + 1}
	\left(x; 
	\begin{bmatrix}
	w_{2}\\ 
	\mathbf{w}_{\overline{\mathbf{s}}}
	\end{bmatrix},
	\begin{bmatrix}
	x_{12} \\ 
	\mathbf{x}_{\overline{\mathbf{s}}}
	\end{bmatrix}
	\right),
	\end{align*}
	we note that
	\begin{align*}
	\binom{n}{m_1,m_2,\overline{\mathbf{m}}}
	=
	\binom{m_1 + m_2}{m_1, m_2}
	\binom{n}{m_1 + m_2,\overline{\mathbf{m}}}.
	\end{align*}
	Using this, the left hand side can be written as
	\begin{align*}
	\sum_{m_1,m_2 \geq 1}	
	\binom{m_1 + m_2}{m_1, m_2}
	w_1^{m_1}
	w_2^{m_2}
	x_{12}^{m_1+m_2}
	\sum_{
		\substack{
			\overline{\mathbf{m}} \geq \mathbf{1}_{\vert \overline{\mathbf{s}} \vert}\\
			m_1 + m_2 + \vert \overline{\mathbf{m}} \vert = n
		}
	}
	\binom{n}{m_1 +m_2,\overline{\mathbf{m}}}
	\prod_{c\in\overline{\mathbf{s}}}
	(w_c x_c)^{m_c}
	\end{align*}
	and the right hand side as
	\begin{align*}
	\sum_{m_{12}\geq 1}
	\left(
	(w_1+w_2)^{m_{12}} - w_1^{m_{12}} - w_2^{m_{12}}
	\right)
	x_{12}^{m_{12}}
	\sum_{
		\substack{
			\overline{\mathbf{m}} \geq \mathbf{1}_{\vert \overline{\mathbf{s}} \vert}\\
			m_{12} + \vert \overline{\mathbf{m}} \vert = n
		}
	}
	\binom{n}{m_{12},\overline{\mathbf{m}}}
	\prod_{c\in\overline{\mathbf{s}}}
	(w_c x_c)^{m_c}.
	\end{align*}
	Using the binomial theorem on $ (w_1 + w_2)^{m_{12}} $ we see that for a fixed $ m_{12} $ we have that
	\begin{align*}
	(w_1+w_2)^{m_{12}} - w_1^{m_{12}} - w_2^{m_{12}}
	=
	\sum_{
		\substack{
		m_1,m_2 \geq 1\\
		m_1 + m_2 = m_{12}	
	}
	}	
	\binom{m_{12}}{m_1, m_2}
	w_1^{m_1}
	w_2^{m_2}.
	\end{align*}
	Therefore, both sides are the same.
	\hfill$ \square $
\end{remark}
We now extend \cref{lemma:power_n} to the polynomial case.
\begin{exmp}
	From \cref{lemma:power_n} and \cref{lemma:coupling_additivity} we have that for single-type networks such that 
	\begin{align*}
	\hat{f}_i
	\left(x; \mathbf{w}_{\mathbf{s}},\mathbf{x}_{\mathbf{s}}
	\right) 
	=
	f_i^0(x)
	+
	\sum_{n=1}^{N}
	a_{n}
	\left(
	\sum_{c\in\mathbf{s}} w_c x_c
	\right)^{n},
	\end{align*}
	the coupling components $ f_i^k $ for $ k > 0 $ are given according to
	\begin{align*}
	f_i^{\vert \mathbf{s} \vert}
	\left(x;
	\mathbf{w}_{\mathbf{s}},
	\mathbf{x}_{\mathbf{s}}
	\right) 
	&=
	\sum_{n=1}^{N}
	a_n n!
	\sum_{
		\substack{
			\mathbf{m} \geq \mathbf{1}_{\vert \mathbf{s} \vert}\\
			\vert \mathbf{m} \vert = n
		}
	}
	\prod_{c\in\mathbf{s}}
	\frac{(w_c x_c)^{m_c}}{m_c!}.
	\end{align*}	
	\label{cor:polynomial}
	\hfill$ \square $
\end{exmp}
\begin{exmp}
	Consider the exponential case 
	\begin{align*}
	\hat{f}_i
	\left(x; \mathbf{w}_{\mathbf{s}},\mathbf{x}_{\mathbf{s}}
	\right) 
	=
	f_i^0(x)
	+
	\exp
	\left(
	\sum_{c\in\mathbf{s}} w_c x_c
	\right)-1.
	\end{align*}
	The coupling components $ f^k $ for $ k > 0 $ are given according to
	\begin{align*}
	f_i^{\vert \mathbf{s} \vert}
	\left(x;
	\mathbf{w}_{\mathbf{s}},
	\mathbf{x}_{\mathbf{s}}
	\right) 
	&=
	\prod_{c\in\mathbf{s}}
	\left(
	\exp
	\left(w_c x_c\right)-1
	\right).
	\end{align*}
	This is proven by creating the sequence of oracle components $ (\leftidx{^N}{}\hat{f}_i)_{N\in\mathbb{N}} $ such that
		\begin{align*}
		\leftidx{^N}{}\hat{f}_i
		\left(x; \mathbf{w}_{\mathbf{s}},\mathbf{x}_{\mathbf{s}}
		\right) 
		=
		f_i^0(x)
		+
		\sum_{n=1}^{N}
		\frac{1}{n!}
		\left(
		\sum_{c\in\mathbf{s}} w_c x_c
		\right)^{n}.
		\end{align*}
		That is, the oracle components obtained by replacing $ \exp(\cdot)-1 $ by its $ N^{th} $ order Taylor series truncation. From \cref{cor:polynomial} we know that for the sequence $ \left( \{\leftidx{^N}{}f_i^{k}\}_{k \geq 0}\right)_{N\in\mathbb{N}} $, the components $ \leftidx{^N}{}f^k $, for $ k > 0 $ are given according to
		\begin{align*}
		\leftidx{^N}{}f_i^{\vert \mathbf{s} \vert}
		\left(x;
		\mathbf{w}_{\mathbf{s}},
		\mathbf{x}_{\mathbf{s}}
		\right) 
		&=
		\sum_{
			\substack{
				\mathbf{m} \geq \mathbf{1}_{\vert \mathbf{s} \vert}\\
				\vert \mathbf{m} \vert \leq N
			}
		}
		\prod_{c\in\mathbf{s}}
		\frac{(w_c x_c)^{m_c}}{m_c!}.
		\end{align*}
		Since we know that $ \lim\limits_{N\to\infty} \leftidx{^N}{}\hat{f}_i = \hat{f}_i $ (pointwise), from \cref{lemma:convergence} we conclude that $ \lim\limits_{N\to\infty} \leftidx{^N}{}f_i^{\vert \mathbf{s} \vert} = f_i^{\vert \mathbf{s} \vert} $. That is,
		\begin{align*}
		f_i^{\vert \mathbf{s} \vert}
		\left(x;
		\mathbf{w}_{\mathbf{s}},
		\mathbf{x}_{\mathbf{s}}
		\right)
		&=
		\sum_{\mathbf{m} \geq \mathbf{1}_{\vert \mathbf{s} \vert}}
		\prod_{c\in\mathbf{s}}
		\frac{(w_c x_c)^{m_c}}{m_c!}.
		\end{align*}
		Here, the infinite sum $ \sum_{\mathbf{m} \geq \mathbf{1}_{\vert \mathbf{s} \vert}} $ is taken as 
		\begin{align*}
		\sum_{\mathbf{m} \geq \mathbf{1}_{\vert \mathbf{s} \vert}}
		:=
		\lim\limits_{N\to\infty}
		\sum_{
			\substack{
				\mathbf{m} \geq \mathbf{1}_{\vert \mathbf{s} \vert}\\
				\vert \mathbf{m} \vert \leq N
			}
		}.
		\end{align*}
		We can prove, however, that this particular infinite sum is absolutely convergent on the index set $ \mathbf{m} \geq \mathbf{1}_{\vert \mathbf{s} \vert} $. That is,
		\begin{align*}
		\sum_{\mathbf{m} \geq \mathbf{1}_{\vert \mathbf{s} \vert}}
		\left|
		\prod_{c\in\mathbf{s}}
		\frac{(w_c x_c)^{m_c}}{m_c!}
		\right|
		=
		\sum_{\mathbf{m} \geq \mathbf{1}_{\vert \mathbf{s} \vert}}
		\prod_{c\in\mathbf{s}}
		\frac{\left|w_c x_c\right|^{m_c}}{m_c!}
		=
		\prod_{c\in\mathbf{s}}
		\sum_{m_c \geq 1}
		\frac{\left|w_c x_c\right|^{m_c}}{m_c!}
		=
		\prod_{c\in\mathbf{s}}
		\left(
		\exp
		\left(\left|w_c x_c\right|\right)-1
		\right)
		< \infty.
		\end{align*}
		This means that the order does not matter and we can freely rearrange the sum into
		\begin{align*}
		f_i^{\vert \mathbf{s} \vert}
		\left(x;
		\mathbf{w}_{\mathbf{s}},
		\mathbf{x}_{\mathbf{s}}
		\right)
		=
		\sum_{\mathbf{m} \geq \mathbf{1}_{\vert \mathbf{s} \vert}}
		\prod_{c\in\mathbf{s}}
		\frac{(w_c x_c)^{m_c}}{m_c!}
		=
		\prod_{c\in\mathbf{s}}
		\sum_{m_c \geq 1}
		\frac{(w_c x_c)^{m_c}}{m_c!}
		=
		\prod_{c\in\mathbf{s}}
		\left(
		\exp
		\left(w_c x_c\right)-1
		\right).
		\end{align*}
	\label{cor:exponential}
	\hfill$ \square $
\end{exmp}
We now extend the previous results for multi-type networks.
\begin{exmp}
	Consider a multi-type network such that 
	\begin{align*}
	\hat{f}_i
	\left(x;
	\mathbf{w}_{\mathbf{s}}
	,
	\mathbf{x}_{\mathbf{s}}
	\right) 
	=
	f_i^{\mathbf{0}}(x)
	+
	\prod_{j\in \tset}
	\left(
	\sum_{c\in\mathbf{s}_j} w_c x_c
	\right)^{n_j}.
	\end{align*}
	with $ \mathbf{n} > \mathbf{0}_{\vert \tset \vert} $, and where $ \mathbf{s}_j \subseteq \mathbf{s}$ represents the subset of cells that are of type $ j\in \tset $. We use the definition $ 0^0 = 1 $, which is standard and avoids many corner cases (e.g., consider the binomial theorem applied to $ (x+0)^n $). Then, the coupling components $ \{f^{\mathbf{k}}\}_{\mathbf{k} > \mathbf{0}_{\vert \tset\vert}} $ are given according to
	\begin{align*}
	f_i^{\mathcal{K}(\mathbf{s})}
	\left(x; 
	\mathbf{w}_{\mathbf{s}},
	\mathbf{x}_{\mathbf{s}}
	\right)
	&=
	\prod_{j\in \tset}
	n_j!
	\sum_{
		\substack{
			\mathbf{m} \geq \mathbf{1}_{\vert \mathbf{s}_j \vert}\\
			\vert \mathbf{m} \vert = n_j
		}
	}
	\prod_{c\in\mathbf{s}_j}
	\frac{(w_c x_c)^{m_c}}{m_c!}.
	\end{align*}
	Note that expanding the outer product gives us
	\begin{align*}
	\sum_{
		\substack{
			\mathbf{m} \geq \mathbf{1}_{\vert \mathbf{s} \vert}
			\\
			\vert \mathbf{m}_1 \vert = n_1\\
			\ldots\\
			\vert \mathbf{m}_{\vert \tset\vert} \vert = n_{\vert \tset\vert}
		}
	}
	\left(
	\prod_{j \in \tset}
	n_j!
	\right)
	\prod_{c\in\mathbf{s}}
	\frac{(w_c x_c)^{m_c}}{m_c!}
	,
	\quad
	\text{with }
	\mathbf{m} := 
	\begin{bmatrix} \mathbf{m}_1 \\ 
	\vdots\\ 
	\mathbf{m}_{\vert \tset\vert}
	\end{bmatrix}.
	\end{align*}
	\label{lemma:power_n_multi}
	This is now proven by strong induction. Assume this to be true for $ \vert\mathbf{k}\vert \in \{1,\ldots,a-1\} $, with $ a>0 $. 
	For any $ \mathbf{k} > \mathbf{0}_{\vert \tset\vert} $ with $ \vert \mathbf{k} \vert = a $, choose any set of cells $ \mathbf{s} := \{\mathbf{s}_1 \cup \ldots \cup \mathbf{s}_{\vert \tset \vert}\} $, such that for every $ j\in\tset $, $ \mathbf{s}_j $ is a set of cells of type $ j $ and $  \vert \mathbf{s}_j \vert = k_j $.\\
	From the recursive definition,	$ f_i^{\mathcal{K}(\mathbf{s})}
	\left(x; 
	\mathbf{w}_{\mathbf{s}},
	\mathbf{x}_{\mathbf{s}}
	\right)
	=
	\hat{f}_i
	\left(x; \mathbf{w}_{\mathbf{s}},\mathbf{x}_{\mathbf{s}}\right) 
	-
	\sum_{\overline{\mathbf{s}}\subset\mathbf{s}} 
	f_i^{\mathcal{K}(\overline{\mathbf{s}})}
	\left(x;
	\mathbf{w}_{\overline{\mathbf{s}}},
	\mathbf{x}_{\overline{\mathbf{s}}}
	\right) $ becomes
		\begin{align*}
		&
		\prod_{j\in \tset}
		\sum_{
			\substack{
				\mathbf{m} \geq \mathbf{0}_{\vert \mathbf{s}_j \vert}\\
				\vert \mathbf{m} \vert = n_j
			}
		}
		\binom{n_j}{\mathbf{m}_j}
		\prod_{c\in\mathbf{s}_j}
		(w_c x_c)^{m_c}
		-
		\left[
		\sum_{
			\substack{
				\overline{\mathbf{s}}\subset\mathbf{s}\\
				\overline{\mathbf{s}} \neq \emptyset
			}
		}
		\sum_{
			\substack{
				\mathbf{m} \geq \mathbf{1}_{\vert \overline{\mathbf{s}} \vert}
				\\
				\vert \mathbf{m}_1 \vert = n_1\\
				\ldots\\
				\vert \mathbf{m}_{\vert \tset \vert} \vert = n_{\vert \tset \vert}
			}
		}
		\left(
		\prod_{j\in \tset}
		n_j!
		\right)
		\prod_{c\in \overline{\mathbf{s}}}
		\frac{(w_c x_c)^{m_c}}{m_c!}
		\right]
		\\
		&\quad=
		\sum_{
			\substack{
				\mathbf{m} \geq \mathbf{0}_{\vert \mathbf{s} \vert}
				\\
				\vert \mathbf{m}_1 \vert = n_1\\
				\ldots\\
				\vert \mathbf{m}_{\vert \tset \vert} \vert = n_{\vert \tset \vert }
			}
		}
		\left(
		\prod_{j \in \tset}
		n_j!
		\right)
		\prod_{c\in \mathbf{s}}
		\frac{(w_c x_c)^{m_c}}{m_c!}
		-
		\sum_{
			\substack{
				\overline{\mathbf{s}}\subset\mathbf{s}\\
				\overline{\mathbf{s}} \neq \emptyset
			}
		}
		\sum_{
			\substack{
				\mathbf{m} \geq \mathbf{1}_{\vert \overline{\mathbf{s}} \vert}
				\\
				\vert \mathbf{m}_1 \vert = n_1\\
				\ldots\\
				\vert \mathbf{m}_{\vert \tset \vert} \vert = n_{\vert \tset \vert }
			}
		}
		\left(
		\prod_{j\in \tset}
		n_j!
		\right)
		\prod_{c\in \overline{\mathbf{s}}}
		\frac{(w_c x_c)^{m_c}}{m_c!}
		\\
		&\quad=
		\sum_{
			\substack{
				\mathbf{m} \geq \mathbf{1}_{\vert \mathbf{s} \vert}
				\\
				\vert \mathbf{m}_1 \vert = n_1\\
				\ldots\\
				\vert \mathbf{m}_{\vert \tset \vert} \vert = n_{\vert \tset \vert }
			}
		}
		\left(
		\prod_{j \in \tset}
		n_j!
		\right)
		\prod_{c\in \mathbf{s}}
		\frac{(w_c x_c)^{m_c}}{m_c!}
		\\
		&\quad=
		\prod_{j\in \tset}
		n_j!
		\sum_{
			\substack{
				\mathbf{m} \geq \mathbf{1}_{\vert \mathbf{s}_j \vert}\\
				\vert \mathbf{m} \vert = n_j
			}
		}
		\prod_{c\in\mathbf{s}_j}
		\frac{(w_c x_c)^{m_c}}{m_c!}.
		\end{align*}
		That is, the case $ \vert \mathbf{k} \vert = a $ is also satisfied, which concludes the proof.
	\hfill$ \square $
\end{exmp}
\begin{exmp}
	From \cref{lemma:power_n_multi} and \cref{lemma:coupling_additivity}, we have that for multi-type networks such that
	\begin{align*}
	\hat{f}_i
	\left(x;
	\mathbf{w}_{\mathbf{s}}
	,
	\mathbf{x}_{\mathbf{s}}
	\right) 
	=
	f_i^{\mathbf{0}}(x)
	+
	\sum_{\mathbf{n} > \mathbf{0}_{\vert \tset \vert}}
	a_{\mathbf{n}}
	\prod_{j\in \tset}
	\left(
	\sum_{c\in\mathbf{s}_j} w_c x_c
	\right)^{n_j},
	\label{lemma:power_poly_multi}
	\end{align*}
	with $ \{a_{\mathbf{n}}\}_{\mathbf{n} > \mathbf{0}_{\vert \tset \vert}} $ with finite support,	the coupling components $ \{f^{\mathbf{k}}\}_{\mathbf{k} > \mathbf{0}_{\vert \tset\vert}} $ are given according to
	\begin{align*}
	f_i^{\mathcal{K}(\mathbf{s})}
	\left(x; 
	\mathbf{w}_{\mathbf{s}},
	\mathbf{x}_{\mathbf{s}}
	\right)
	&=
	\sum_{\mathbf{n} > \mathbf{0}_{\vert \tset \vert}}
	a_{\mathbf{n}}
	\prod_{j\in \tset}
	n_j!
	\sum_{
		\substack{
			\mathbf{m} \geq \mathbf{1}_{\vert \mathbf{s}_j \vert}\\
			\vert \mathbf{m} \vert = n_j
		}
	}
	\prod_{c\in\mathbf{s}_j}
	\frac{(w_c x_c)^{m_c}}{m_c!}.
	\end{align*}	
	\label{cor:polynomial_multi}
	\hfill$ \square $
\end{exmp}
The following example illustrates how an oracle component for multi-type networks can have the form of \cref{cor:polynomial_multi} while being constructed in a more natural manner.
\begin{exmp}
Consider multi-type networks such that
\begin{align*}
\hat{f}_i
\left(x;
\mathbf{w}_{\mathbf{s}}
,
\mathbf{x}_{\mathbf{s}}
\right) 
=
f_i^{\mathbf{0}}(x)
+
F\left(
\sum_{j\in\tset}
F_j\left(
\sum_{c\in\mathbf{s}_j} w_c x_c
\right)
\right).
\end{align*}
with $ F(X) = \sum_{n=1}^{N} a_n X^n $ and $ F_j(X) = \sum_{n=1}^{N_j} a_n^j X^n $, for all $ j\in\tset $. Then, we have that
\begin{align*}
F\left(
\sum_{j\in\tset}
F_j\left(
\sum_{c\in\mathbf{s}_j} w_c x_c
\right)
\right)
&=
\sum_{n=1}^{N}
a_n
\left(
\sum_{j\in\tset}
F_j\left(
\sum_{c\in\mathbf{s}_j} w_c x_c
\right)
\right)^{n}
\\
&=
\sum_{n=1}^{N}
a_n
\sum_{
	\substack{
	\mathbf{m} \geq \mathbf{0}_{\vert\tset\vert}\\
	\vert \mathbf{m} \vert = n
	}
}
\binom{n}{\mathbf{m}}
\prod_{j\in\tset}
F_j\left(
\sum_{c\in\mathbf{s}_j} w_c x_c
\right)^{m_j}
\\
&=
\sum_{n=1}^{N}
a_n
n!
\sum_{
	\substack{
		\mathbf{m} \geq \mathbf{0}_{\vert\tset\vert}\\
		\vert \mathbf{m} \vert = n
	}
}
\prod_{j\in\tset}
\frac{1}{m_j!}
F_j\left(
\sum_{c\in\mathbf{s}_j} w_c x_c
\right)^{m_j}
\\
&=
\sum_{n=1}^{N}
a_n
n!
\sum_{
	\substack{
		\mathbf{m} \geq \mathbf{0}_{\vert\tset\vert}\\
		\vert \mathbf{m} \vert = n
	}
}
\prod_{j\in\tset}
\frac{1}{m_j!}
\left(
\sum_{l=1}^{N_j}
a_{l}^{j}
\left(
\sum_{c\in\mathbf{s}_j} w_c x_c
\right)^l
\right)^{m_j}.
\end{align*}
Note that the product of polynomials can be obtained by the convolution of their coefficients. Therefore, raising a polynomial to the power $ n $ is equivalent to convolving its coefficients with themselves $ n $ times. In particular, $ (\sum_{n=1}^{N} a_n X^n)^m $ can be written as $ \sum_{n=m}^{Nm} b_n X^n $, with
\begin{align*}
b_n =
\sum_{
	\substack{
		\mathbf{l} \geq \mathbf{1}_{m}\\
		\mathbf{l} \leq N \, \mathbf{1}_{m}\\
		\vert \mathbf{l} \vert = n
	}
}
\prod_{i=1}^{m} 
a_{l_i}.
\end{align*}
Therefore, for any fixed $ \mathbf{n} > \mathbf{0}_{\vert \tset \vert} $, the coefficient $ a_{\mathbf{n}} $ associated with $ \prod_{j\in \tset}
\left(
\sum_{c\in\mathbf{s}_j} w_c x_c
\right)^{n_j} $ as in \cref{cor:polynomial_multi} is given by
\begin{align*}
a_{\mathbf{n}} = 
\sum_{n=1}^{N}
a_n n!
\sum_{
	\substack{
		\mathbf{m} \geq \mathbf{0}_{\vert\tset\vert}\\
		\vert \mathbf{m} \vert = n
	}
}
\prod_{j\in\tset}
\left(
\frac{1}{m_j !}
\sum_{
	\substack{
		\mathbf{l} \geq \mathbf{1}_{m_j}\\
		\mathbf{l} \leq N_j \, \mathbf{1}_{m_j}\\
		\vert \mathbf{l} \vert = n_j
	}
}
\prod_{i=1}^{m_j} 
a_{l_i}^{j}
\right).
\end{align*}
Note that the outer function $ F $ is the one responsible for the existence of non-zero coupling components with mixed typing. Consider, for instance, $ N =1 $. Then, $ a_{\mathbf{n}} $ with $ \mathbf{n} > \mathbf{0}_{\vert \tset \vert} $ can only be non-zero whenever $ \mathbf{n} = a \, 1_j $ for some $ j\in\tset $. The reason is that the only way for the innermost sum to be non-zero whenever $ m_j =0 $, is for $ n_j $ to be zero as well. In that situation, we have a sum over one valid index (the $ 0 $-tuple) of an empty product, which results in $ 1 $. Similarly, if we consider $ N=2 $, then, $ a_{\mathbf{n}} $ with $ \mathbf{n} > \mathbf{0}_{\vert \tset \vert} $ can only be non-zero whenever $ \mathbf{n} = a \, 1_j + b\, 1_k $ for some $ j,k\in\tset $, and so on.
\hfill$ \square $
\end{exmp}
We now introduce the second composition scheme.
\section{Decomposition into basis components} \label{sec:decomp_basis_comp}
In \cref{sec:decomp_coupl_comp} we introduced a scheme that decomposes any given oracle component $ \hat{f}_i \in \hat{\mathcal{F}}_i $ into a family of coupling components $ \{f_i^{\mathbf{k}}\}_{\mathbf{k} \geq \mathbf{0}_{\vert\tset\vert}} $ that have the properties described in \cref{thm:decomp_f}. With that, one can easily verify in a very systematic way if some function $ \hat{f}_i $ that is used to model the behavior of cells in a network satisfies the properties given by \cref{defi:oracle}.\\
Although this decomposition works well for verification, it is lacking from the perspective of design. The reason for this is the \cref{thm:decomp_f_dependence} of \cref{thm:decomp_f}. It forces all coupling components $ \{f_i^{\mathbf{k}}\}_{\mathbf{k} \geq \mathbf{0}_{\vert\tset\vert}} $ to be interdependent. Therefore, it is not clear at all what are exactly the degrees of freedom that are available for us, nor how one would even start when choosing such functions.\\
In this section, we use the previous decomposition as an essential stepping stone in order to create another with more desirable properties.\\
For that, we require the use of the multiplicity notation and also Stirling numbers of the first and second kinds, which we now describe.
\subsection{Multiplicity notation}
We now introduce the multiplicity notation, which simplifies the following work.\\
By $ \mathbf{m}\mathbf{w}_{\mathbf{s}} $, with $ \mathbf{m} \geq \mathbf{0}_{\vert \mathbf{s}\vert}$, we mean that each entry $ w_c $ of the vector $ \mathbf{w}_{\mathbf{s}} $ is expanded into $ m_c $ entries of the same value. For instance, consider
\begin{align*}
\mathbf{w}_{\mathbf{s}} 
= 
\begin{bmatrix}
w_a \\ w_b
\end{bmatrix}
,\quad
\mathbf{m}
=
\begin{bmatrix}
1 \\ 2
\end{bmatrix}.
\qquad \text{Then, } 
\mathbf{m}\mathbf{w}_{\mathbf{s}}  
= 
\begin{bmatrix}
w_a \\ w_b \\ w_b
\end{bmatrix}.
\end{align*}
Note that the number of elements in the resulting vector $ \mathbf{m}\mathbf{w}_{\mathbf{s}} $ is $ \vert \mathbf{m}\mathbf{s} \vert = \vert \mathbf{m} \vert $, which in this case is $ 3 $. Moreover, multiplicities can be composed. That is, we can apply some $ \overline{\mathbf{m}} $ to the previous $ \mathbf{m}\mathbf{w}_{\mathbf{s}} $, in order to obtain $ \overline{\mathbf{m}}\mathbf{m}\mathbf{w}_{\mathbf{s}}$, which requires $ \overline{\mathbf{m}} \geq  \mathbf{0}_{\vert \mathbf{m} \vert} $. For instance, we could have
\begin{align*}
\overline{\mathbf{m}}
=
\begin{bmatrix}
2 \\ 1 \\ 2
\end{bmatrix},
\quad
\overline{\mathbf{m}}\mathbf{m}\mathbf{w}_{\mathbf{s}}= 
\begin{bmatrix}
w_a\\ w_a \\ 
\hline
w_b \\ 
\hline
w_b \\ w_b
\end{bmatrix},
\end{align*}
where the horizontal bars are just for illustration purposes in order to make the expansion of $ \mathbf{m}\mathbf{w}_{\mathbf{s}}  $ into $ \overline{\mathbf{m}}\mathbf{m}\mathbf{w}_{\mathbf{s}} $ clearer. Note that $ \vert \overline{\mathbf{m}}\mathbf{m}\mathbf{s} \vert = \vert \overline{\mathbf{m}} \vert = 5 $. Moreover, applying the successive multiplicities ($ \overline{\mathbf{m}} $ after $ \mathbf{m} $) is equivalent to applying a single multiplicity $ \mathbf{M} $, in our case, we have,
\begin{align*}
\mathbf{M}
=
\begin{bmatrix}
2 \\ 3
\end{bmatrix},
\quad
\mathbf{M}\mathbf{w}_{\mathbf{s}} 
=
\begin{bmatrix}
w_a\\ w_a \\ 
\hline
w_b \\ w_b \\ w_b
\end{bmatrix}.
\end{align*}
Note that $ \vert \mathbf{M} \vert = \vert \overline{\mathbf{m}} \vert = 5 $. We say that $ \mathbf{M} = \overline{\mathbf{m}}\mathbf{m} $, where $ \overline{\mathbf{m}}\mathbf{m} $ is the composition of the two multiplicities $ \overline{\mathbf{m}}$ and $\mathbf{m} $. This should not be confused with extending $ \mathbf{m} $ according to $ \overline{\mathbf{m}} $, which has a completely different meaning.
In our example, we have
\begin{align*}
\overline{\mathbf{m}}\mathbf{m}
=
\begin{bmatrix}
2 \\ 
\hline
1 \\ 2
\end{bmatrix}
\begin{bmatrix}
1 \\ 2
\end{bmatrix}
=
\begin{bmatrix}
2 \\ 3
\end{bmatrix}
=
\mathbf{M}.
\end{align*}
The $ \vert \mathbf{m} \vert $ entries of $ \overline{\mathbf{m}} $ can be divided according to the values of $ \mathbf{m} $, which in this case is a first block with one element and a second block with two elements. Note that each block will affect a different element of the original vector we are applying $ \overline{\mathbf{m}}\mathbf{m}$ to (e.g., $ \mathbf{w}_{\mathbf{s}} $), that is, each element of the $ i^{th} $ block of $ \overline{\mathbf{m}} $ expands the $ i^{th} $ element of $ \mathbf{w}_{\mathbf{s}}  $ that amount of times. In conclusion, to find the equivalent multiplicity $ \mathbf{M} $ we just need to sum each block of the multiplicity $ \overline{\mathbf{m}} $, in which, the blocks are defined according to $ \mathbf{m} $.
\subsection{Stirling numbers}
The Stirling numbers of the first and second kinds are integers that appear in combinatorics, in particular when studying partitions and permutations \cite{comtet2012advanced}. In this section we will define them with respect to their recurrence relations.
The related results that will be used in the sections are presented and proved in \cref{apend:decomposition_finite_intermediate_results}.
\begin{defi}
	The unsigned Stirling numbers of the first kind, $ \stirst{n}{k} $, with $ n,k \geq 0$, are given by the recurrence relation
	\begin{align*}
	\stirst{n}{k} = (n-1) \stirst{n-1}{k} + \stirst{n-1}{k-1}
	,\quad n,k>0,
	\end{align*}
	together with the boundary conditions
	\begin{align*}
	\stirst{0}{0} &= 1,\\
	\stirst{0}{k} &= 0, \quad k > 0,\\
	\stirst{n}{0} &= 0, \quad n > 0.
	\end{align*}
	\label{defi:stirling_1}
	\hfill$ \square $
\end{defi}
\begin{defi}
	The Stirling numbers of the second kind, $ \stirnd{n}{k} $, with $ n,k \geq 0$, are given by the recurrence relation
	\begin{align*}
	\stirnd{n}{k} = k \stirnd{n-1}{k} + \stirnd{n-1}{k-1}
	,\quad n,k>0,
	\end{align*}
	together with the boundary conditions
	\begin{align*}
	\stirnd{0}{0} &= 1,\\
	\stirnd{0}{k} &= 0, \quad k > 0,\\
	\stirnd{n}{0} &= 0, \quad n > 0.
	\end{align*}
	\label{defi:stirling_2}
	\hfill$ \square $
\end{defi}
\subsection{Finite coupling order}
\label{subsec:finite_order}
We denote by $ \hat{\mathcal{F}}_i^{<\infty} $ the subset of elements in $ \hat{\mathcal{F}}_i $ whose set of coupling components $ \{f_i^{\mathbf{k}}\}_{\mathbf{k} \geq \mathbf{0}_{\vert\tset\vert}} $ has only finitely many non-zero terms. From \cref{lemma:coupling_additivity}, this forms a subspace. We now show that we can represent the elements of $ \hat{\mathcal{F}}_i^{<\infty} $ by a set of functions $ \{\leftidx{^b}{}f_i^{\mathbf{k}}\}_{\mathbf{k} \geq \mathbf{0}_{\vert\tset\vert}} $, called \textbf{basis components}, which have simpler properties than the coupling components $ \{f_i^{\mathbf{k}}\}_{\mathbf{k} \geq \mathbf{0}_{\vert\tset\vert}} $. In particular, they are decoupled from one another. These functions have the following structure.
\begin{defi}
	\label{defi:decomp_pure_basis}
	A \textbf{basis component} $ \leftidx{^b}{}f_i^{\mathbf{k}} $, with $ \mathbf{k}\geq \mathbf{0}_{\vert\tset\vert} $, is a function defined on
	\begin{align}
	\leftidx{^b}{}f_i^{\mathbf{k}} \colon
	\xset_i\times
	\mathcal{M}_i^{\mathbf{k}} \times \xset^{\mathbf{k}}
	\to \yset_i, 
	\end{align}
	such that:
	\begin{enumerate}
		\item \label{thm:basis_f_permutation_invariance}If $ \permut $ is any \textbf{permutation} matrix (of appropriate dimension), then
		\begin{align}
		\leftidx{^b}{}f_i^{\mathbf{k}} \left(x;\mathbf{w},\mathbf{x}\right)
		=
		\leftidx{^b}{}f_i^{\mathbf{k}} \left(x;\permut\mathbf{w},\permut\mathbf{x}\right).
		\label{eq:basis_perm_inv}
		\end{align}
		\item \label{thm:basis_f_additive} If $ k_j > 0 $, then $ \leftidx{^b}{}f_i^{\mathbf{k}} $ is \textbf{additive in the weights} with respect to type $ j $. That is,
		\begin{align}
		\leftidx{^b}{}f_i^{\mathbf{k}}
		\left(x; 
		\begin{bmatrix}
		w_{j_1} \| w_{j_2}\\ 
		\mathbf{w}_{\overline{\mathbf{s}}}
		\end{bmatrix},
		\begin{bmatrix}
		x_{j_{12}} \\ 
		\mathbf{x}_{\overline{\mathbf{s}}}
		\end{bmatrix}
		\right)
		=
		\leftidx{^b}{}f_i^{\mathbf{k}} 
		\left(x; 
		\begin{bmatrix}
		w_{j_1} \\ 
		\mathbf{w}_{\overline{\mathbf{s}}}
		\end{bmatrix},
		\begin{bmatrix}
		x_{j_{12}} \\ 
		\mathbf{x}_{\overline{\mathbf{s}}}
		\end{bmatrix}
		\right)
		+
		\leftidx{^b}{}f_i^{\mathbf{k}} 
		\left(x; 
		\begin{bmatrix}
		w_{j_2}\\ 
		\mathbf{w}_{\overline{\mathbf{s}}}
		\end{bmatrix},
		\begin{bmatrix}
		x_{j_{12}} \\ 
		\mathbf{x}_{\overline{\mathbf{s}}}
		\end{bmatrix}
		\right).
		\end{align} 
	\end{enumerate}
	\hfill$ \square $
\end{defi}
\begin{corollary}\label{thm:basis_f_zero_kill}
	Given a basis component $ \leftidx{^b}{}f_i^{\mathbf{k}} $, if \textbf{any} of the entries of $ \mathbf{w} $ is $ 0_{ij} $ for some $ j\in\tset $, then
	\begin{align}
	\leftidx{^b}{}f_i^{\mathbf{k}} \left(x;\mathbf{w},\mathbf{x}\right)
	=
	0_{\yset_i}.
	\end{align}
	\hfill$ \square $
\end{corollary}
\begin{proof}
From \cref{thm:basis_f_additive} of \cref{defi:decomp_pure_basis}, we know that
\begin{align*}
\leftidx{^b}{}f_i^{\mathbf{k}}
\left(x; 
\begin{bmatrix}
w_{j_1} \| 0_{ij}\\ 
\mathbf{w}_{\overline{\mathbf{s}}}
\end{bmatrix},
\begin{bmatrix}
x_{j_{12}} \\ 
\mathbf{x}_{\overline{\mathbf{s}}}
\end{bmatrix}
\right)
=
\leftidx{^b}{}f_i^{\mathbf{k}} 
\left(x; 
\begin{bmatrix}
w_{j_1} \\ 
\mathbf{w}_{\overline{\mathbf{s}}}
\end{bmatrix},
\begin{bmatrix}
x_{j_{12}} \\ 
\mathbf{x}_{\overline{\mathbf{s}}}
\end{bmatrix}
\right)
+
\leftidx{^b}{}f_i^{\mathbf{k}} 
\left(x; 
\begin{bmatrix}
0_{ij}\\ 
\mathbf{w}_{\overline{\mathbf{s}}}
\end{bmatrix},
\begin{bmatrix}
x_{j_{12}} \\ 
\mathbf{x}_{\overline{\mathbf{s}}}
\end{bmatrix}
\right).
\end{align*}
Since $ w_{j_1} \| 0_{ij} = w_{j_1} $, this implies that
\begin{align*}
\leftidx{^b}{}f_i^{\mathbf{k}} 
\left(x; 
\begin{bmatrix}
	0_{ij}\\ 
	\mathbf{w}_{\overline{\mathbf{s}}}
\end{bmatrix},
\begin{bmatrix}
	x_{j_{12}} \\ 
	\mathbf{x}_{\overline{\mathbf{s}}}
\end{bmatrix}
\right)
=
0_{\yset_i}.
\end{align*}
The fact that this applies to every $ j\in\tset $, together with \cref{thm:basis_f_permutation_invariance} of \cref{defi:decomp_pure_basis}, proves the result for a zero in any entry of $ \mathbf{w} $.
\end{proof}
The following result assigns the appropriate basis components $ \{\leftidx{^b}{}f_i^{\mathbf{k}}\}_{\mathbf{k} \geq \mathbf{0}_{\vert\tset\vert}} $ to the elements of $ \hat{\mathcal{F}}_i^{<\infty} $ by relating them, bijectively, to the coupling components $ \{f_i^{\mathbf{k}}\}_{\mathbf{k} \geq \mathbf{0}_{\vert\tset\vert}} $.
We now use the following shorthand notation
\begin{align*}
f_i^{\mathcal{K}(\mathbf{m}\mathbf{s})}
\left(x;
\mathbf{m},
\mathbf{w}_{\mathbf{s}},
\mathbf{x}_{\mathbf{s}}
\right)
:=
f_i^{\mathcal{K}(\mathbf{m}\mathbf{s})}
\left(x;
\mathbf{m}\mathbf{w}_{\mathbf{s}},
\mathbf{m}\mathbf{x}_{\mathbf{s}}
\right).
\end{align*}
\begin{theorem}
	Assuming the related set $ \yset_i $ to be a vector space, there is a bijection between the set of coupling components $ \{f_i^{\mathbf{k}}\}_{\mathbf{k} \geq \mathbf{0}_{\vert\tset\vert}} $ of elements in $ \hat{\mathcal{F}}_i^{<\infty} $, and the set of basis components $ \{\leftidx{^b}{}f_i^{\mathbf{k}}\}_{\mathbf{k} \geq \mathbf{0}_{\vert\tset\vert}} $ with finitely many non-zero terms. In particular, this bijection is given by the following equivalent expressions,
	\begin{align}
	f_i^{\mathcal{K}(\mathbf{s})}\left(x;\mathbf{w}_{\mathbf{s}},\mathbf{x}_{\mathbf{s}}\right)
	= 
	\sum_{\mathbf{m} \geq \mathbf{1}_{\vert \mathbf{s} \vert}}
	\frac{1}{\prod_{c\in\mathbf{s}} m_c!}
	\leftidx{^b}{}f_i^{\mathcal{K}(\mathbf{m}\mathbf{s})}
	\left(x;
	\mathbf{m},
	\mathbf{w}_{\mathbf{s}},
	\mathbf{x}_{\mathbf{s}}
	\right),
	\label{eq:component_from_basis1_inf}
	\end{align}
	\begin{align}
	\leftidx{^b}{}f_i^{\mathcal{K}(\mathbf{s})}\left(x;\mathbf{w}_{\mathbf{s}},\mathbf{x}_{\mathbf{s}}\right)
	= 
	\sum_{\mathbf{m} \geq \mathbf{1}_{\vert \mathbf{s} \vert}}
	\frac{(-1)^{\vert\mathbf{m}\vert-\vert\mathbf{s}\vert}}{\prod_{c\in\mathbf{s}} m_c}
	f_i^{\mathcal{K}(\mathbf{m}\mathbf{s})}
	\left(x;
	\mathbf{m},
	\mathbf{w}_{\mathbf{s}},
	\mathbf{x}_{\mathbf{s}}
	\right),
	\label{eq:basis_from_component1_inf}
	\end{align}
	and in general for multiplicities $ \mathbf{m}\geq \mathbf{0}_{\vert \mathbf{s} \vert} $,
	\begin{align}
	f_i^{\mathcal{K}(\mathbf{m}\mathbf{s})}
	\left(x;
	\mathbf{m},
	\mathbf{w}_{\mathbf{s}},
	\mathbf{x}_{\mathbf{s}}
	\right)
	=
	\sum_{\mathbf{M} \geq \mathbf{0}_{\vert \mathbf{s}\vert}}
	\left(
	\prod_{c\in\mathbf{s}}
	\frac{m_c!}{ M_c!} \stirnd{M_c}{m_c}
	\right)
	\leftidx{^b}{}f_i^{\mathcal{K}(\mathbf{M}\mathbf{s})}
	\left(x;
	\mathbf{M},
	\mathbf{w}_{\mathbf{s}},
	\mathbf{x}_{\mathbf{s}}
	\right),
	\label{eq:component_from_basis1_inf_multi}
	\end{align}
	\begin{align}
	\leftidx{^b}{}f_i^{\mathcal{K}(\mathbf{m}\mathbf{s})}
	\left(x;
	\mathbf{m},
	\mathbf{w}_{\mathbf{s}},
	\mathbf{x}_{\mathbf{s}}
	\right)
	= 
	\sum_{\mathbf{M} \geq \mathbf{0}_{\vert \mathbf{s}\vert}}
	(-1)^{\vert\mathbf{M}\vert-\vert\mathbf{m}\vert}
	\left(
	\prod_{c\in\mathbf{s}}
	\frac{m_c!}{ M_c!} \stirst{M_c}{m_c}
	\right)
	f_i^{\mathcal{K}(\mathbf{M}\mathbf{s})}
	\left(x;
	\mathbf{M},
	\mathbf{w}_{\mathbf{s}},
	\mathbf{x}_{\mathbf{s}}
	\right).
	\label{eq:basis_from_component1_inf_multi}
	\end{align}
	\label{thm:decomposition_finite}
	\hfill$ \square $
\end{theorem}
In order to prove this, we require \cref{lemma:stirling_sums_multi_2,lemma:stirling_sums_alternated_multi_2,lemma:stirling_sums_multi_1,lemma:stirling_sums_alternated_multi_1,lemma:basis_comp_multi_expansion,lemma:coupl_comp_multi_expansion,lemma:comb_sum_coupling_to_basis_major}, which are proven in \cref{apend:decomposition_finite_intermediate_results}.
	\begin{lemma}
	For $ \mathbf{m},\mathbf{M} \geq \mathbf{0}_{k} $, with $ k\geq 0 $, we have that
	\begin{align}
	\sum_{ 
		\substack{
			\overline{\mathbf{m}} \geq \mathbf{1}_{\vert \mathbf{m} \vert}
			\\
			\overline{\mathbf{m}}\mathbf{m} = \mathbf{M}
		}
	}
	\frac{1}{
		\prod_{i=1}^{\vert \mathbf{m} \vert} \overline{m}_i 
	} 
	=
	\prod_{i=1}^{k}
	\frac{m_i!}{M_i!}
	\stirst{M_i}{m_i}.
	\label{eq:stirling_sums_multi_1}
	\end{align}
	\label{lemma:stirling_sums_multi_1}
	\hfill$ \square $
\end{lemma}
	\begin{lemma}
	For $ \mathbf{m},\mathbf{M} \geq \mathbf{0}_{k} $, with $ k\geq 0 $, we have that
	\begin{align}
	\sum_{ 
		\substack{
			\overline{\mathbf{m}} \geq \mathbf{1}_{\vert \mathbf{m} \vert}
			\\
			\overline{\mathbf{m}}\mathbf{m} = \mathbf{M}
		}
	}
	\frac{1}{
		\prod_{i=1}^{\vert \mathbf{m} \vert} \overline{m}_i !
	} 
	=
	\prod_{i=1}^{k}
	\frac{m_i!}{M_i!}
	\stirnd{M_i}{m_i}.
	\label{eq:stirling_sums_multi_2}
	\end{align}
	\label{lemma:stirling_sums_multi_2}
	\hfill$ \square $
\end{lemma}
	\begin{lemma}
	For $ \mathbf{M} \geq \mathbf{0}_{k} $, with $ k\geq 0 $, we have that
	\begin{align}
	\sum_{\mathbf{m} \geq \mathbf{1}_{k}}
	\prod_{i=1}^{k}
	(-1)^{m_i}
	\stirst{M_i}{m_i}
	=
	\begin{cases}
	(-1)^{k} & \text{if $ \mathbf{M} = \mathbf{1}_{k} $},\\
	0 & \text{otherwise}.
	\end{cases} 
	\label{eq:stirling_sums_alternated_multi_1}
	\end{align}
	\label{lemma:stirling_sums_alternated_multi_1}
	\hfill$ \square $
\end{lemma}
	\begin{lemma}
	For $ \mathbf{M} \geq \mathbf{0}_{k} $, with $ k\geq 0 $, we have that
	\begin{align}
	\sum_{\mathbf{m} \geq \mathbf{1}_{k}}
	\prod_{i=1}^{k}
	(-1)^{m_i}
	(m_i -1)!
	\stirnd{M_i}{m_i}
	=
	\begin{cases}
	(-1)^{k} & \text{if $ \mathbf{M} = \mathbf{1}_{k} $},\\
	0 & \text{otherwise}.
	\end{cases} 
	\label{eq:stirling_sums_alternated_multi_2}
	\end{align}
	\label{lemma:stirling_sums_alternated_multi_2}
	\hfill$ \square $
\end{lemma}
\begin{lemma}
	\label{lemma:basis_comp_multi_expansion}
	Consider a function $ \leftidx{^b}{}f_i^{\mathbf{k}} $ with the properties in \cref{defi:decomp_pure_basis}, for some $ \mathbf{k}\geq \mathbf{0}_{\vert\tset\vert} $. For every $ m_{12} \geq 0, \overline{\mathbf{m}} \geq  \mathbf{0}_{\vert \overline{\mathbf{s}} \vert}$, such that $ \mathbf{k} = \mathcal{K}(\overline{\mathbf{m}}\overline{\mathbf{s}})+m_{12} \, 1_j $, we have that
	\begin{align}
	\label{eq:basis_component_multiplicity_expansion}
	\leftidx{^b}{}f_i^{\mathbf{k}}
	\left(x;
	\begin{bmatrix}
	m_{12}\\
	\overline{\mathbf{m}}
	\end{bmatrix},
	\begin{bmatrix}
	w_{j_1} \| w_{j_2}\\ 
	\mathbf{w}_{\overline{\mathbf{s}}}
	\end{bmatrix},
	\begin{bmatrix}
	x_{j_{12}} \\
	\mathbf{x}_{\overline{\mathbf{s}}}
	\end{bmatrix}
	\right)
	=
	\sum_{
		\substack{
			m_1,m_2 \geq 0\\
			m_1 + m_2 = m_{12}	
		}
	}
	\binom{m_{12}}{m_1,m_2}
	\leftidx{^b}{}		f_i^{\mathbf{k}}
	\left(x;
	\begin{bmatrix}
	m_{1}\\
	m_{2}\\
	\overline{\mathbf{m}}
	\end{bmatrix}
	,
	\begin{bmatrix}
	w_{j_1}\\
	w_{j_2}\\ 
	\mathbf{w}_{\overline{\mathbf{s}}}
	\end{bmatrix}
	,
	\begin{bmatrix}
	x_{j_{12}} \\ 
	x_{j_{12}} \\
	\mathbf{x}_{\overline{\mathbf{s}}}
	\end{bmatrix}
	\right).
	\end{align}
	\hfill$ \square $
\end{lemma}
	\begin{lemma}\label{lemma:coupl_comp_multi_expansion}
	Consider a family of functions $ \{f_i^{\mathbf{k}}\}_{\mathbf{k} \geq\mathbf{0}_{\vert\tset\vert}} $ with the properties in \cref{thm:decomp_f}, for some $ \mathbf{k}\geq \mathbf{0}_{\vert\tset\vert} $. For every $ m_{12} \geq 0, \overline{\mathbf{m}} \geq  \mathbf{0}_{\vert \overline{\mathbf{s}} \vert}$, such that $ \overline{\mathbf{k}} = \mathcal{K}(\overline{\mathbf{m}}\overline{\mathbf{s}}) $, we have that
	\begin{align}
	\label{eq:coupl_component_multiplicity_expansion}
	&f_i^{
		\overline{\mathbf{k}}
		+ m_{12} \, 1_j
	}
	\left(x;
	\begin{bmatrix}
	m_{12}\\
	\overline{\mathbf{m}}
	\end{bmatrix}
	,
	\begin{bmatrix}
	w_{j_1} \| w_{j_2}\\ 
	\mathbf{w}_{\overline{\mathbf{s}}}
	\end{bmatrix}
	,
	\begin{bmatrix}
	x_{j_{12}} \\
	\mathbf{x}_{\overline{\mathbf{s}}}
	\end{bmatrix}
	\right)
	\\
	&\quad=
	\sum_{ 
		\substack{
			m_1,m_2 \geq 0\\
			m_1,m_2 \leq m_{12}\\
			m_1 + m_2 \geq m_{12}
		}
	}
	B(m_1,m_2,m_{12}) 
	f_i^{
		\overline{\mathbf{k}}
		+ (m_{1}+m_{2}) \, 1_j 
	}
	\left(x;
	\begin{bmatrix}
	m_{1}\\
	m_{2}\\
	\overline{\mathbf{m}}
	\end{bmatrix}
	,
	\begin{bmatrix}
	w_{j_1} \\ w_{j_2}\\ 
	\mathbf{w}_{\overline{\mathbf{s}}}
	\end{bmatrix}
	,
	\begin{bmatrix}
	x_{j_{12}} \\ 
	x_{j_{12}} \\
	\mathbf{x}_{\overline{\mathbf{s}}}
	\end{bmatrix}
	\right),
	\notag
	\end{align}
	where $ B(m_1,m_2,m_{12}) $ is defined as
	\begin{align*}
	B(m_1,m_2,m_{12})
	&:=
	\binom{m_{12}}{m_{12} - m_1,m_{12} - m_2,m_1 + m_2 - m_{12}}.
	\end{align*}
	\hfill$ \square $
\end{lemma}
	\begin{lemma}
	For every $ m_1,m_2 \in \mathbb{N}_0 $, we have that
	\begin{align}
	\sum_{
		\substack{
			n \geq 1, m_1, m_2\\
			n \leq m_1 + m_2
		}
	}
	\frac{(-1)^n}{n}  \binom{n}{n-m_1,n-m_2, m_1+m_2-n}
	=
	\begin{cases}
	\frac{(-1)^{m_1}}{m_1} & \text{if $m_1\geq 1, m_2=0$},\\
	\frac{(-1)^{m_2}}{m_2} & \text{if $m_1 = 0, m_2\geq 1$},\\
	0 & \text{otherwise}.
	\end{cases}
	\label{eq:comb_sum_coupling_to_basis_major}
	\end{align}
	\label{lemma:comb_sum_coupling_to_basis_major}.
	\hfill$ \square $
\end{lemma}
\begin{proof}[Proof of \cref{thm:decomposition_finite}]
Firstly, we prove that if both $ \{f_i^{\mathbf{k}}\}_{\mathbf{k} \geq \mathbf{0}_{\vert\tset\vert}} $ and $ \{\leftidx{^b}{}f_i^{\mathbf{k}}\}_{\mathbf{k} \geq \mathbf{0}_{\vert\tset\vert}} $ have finitely many non-zero terms, then, \cref{eq:component_from_basis1_inf,eq:basis_from_component1_inf,eq:component_from_basis1_inf_multi,eq:basis_from_component1_inf_multi} are all equivalent. Note that the assumption implies that all the sums indexed at $ \mathbf{m} \geq \mathbf{1}_{\vert \mathbf{s} \vert} $ and $ \mathbf{M} \geq \mathbf{0}_{\vert \mathbf{s}\vert} $ have finite support. That is, they are actually finite sums in disguise and there are no convergence issues.\\
We now prove the equivalence of \cref{eq:component_from_basis1_inf,eq:basis_from_component1_inf,eq:component_from_basis1_inf_multi,eq:basis_from_component1_inf_multi} by proving the cycle of implications \cref{eq:component_from_basis1_inf} $ \implies $ \cref{eq:component_from_basis1_inf_multi} $ \implies $ \cref{eq:basis_from_component1_inf} $ \implies $ \cref{eq:basis_from_component1_inf_multi} $ \implies $ \cref{eq:component_from_basis1_inf}.\\
Assume \cref{eq:component_from_basis1_inf}.
Direct substitution gives us
\begin{align*}
f_i^{\mathcal{K}(\mathbf{m}\mathbf{s})}\left(x;
\mathbf{m},
\mathbf{w}_{\mathbf{s}},
\mathbf{x}_{\mathbf{s}}
\right)
&= 
\sum_{\overline{\mathbf{m}} \geq \mathbf{1}_{\vert \mathbf{m} \vert}}
\frac{1}{\prod_{i=1}^{\vert \mathbf{m} \vert} \overline{m}_i!}
\leftidx{^b}{}f_i^{\mathcal{K}(\overline{\mathbf{m}}\mathbf{m}\mathbf{s})}
\left(x;
\overline{\mathbf{m}}\mathbf{m},
\mathbf{w}_{\mathbf{s}},
\mathbf{x}_{\mathbf{s}}
\right).
\end{align*}
Since we are dealing with finite sums, we can freely reorder the terms such that we merge together the pairs $ (\overline{\mathbf{m}},\mathbf{m}) $ such that $ \overline{\mathbf{m}}\mathbf{m} = \mathbf{M} $. That is,
\begin{align*}
f_i^{\mathcal{K}(\mathbf{m}\mathbf{s})}\left(x;
\mathbf{m},
\mathbf{w}_{\mathbf{s}},
\mathbf{x}_{\mathbf{s}}
\right)
&= 
\sum_{\mathbf{M} \geq \mathbf{0}_{\vert \mathbf{s}\vert}}
\sum_{ 
	\substack{
		\overline{\mathbf{m}} \geq \mathbf{1}_{\vert \mathbf{m} \vert}
		\\
		\overline{\mathbf{m}}\mathbf{m} = \mathbf{M}
	}
}
\frac{1}{\prod_{i=1}^{\vert \mathbf{m} \vert} \overline{m}_i!}
\leftidx{^b}{}f_i^{\mathcal{K}(\mathbf{M}\mathbf{s})}
\left(x;
\mathbf{M},
\mathbf{w}_{\mathbf{s}},
\mathbf{x}_{\mathbf{s}}
\right),
\end{align*}
which from \cref{lemma:stirling_sums_multi_2} simplifies into \cref{eq:component_from_basis1_inf_multi}. Therefore, \cref{eq:component_from_basis1_inf} $ \implies $ \cref{eq:component_from_basis1_inf_multi}.\\
We now assume \cref{eq:component_from_basis1_inf_multi}. Using this on the right hand side of \cref{eq:basis_from_component1_inf} we get
\begin{align*}
&\sum_{\mathbf{m} \geq \mathbf{1}_{\vert \mathbf{s} \vert}}
\frac{(-1)^{\vert\mathbf{m}\vert-\vert\mathbf{s}\vert}}{\prod_{c\in\mathbf{s}} m_c}
f_i^{\mathcal{K}(\mathbf{m}\mathbf{s})}
\left(x;
\mathbf{m},
\mathbf{w}_{\mathbf{s}},
\mathbf{x}_{\mathbf{s}}
\right)
\\
&=
\sum_{\mathbf{m} \geq \mathbf{1}_{\vert \mathbf{s} \vert}}
\frac{(-1)^{\vert\mathbf{m}\vert-\vert\mathbf{s}\vert}}{\prod_{c\in\mathbf{s}} m_c}
\sum_{\mathbf{M} \geq \mathbf{0}_{\vert \mathbf{s}\vert}}
\left(
\prod_{c\in\mathbf{s}}
\frac{m_c!}{ M_c!} \stirnd{M_c}{m_c}
\right)
\leftidx{^b}{}f_i^{\mathcal{K}(\mathbf{M}\mathbf{s})}
\left(x;
\mathbf{M},
\mathbf{w}_{\mathbf{s}},
\mathbf{x}_{\mathbf{s}}
\right).
\end{align*}
Exchanging the order of the two sums and simplifying we get
\begin{align*}
\sum_{\mathbf{M} \geq \mathbf{0}_{\vert \mathbf{s}\vert}}
\frac{(-1)^{\vert\mathbf{s}\vert}}{\prod_{c\in\mathbf{s}} M_c!}
\left(
\sum_{\mathbf{m} \geq \mathbf{1}_{\vert \mathbf{s} \vert}}
\prod_{c\in\mathbf{s}}
(-1)^{m_c}
(m_c-1)! \stirnd{M_c}{m_c}
\right)
\leftidx{^b}{}f_i^{\mathcal{K}(\mathbf{M}\mathbf{s})}
\left(x;
\mathbf{M},
\mathbf{w}_{\mathbf{s}},
\mathbf{x}_{\mathbf{s}}
\right),
\end{align*}
which from \cref{lemma:stirling_sums_alternated_multi_2} simplifies into $ \leftidx{^b}{}f_i^{\mathcal{K}(\mathbf{s})}\left(x;\mathbf{w}_{\mathbf{s}},\mathbf{x}_{\mathbf{s}}\right) $, the left hand side of \cref{eq:basis_from_component1_inf}. Therefore, \cref{eq:component_from_basis1_inf_multi} $ \implies $ \cref{eq:basis_from_component1_inf}.\\
We now assume \cref{eq:basis_from_component1_inf}. Direct substitution gives us
\begin{align*}
\leftidx{^b}{}f_i^{\mathcal{K}(\mathbf{m}\mathbf{s})}\left(x;
\mathbf{m},
\mathbf{w}_{\mathbf{s}},
\mathbf{x}_{\mathbf{s}}
\right)
= 
\sum_{\overline{\mathbf{m}} \geq \mathbf{1}_{\vert \mathbf{m} \vert}}
\frac{(-1)^{\vert\overline{\mathbf{m}}\vert-\vert\mathbf{m}\vert}}{\prod_{i=1}^{\vert \mathbf{m} \vert} \overline{m}_i}
f_i^{\mathcal{K}(\overline{\mathbf{m}}\mathbf{m}\mathbf{s})}
\left(x;
\overline{\mathbf{m}}\mathbf{m},
\mathbf{w}_{\mathbf{s}},
\mathbf{x}_{\mathbf{s}}
\right).
\end{align*}
Merging together the pairs $ (\overline{\mathbf{m}},\mathbf{m}) $ such that $ \overline{\mathbf{m}}\mathbf{m} = \mathbf{M} $, we obtain
\begin{align*}
\leftidx{^b}{}f_i^{\mathcal{K}(\mathbf{m}\mathbf{s})}
\left(x;
\mathbf{m},
\mathbf{w}_{\mathbf{s}},
\mathbf{x}_{\mathbf{s}}
\right)
=
\sum_{\mathbf{M} \geq \mathbf{0}_{\vert \mathbf{s}\vert}}
(-1)^{\vert\mathbf{M}\vert-\vert\mathbf{m}\vert}
\sum_{ 
	\substack{
		\overline{\mathbf{m}} \geq \mathbf{1}_{\vert \mathbf{m} \vert}
		\\
		\overline{\mathbf{m}}\mathbf{m} = \mathbf{M}
	}
}
\frac{1}{\prod_{i=1}^{\vert \mathbf{m} \vert} \overline{m}_i}
f_i^{\mathcal{K}(\mathbf{M}\mathbf{s})}
\left(x;
\mathbf{M},
\mathbf{w}_{\mathbf{s}},
\mathbf{x}_{\mathbf{s}}
\right).
\end{align*}
Note that $ \vert \overline{\mathbf{m}} \vert = \vert \mathbf{M} \vert $. From \cref{lemma:stirling_sums_multi_1}, this simplifies into \cref{eq:basis_from_component1_inf_multi}. Therefore, \cref{eq:basis_from_component1_inf} $ \implies $ \cref{eq:basis_from_component1_inf_multi}.\\
Finally, we assume \cref{eq:basis_from_component1_inf_multi}. Using this on the right hand side of \cref{eq:component_from_basis1_inf} we get
\begin{align*}
&\sum_{\mathbf{m} \geq \mathbf{1}_{\vert \mathbf{s} \vert}}
\frac{1}{\prod_{c\in\mathbf{s}} m_c!}
\leftidx{^b}{}f_i^{\mathcal{K}(\mathbf{m}\mathbf{s})}
\left(x;
\mathbf{m},
\mathbf{w}_{\mathbf{s}},
\mathbf{x}_{\mathbf{s}}
\right)
\\
&=
\sum_{\mathbf{m} \geq \mathbf{1}_{\vert \mathbf{s} \vert}}
\frac{1}{\prod_{c\in\mathbf{s}} m_c!}
	\sum_{\mathbf{M} \geq \mathbf{0}_{\vert \mathbf{s}\vert}}
(-1)^{\vert\mathbf{M}\vert-\vert\mathbf{m}\vert}
\left(
\prod_{c\in\mathbf{s}}
\frac{m_c!}{ M_c!} \stirst{M_c}{m_c}
\right)
f_i^{\mathcal{K}(\mathbf{M}\mathbf{s})}
\left(x;
\mathbf{M},
\mathbf{w}_{\mathbf{s}},
\mathbf{x}_{\mathbf{s}}
\right).
\end{align*}
Exchanging the order of the two sums and simplifying we get
\begin{align*}
\sum_{\mathbf{M} \geq \mathbf{0}_{\vert \mathbf{s}\vert}}
\frac{(-1)^{\vert\mathbf{M}\vert}}{\prod_{c\in\mathbf{s}} M_c! }
\left(
\sum_{\mathbf{m} \geq \mathbf{1}_{\vert \mathbf{s} \vert}}
\prod_{c\in\mathbf{s}}
(-1)^{m_c} \stirst{M_c}{m_c}
\right)
f_i^{\mathcal{K}(\mathbf{M}\mathbf{s})}
\left(x;
\mathbf{M},
\mathbf{w}_{\mathbf{s}},
\mathbf{x}_{\mathbf{s}}
\right),
\end{align*}
which from \cref{lemma:stirling_sums_alternated_multi_1} simplifies into \cref{eq:component_from_basis1_inf}. This completes the proof that \cref{eq:component_from_basis1_inf,eq:basis_from_component1_inf,eq:component_from_basis1_inf_multi,eq:basis_from_component1_inf_multi} are equivalent under the assumption that both $ \{f_i^{\mathbf{k}}\}_{\mathbf{k} \geq \mathbf{0}_{\vert\tset\vert}} $ and $ \{\leftidx{^b}{}f_i^{\mathbf{k}}\}_{\mathbf{k} \geq \mathbf{0}_{\vert\tset\vert}} $ have finitely many non-zero terms. We now weaken this assumption by showing that one of them having finitely many non-zero terms implies the other also having that property.\\
Assume $ \{f_i^{\mathbf{k}}\}_{\mathbf{k} \geq \mathbf{0}_{\vert\tset\vert}} $ has finitely many terms. Then, there is some $ \mathbf{K} \geq \mathbf{0}_{\vert\tset\vert} $ such that all non-zero terms are inside the subset $ \{f_i^{\mathbf{k}}\}_{\mathbf{k} \leq \mathbf{K}} $. Note that the sums \cref{eq:basis_from_component1_inf} are always finite, which means that the corresponding $ \{\leftidx{^b}{}f_i^{\mathbf{k}}\}_{\mathbf{k} \geq \mathbf{0}_{\vert\tset\vert}} $ is well-defined. Furthermore, every $ \leftidx{^b}{}f_i^{\mathbf{k}} $ such that $ \mathbf{k} $ does not obey $ \mathbf{k} \leq \mathbf{K} $, is given by a sum of zero terms. Therefore, $ \{\leftidx{^b}{}f_i^{\mathbf{k}}\}_{\mathbf{k} \geq \mathbf{0}_{\vert\tset\vert}} $ also has all of its non-zero terms inside the subset 
$ \{\leftidx{^b}{}f_i^{\mathbf{k}}\}_{\mathbf{k} \leq \mathbf{K}} $, which means that it also has finitely many non-zero terms. Then, the previous assumptions are satisfied and consequently \cref{eq:component_from_basis1_inf,eq:basis_from_component1_inf,eq:component_from_basis1_inf_multi,eq:basis_from_component1_inf_multi} are equivalent.
\\
The exact same argument applies when starting with some $ \{\leftidx{^b}{}f_i^{\mathbf{k}}\}_{\mathbf{k} \geq \mathbf{0}_{\vert\tset\vert}} $ that has finitely many non-zero terms and constructing the corresponding $ \{f_i^{\mathbf{k}}\}_{\mathbf{k} \geq \mathbf{0}_{\vert\tset\vert}} $ through \cref{eq:component_from_basis1_inf}.\\
We now prove that $ \{f_i^{\mathbf{k}}\}_{\mathbf{k} \geq \mathbf{0}_{\vert\tset\vert}} $ has the properties in \cref{thm:decomp_f} if and only if the corresponding $ \{\leftidx{^b}{}f_i^{\mathbf{k}}\}_{\mathbf{k} \geq \mathbf{0}_{\vert\tset\vert}} $ has the properties in \cref{defi:decomp_pure_basis}.\\
Assume some $ \{\leftidx{^b}{}f_i^{\mathbf{k}}\}_{\mathbf{k} \geq \mathbf{0}_{\vert\tset\vert}} $ has the properties in \cref{defi:decomp_pure_basis}. Then, from \cref{eq:component_from_basis1_inf}, we have that for any permutation matrix $ \permut $,
\begin{align*}
f_i^{\mathcal{K}(\mathbf{s})}\left(x;\mathbf{w}_{\mathbf{s}},\mathbf{x}_{\mathbf{s}}\right)
&=
\sum_{\mathbf{m} \geq \mathbf{1}_{\vert \mathbf{s} \vert}}
\frac{1}{\prod_{c\in\mathbf{s}} m_c!}
\leftidx{^b}{}f_i^{\mathcal{K}(\mathbf{m}\mathbf{s})}
\left(x;
\mathbf{m},   
\mathbf{w}_{\mathbf{s}},
\mathbf{x}_{\mathbf{s}}
\right)
\\
&=
\sum_{\mathbf{m} \geq \mathbf{1}_{\vert \mathbf{s} \vert}}
\frac{1}{\prod_{c\in\mathbf{s}} m_c!}
\leftidx{^b}{}f_i^{\mathcal{K}(\left[\permut\mathbf{m}\right]\permut\mathbf{s})}
\left(x;
\permut\mathbf{m},
\permut\mathbf{w}_{\mathbf{s}},
\permut\mathbf{x}_{\mathbf{s}}
\right)
\\
&= 
\sum_{\overline{\mathbf{m}} \geq \mathbf{1}_{\vert \mathbf{s} \vert}}
\frac{1}{\prod_{c\in\mathbf{s}} \overline{m}_c!}
\leftidx{^b}{}f_i^{\mathcal{K}(\overline{\mathbf{m}}\permut\mathbf{s})}
\left(x;
\overline{\mathbf{m}},
\permut\mathbf{w}_{\mathbf{s}},
\permut\mathbf{x}_{\mathbf{s}}
\right)
\\
&=
f_i^{\mathcal{K}(\mathbf{s})}\left(x;\permut\mathbf{w}_{\mathbf{s}},\permut\mathbf{x}_{\mathbf{s}}\right),
\end{align*}
where $ \overline{\mathbf{m}} = \permut \mathbf{m} $ establishes a bijection between the sets of indexes $ \mathbf{m} \geq \mathbf{1}_{\vert \mathbf{s} \vert} $ and $ \overline{\mathbf{m}} \geq \mathbf{1}_{\vert \mathbf{s} \vert} $. Therefore, $ \{f_i^{\mathbf{k}}\}_{\mathbf{k} \geq \mathbf{0}_{\vert\tset\vert}} $ satisfies \cref{thm:decomp_f_permutation_invariance} of \cref{thm:decomp_f}.\\
Again from \cref{eq:component_from_basis1_inf}, we have that
\begin{align*}
f_i^{\mathcal{K}(\mathbf{s})}\left(x;
\begin{bmatrix}
w_{j_1} \| w_{j_2}\\ 
\mathbf{w}_{\overline{\mathbf{s}}}
\end{bmatrix},
\mathbf{x}_{\mathbf{s}}\right)
&= 
\sum_{\mathbf{m} \geq \mathbf{1}_{\vert \mathbf{s} \vert}}
\frac{1}{\prod_{c\in\mathbf{s}} m_c!}
\leftidx{^b}{}f_i^{\mathcal{K}(\mathbf{m}\mathbf{s})}
\left(x;
\mathbf{m},
\begin{bmatrix}
w_{j_1} \| w_{j_2}\\ 
\mathbf{w}_{\overline{\mathbf{s}}}
\end{bmatrix},
\mathbf{x}_{\mathbf{s}}
\right),
\end{align*}
with $ \mathbf{m}
=
\begin{bmatrix}
m_{12}\\
\overline{\mathbf{m}}
\end{bmatrix} $ and $ \mathbf{x}_{\mathbf{s}}
=
\begin{bmatrix}
x_{j_{12}} \\ 
\mathbf{x}_{\overline{\mathbf{s}}}
\end{bmatrix} $. Applying \cref{lemma:basis_comp_multi_expansion}, the right hand side expands into
\begin{align*}
\sum_{ 
	\substack{
		m_{12} \geq 1\\
		\overline{\mathbf{m}} \geq \mathbf{1}_{\vert \overline{\mathbf{s}} \vert}
	}
}
\frac{1}{m_{12}!\prod_{c\in\overline{\mathbf{s}}} \overline{m}_c!}
\sum_{
	\substack{
		m_1,m_2 \geq 0\\
		m_1 + m_2 = m_{12}	
	}
}
\frac{m_{12}!}{m_1! m_2!}
\leftidx{^b}{}f_i^{\mathbf{k}}
\left(x;
\begin{bmatrix}
m_{1}\\
m_{2}\\
\overline{\mathbf{m}}
\end{bmatrix}
,
\begin{bmatrix}
w_{j_1}\\
w_{j_2}\\ 
\mathbf{w}_{\overline{\mathbf{s}}}
\end{bmatrix}
,
\begin{bmatrix}
x_{j_{12}} \\ 
x_{j_{12}} \\
\mathbf{x}_{\overline{\mathbf{s}}}
\end{bmatrix}
\right).
\end{align*}
We cancel the $ m_{12}! $ terms and merge the two sums, which simplifies the expression into 
\begin{align*}
\sum_{ 
	\substack{
		m_1,m_2 \geq 0\\
		m_1+m_2 \geq 1\\
		\overline{\mathbf{m}} \geq \mathbf{1}_{\vert \overline{\mathbf{s}} \vert}
	}
}
\frac{1}{m_1! m_2!\prod_{c\in\overline{\mathbf{s}}} \overline{m}_c!}
\leftidx{^b}{}f_i^{\mathbf{k}}
\left(x;
\begin{bmatrix}
m_{1}\\
m_{2}\\
\overline{\mathbf{m}}
\end{bmatrix}
,
\begin{bmatrix}
w_{j_1}\\
w_{j_2}\\ 
\mathbf{w}_{\overline{\mathbf{s}}}
\end{bmatrix}
,
\begin{bmatrix}
x_{j_{12}} \\ 
x_{j_{12}} \\
\mathbf{x}_{\overline{\mathbf{s}}}
\end{bmatrix}
\right).
\end{align*}
We now split this sum into three parts. The first with $ m_1 \geq 1, m_2 =0 $, the second with $ m_1 = 0, m_2 \geq 1 $ and the third with $ m_1,m_2\geq 1 $. Applying \cref{eq:component_from_basis1_inf} again, gives us the three terms of the right hand side of \cref{thm:decomp_f_dependence} of \cref{thm:decomp_f}.\\
Finally, consider that some entry of $ \mathbf{w} $ is $ 0_{ij} $ for some $ j\in\tset $. Then, from \cref{thm:basis_f_zero_kill}, every term of the sum \cref{eq:component_from_basis1_inf} is zero, which means that $ \{f_i^{\mathbf{k}}\}_{\mathbf{k} \geq \mathbf{0}_{\vert\tset\vert}} $ satisfies \cref{thm:decomp_f_zero_kill} of \cref{thm:decomp_f}.\\
We now prove the converse direction. Assume some $ \{f_i^{\mathbf{k}}\}_{\mathbf{k} \geq \mathbf{0}_{\vert\tset\vert}} $ has the properties in \cref{thm:decomp_f}. Then, from \cref{eq:basis_from_component1_inf}, we have that for any permutation matrix $ \permut $,
\begin{align*}
\leftidx{^b}{}f_i^{\mathcal{K}(\mathbf{s})}\left(x;\mathbf{w}_{\mathbf{s}},\mathbf{x}_{\mathbf{s}}\right)
&= 
\sum_{\mathbf{m} \geq \mathbf{1}_{\vert \mathbf{s} \vert}}
\frac{(-1)^{\vert\mathbf{m}\vert-\vert\mathbf{s}\vert}}{\prod_{c\in\mathbf{s}} m_c}
f_i^{\mathcal{K}(\mathbf{m}\mathbf{s})}
\left(x;
\mathbf{m},
\mathbf{w}_{\mathbf{s}},
\mathbf{x}_{\mathbf{s}}
\right)
\\
&=
\sum_{\mathbf{m} \geq \mathbf{1}_{\vert \mathbf{s} \vert}}
\frac{(-1)^{\vert\mathbf{m}\vert-\vert\mathbf{s}\vert}}{\prod_{c\in\mathbf{s}} m_c}
f_i^{\mathcal{K}(\left[\permut\mathbf{m}\right]\permut\mathbf{s})}
\left(x;
\permut\mathbf{m},
\permut\mathbf{w}_{\mathbf{s}},
\permut\mathbf{x}_{\mathbf{s}}
\right)
\\
&= 
\sum_{\overline{\mathbf{m}} \geq \mathbf{1}_{\vert \mathbf{s} \vert}}
\frac{(-1)^{\vert\overline{\mathbf{m}}\vert-\vert\mathbf{s}\vert}}{\prod_{c\in\mathbf{s}} \overline{m}_c}
f_i^{\mathcal{K}(\overline{\mathbf{m}}\permut\mathbf{s})}
\left(x;
\overline{\mathbf{m}},
\permut\mathbf{w}_{\mathbf{s}},
\permut\mathbf{x}_{\mathbf{s}}
\right)
\\
&=
\leftidx{^b}{}f_i^{\mathcal{K}(\mathbf{s})}\left(x;\permut\mathbf{w}_{\mathbf{s}},\permut\mathbf{x}_{\mathbf{s}}\right),
\end{align*}
where $ \overline{\mathbf{m}} = \permut \mathbf{m} $ establishes a bijection between the sets of indexes $ \mathbf{m} \geq \mathbf{1}_{\vert \mathbf{s} \vert} $ and $ \overline{\mathbf{m}} \geq \mathbf{1}_{\vert \mathbf{s} \vert} $. Therefore, $ \{\leftidx{^b}{}f_i^{\mathbf{k}}\}_{\mathbf{k} \geq \mathbf{0}_{\vert\tset\vert}} $ satisfies \cref{thm:basis_f_permutation_invariance} of \cref{defi:decomp_pure_basis}.\\
Finally, we have that
\begin{align*}
\leftidx{^b}{}f_i^{\mathcal{K}(\mathbf{s})}
\left(x; 
\begin{bmatrix}
w_{j_1} \| w_{j_2}\\ 
\mathbf{w}_{\overline{\mathbf{s}}}
\end{bmatrix},
\mathbf{x}_{\mathbf{s}}
\right)
&=
\sum_{\mathbf{m} \geq \mathbf{1}_{\vert \mathbf{s} \vert}}
\frac{(-1)^{\vert\mathbf{m}\vert-\vert\mathbf{s}\vert}}{\prod_{c\in\mathbf{s}} m_c}
f_i^{\mathcal{K}(\mathbf{m}\mathbf{s})}
\left(x;
\mathbf{m}
,
\begin{bmatrix}
w_{j_1} \| w_{j_2}\\ 
\mathbf{w}_{\overline{\mathbf{s}}}
\end{bmatrix}
,
\mathbf{x}_{\mathbf{s}}
\right),
\end{align*}
with $ \mathbf{m}
=
\begin{bmatrix}
m_{12}\\
\overline{\mathbf{m}}
\end{bmatrix} $ and $ \mathbf{x}_{\mathbf{s}}
=
\begin{bmatrix}
x_{j_{12}} \\ 
\mathbf{x}_{\overline{\mathbf{s}}}
\end{bmatrix} $. Applying \cref{lemma:coupl_comp_multi_expansion}, the right hand side expands into
\begin{align*}
\sum_{ 
	\substack{
		m_{12} \geq 1\\
		\overline{\mathbf{m}} \geq \mathbf{1}_{\vert \overline{\mathbf{s}} \vert}
	}
}
\frac{(-1)^{\vert\overline{\mathbf{m}}\vert-\vert\mathbf{s}\vert}}{\prod_{c\in\overline{\mathbf{s}}} m_c}
\frac{(-1)^{m_{12}}}{m_{12}}
\sum_{ 
	\substack{
	m_1,m_2 \geq 0\\
	m_1,m_2 \leq m_{12}\\
	m_1 + m_2 \geq m_{12}
	}
}
B(m_1,m_2,m_{12})
f_i^{\mathcal{K}(\mathbf{m}\mathbf{s})}
\left(x;
\begin{bmatrix}
m_{1}\\
m_{2}\\
\overline{\mathbf{m}}
\end{bmatrix},
\begin{bmatrix}
w_{j_1} \\ w_{j_2}\\ 
\mathbf{w}_{\overline{\mathbf{s}}}
\end{bmatrix}
,
\begin{bmatrix}
x_{j_{12}} \\ 
x_{j_{12}} \\
\mathbf{x}_{\overline{\mathbf{s}}}
\end{bmatrix}
\right),
\end{align*}
with $ B(m_1,m_2,m_{12}) $ as defined in \cref{lemma:coupl_comp_multi_expansion}. This can be rearranged into
\begin{align*}
\scriptsize
\sum_{ 
	\substack{
		m_1,m_2 \geq 0\\
		\overline{\mathbf{m}} \geq \mathbf{1}_{\vert \overline{\mathbf{s}} \vert}
	}
}
\frac{(-1)^{\vert\overline{\mathbf{m}}\vert-\vert\mathbf{s}\vert}}{\prod_{c\in\overline{\mathbf{s}}} m_c}
\left[
\sum_{ 
	\substack{
		m_{12} \geq 1,m_1,m_2\\
		m_{12} \leq m_1 + m_2 
	}
}
\frac{(-1)^{m_{12}}}{m_{12}}
B(m_1,m_2,m_{12})
\right]
f_i^{\mathcal{K}(\mathbf{m}\mathbf{s})}
\left(x;
\begin{bmatrix}
m_{1}\\
m_{2}\\
\overline{\mathbf{m}}
\end{bmatrix}
,
\begin{bmatrix}
w_{j_1} \\ w_{j_2}\\ 
\mathbf{w}_{\overline{\mathbf{s}}}
\end{bmatrix}
,
\begin{bmatrix}
x_{j_{12}} \\ 
x_{j_{12}} \\
\mathbf{x}_{\overline{\mathbf{s}}}
\end{bmatrix}
\right).
\end{align*}
From \cref{lemma:comb_sum_coupling_to_basis_major}, this simplifies into
\begin{align*}
&
\sum_{ 
	\substack{
		m_1 \geq 1\\
		\overline{\mathbf{m}} \geq \mathbf{1}_{\vert \overline{\mathbf{s}} \vert}
	}
}
\frac{(-1)^{\vert\overline{\mathbf{m}}\vert+m_1-\vert\mathbf{s}\vert}}
{m_1\prod_{c\in\overline{\mathbf{s}}} m_c}
f_i^{\mathcal{K}(\mathbf{m}_1 \mathbf{s})}
\left(x;
\begin{bmatrix}
m_{1}\\
\overline{\mathbf{m}}
\end{bmatrix}
,
\begin{bmatrix}
w_{j_1} \\ 
\mathbf{w}_{\overline{\mathbf{s}}}
\end{bmatrix}
,
\mathbf{x}_{\mathbf{s}}
\right)
\\
&\quad+
\sum_{ 
	\substack{
		m_2 \geq 1\\
		\overline{\mathbf{m}} \geq \mathbf{1}_{\vert \overline{\mathbf{s}} \vert}
	}
}
\frac{(-1)^{\vert\overline{\mathbf{m}}\vert+m_2-\vert\mathbf{s}\vert}}
{m_2\prod_{c\in\overline{\mathbf{s}}} m_c}
f_i^{\mathcal{K}(\mathbf{m}_2 \mathbf{s})}
\left(x;
\begin{bmatrix}
m_{2}\\
\overline{\mathbf{m}}
\end{bmatrix}
,
\begin{bmatrix}
w_{j_2} \\ 
\mathbf{w}_{\overline{\mathbf{s}}}
\end{bmatrix}
,
\mathbf{x}_{\mathbf{s}}
\right)
\\
&\quad=
\leftidx{^b}{}f_i^{\mathcal{K}(\mathbf{s})}
\left(x; 
\begin{bmatrix}
w_{j_1} \\ 
\mathbf{w}_{\overline{\mathbf{s}}}
\end{bmatrix},
\mathbf{x}_{\mathbf{s}}
\right)
+
\leftidx{^b}{}f_i^{\mathcal{K}(\mathbf{s})}
\left(x; 
\begin{bmatrix}
w_{j_2} \\ 
\mathbf{w}_{\overline{\mathbf{s}}}
\end{bmatrix},
\mathbf{x}_{\mathbf{s}}
\right),
\end{align*}
with $ \mathbf{m}_1
=
\begin{bmatrix}
m_{1}\\
\overline{\mathbf{m}}
\end{bmatrix} $ and $ \mathbf{m}_2
=
\begin{bmatrix}
m_{2}\\
\overline{\mathbf{m}}
\end{bmatrix} $, which gives us the right hand side of \cref{thm:basis_f_additive} of \cref{defi:decomp_pure_basis}.
\end{proof}
An evident but important consequence of \cref{thm:decomposition_finite} is the following.
\begin{corollary}
	Consider a finite order $ \hat{f}_i \in \hat{\mathcal{F}}_i^{<\infty} $, with coupling components $ \{f_i^{\mathbf{k}}\}_{\mathbf{k} \geq \mathbf{0}_{\vert\tset\vert}} $ and with basis components $ \{\leftidx{^b}{}f_i^{\mathbf{k}}\}_{\mathbf{k} \geq \mathbf{0}_{\vert\tset\vert}} $. Then,
	\begin{align}
	f_i^{\mathbf{0}} = \leftidx{^b}{}f_i^{\mathbf{0}}.
	\end{align}
\hfill$ \square $
\end{corollary}
This can be generalized with the help of the following definition.
\begin{defi}
	Consider a family of functions $ \{f_i^{\mathbf{k}}\}_{\mathbf{k} \geq \mathbf{0}_{\vert\tset\vert}} $ defined on $ f_i^{\mathbf{k}} \colon
	\xset_i\times
	\mathcal{M}_i^{\mathbf{k}} \times \xset^{\mathbf{k}}
	\to \yset_i$.
	We say that a given index $ \mathbf{k} \geq \mathbf{0}_{\vert\tset\vert} $ is a \textbf{locally maximal order} if $ f_i^{\mathbf{k}} \neq 0_{\yset_i} $ and $ f_i^{\overline{\mathbf{k}}} = 0_{\yset_i} $ for all $ \overline{\mathbf{k}} > \mathbf{k} $ such that $ \mathbf{k}$ and $ \overline{\mathbf{k}} $ have zeros in the same entries.
	\label{defi:locally_maximal_order}
	\hfill$ \square $
\end{defi}
\begin{lemma}
	Consider a finite order $ \hat{f}_i \in \hat{\mathcal{F}}_i^{<\infty} $, with coupling components $ \{f_i^{\mathbf{k}}\}_{\mathbf{k} \geq \mathbf{0}_{\vert\tset\vert}} $ and with basis components $ \{\leftidx{^b}{}f_i^{\mathbf{k}}\}_{\mathbf{k} \geq \mathbf{0}_{\vert\tset\vert}} $.\\
	A given index $ \mathbf{k} \geq \mathbf{0}_{\vert\tset\vert} $ is a locally maximal order with respect to $ \{f_i^{\mathbf{k}}\}_{\mathbf{k} \geq \mathbf{0}_{\vert\tset\vert}}  $ if and only if it is a locally maximal order with respect to $ \{\leftidx{^b}{}f_i^{\mathbf{k}}\}_{\mathbf{k} \geq \mathbf{0}_{\vert\tset\vert}} $. Furthermore, if $ \mathbf{k} \geq \mathbf{0}_{\vert\tset\vert} $ is a locally maximal order, then,
	\begin{align}
	f_i^{\mathbf{k}} = \leftidx{^b}{}f_i^{\mathbf{k}}
	\label{eq:locally_maximal_order_same_func}.
	\end{align}
	\label{lemma:locally_maximal_order_same_func}
	\hfill$ \square $
\end{lemma}
\begin{proof}
	Assume $ \mathbf{k} \geq \mathbf{0}_{\vert\tset\vert} $ is a locally maximal order with respect to $ \{f_i^{\mathbf{k}}\}_{\mathbf{k} \geq \mathbf{0}_{\vert\tset\vert}} $. This implies, from \cref{eq:basis_from_component1_inf}, that $ \leftidx{^b}{}f_i^{\overline{\mathbf{k}}} = 0_{\yset_i} $ whenever $ \mathbf{k} < \overline{\mathbf{k}} $ and $ \mathbf{k}, \overline{\mathbf{k}} $ have zeros in the same entries. Moreover, when $ \overline{\mathbf{k}} = \mathbf{k} $, \cref{eq:basis_from_component1_inf} simplifies into \cref{eq:locally_maximal_order_same_func}.\\
	The exact same reasoning applies in order to prove the converse direction using \cref{eq:component_from_basis1_inf}.
\end{proof}
The following result allows us to build any $ \hat{f}_i \in \hat{\mathcal{F}}_i^{<\infty} $ directly from the specification of a simple and decoupled family of basis components  $ \{\leftidx{^b}{}f_i^{\mathbf{k}}\}_{\mathbf{k} \geq \mathbf{0}_{\vert\tset\vert}} $.
\begin{theorem}
	Every finite order oracle component $ \hat{f}_i \in \hat{\mathcal{F}}_i^{<\infty} $, can be directly expressed in terms of its basis components $ \{\leftidx{^b}{}f_i^{\mathbf{k}}\}_{\mathbf{k} \geq \mathbf{0}_{\vert\tset\vert}} $ according to
	\begin{align}
	\hat{f}_i\left(x; \mathbf{w}_{\mathbf{s}},\mathbf{x}_{\mathbf{s}}\right) 
	=
	\sum_{
			\mathbf{m} \geq \mathbf{0}_{\vert \mathbf{s} \vert}
	}
	\frac{1}{\prod_{c\in\mathbf{s}} m_c!}
	\leftidx{^b}{}f_i^{\mathcal{K}(\mathbf{m}\mathbf{s})}
	\left(x;
	\mathbf{m},
	\mathbf{w}_{\mathbf{s}},
	\mathbf{x}_{\mathbf{s}}
	\right).
	\label{eq:oracle_from_basis1_inf}
	\end{align}
	\label{thm:oracle_basis_relation_basis_to_oracle_to}
	\hfill$ \square $
\end{theorem}
\begin{proof}
We plug in \cref{eq:component_from_basis1_inf} on \cref{eq:composition_f}, which gives us
\begin{align*}
\hat{f}_i\left(x; \mathbf{w}_{\mathbf{s}},\mathbf{x}_{\mathbf{s}}\right) 
=
\sum_{\overline{\mathbf{s}}\subseteq\mathbf{s}}
\sum_{
		\mathbf{m} \geq \mathbf{1}_{\vert \overline{\mathbf{s}} \vert}
}
\frac{1}{\prod_{c\in\overline{\mathbf{s}}} m_c!}
\leftidx{^b}{}f_i^{\mathcal{K}(\mathbf{m}\overline{\mathbf{s}})}
\left(x;
\mathbf{m},
\mathbf{w}_{\overline{\mathbf{s}}},
\mathbf{x}_{\overline{\mathbf{s}}}\right).
\end{align*}
The result comes directly from merging the two sums.
\end{proof}
Similarly to \cref{lemma:coupling_additivity}, we see that the representation on this second decomposition is also component-wise linear.
\begin{lemma}
	For two finite order oracle components $ \hat{f}_i, \hat{g}_i \in \hat{\mathcal{F}}_i^{<\infty} $ with basis components $ \{\leftidx{^b}{}f_i^{\mathbf{k}}\}_{\mathbf{k} \geq \mathbf{0}_{\vert\tset\vert}} $ and $ \{\leftidx{^b}{}g_i^{\mathbf{k}}\}_{\mathbf{k} \geq \mathbf{0}_{\vert\tset\vert}} $ respectively, the basis components of $ \hat{h}_i = \alpha\hat{f}_i  + \hat{g}_i  $ are given by $ \{\alpha \leftidx{^b}{}f_i^{\mathbf{k}} + \leftidx{^b}{}g_i^{\mathbf{k}} \}_{\mathbf{k} \geq \mathbf{0}_{\vert\tset\vert}} $ for any scalar $ \alpha $.
	\label{lemma:additivity_basis}
	\hfill$ \square $
\end{lemma}
\begin{proof}
	This comes directly from writing the basis components explicitly in terms of the coupling components as in \cref{eq:basis_from_component1_inf}, together with \cref{lemma:coupling_additivity}.
	\begin{align*}
	&\leftidx{^b}{}h_i^{\mathcal{K}(\mathbf{s})}\left(x;\mathbf{w}_{\mathbf{s}},\mathbf{x}_{\mathbf{s}}\right)
	\\
	&= 
	\sum_{\mathbf{m} \geq \mathbf{1}_{\vert \mathbf{s} \vert}}
	\frac{(-1)^{\vert\mathbf{m}\vert-\vert\mathbf{s}\vert}}{\prod_{c\in\mathbf{s}} m_c}
	\left(
	\alpha f_i^{\mathcal{K}(\mathbf{m}\mathbf{s})}
	+
	g_i^{\mathcal{K}(\mathbf{m}\mathbf{s})}
	\right)
	\left(x;
	\mathbf{m},
	\mathbf{w}_{\mathbf{s}},
	\mathbf{x}_{\mathbf{s}}
	\right)\\
	&=
	\alpha
	\left(
	\sum_{\mathbf{m} \geq \mathbf{1}_{\vert \mathbf{s} \vert}}
	\frac{(-1)^{\vert\mathbf{m}\vert-\vert\mathbf{s}\vert}}{\prod_{c\in\mathbf{s}} m_c}
	f_i^{\mathcal{K}(\mathbf{m}\mathbf{s})}
	\left(x;
	\mathbf{m},
	\mathbf{w}_{\mathbf{s}},
	\mathbf{x}_{\mathbf{s}}
	\right)
	\right)
	+
	\sum_{\mathbf{m} \geq \mathbf{1}_{\vert \mathbf{s} \vert}}
	\frac{(-1)^{\vert\mathbf{m}\vert-\vert\mathbf{s}\vert}}{\prod_{c\in\mathbf{s}} m_c}
	g_i^{\mathcal{K}(\mathbf{m}\mathbf{s})}
	\left(x;
	\mathbf{m},
	\mathbf{w}_{\mathbf{s}},
	\mathbf{x}_{\mathbf{s}}
	\right)
	\\
	&=
	\alpha
	\leftidx{^b}{}f_i^{\mathcal{K}(\mathbf{s})}\left(x;\mathbf{w}_{\mathbf{s}},\mathbf{x}_{\mathbf{s}}\right)
	+
	\leftidx{^b}{}g_i^{\mathcal{K}(\mathbf{s})}\left(x;\mathbf{w}_{\mathbf{s}},\mathbf{x}_{\mathbf{s}}\right)
	\\
	&=
	\left(
	\alpha 
	\leftidx{^b}{}f_i^{\mathcal{K}(\mathbf{s})}
	+
	\leftidx{^b}{}g_i^{\mathcal{K}(\mathbf{s})}
	\right)
	\left(x;\mathbf{w}_{\mathbf{s}},\mathbf{x}_{\mathbf{s}}\right).
	\end{align*}
\end{proof}
The following examples illustrate the proposed decomposition.
\begin{exmp}
	Consider a single-type finite order oracle component $ \hat{f}_i\in \hat{\mathcal{F}}_i^{<\infty} $ with basis components $ \{\leftidx{^b}{}f_i^{k}\}_{k \geq 0} $ such that, for some fixed $ n>0 $
	\begin{align*}
	\leftidx{^b}{}f_i^{\vert \mathbf{s} \vert}
	\left(x;
	\mathbf{w}_{\mathbf{s}},
	\mathbf{x}_{\mathbf{s}}
	\right) 
	= 
	\begin{cases}
	n! 	\prod_{c\in\mathbf{s}}
	(w_c x_c) & \vert \mathbf{s} \vert = n,\\
	f_i^0(x) & \vert \mathbf{s} \vert  = 0,\\
	0 & \text{otherwise}.
	\end{cases}
	\end{align*}
	It is clear that $ \{\leftidx{^b}{}f_i^{k}\}_{k \geq 0} $ satisfy \cref{thm:basis_f_permutation_invariance,thm:basis_f_additive} of \cref{defi:decomp_pure_basis}. Using \cref{thm:oracle_basis_relation_basis_to_oracle_to} we can find the corresponding oracle component directly. That is,
	\begin{align*}
	\hat{f}_i\left(x; \mathbf{w}_{\mathbf{s}},\mathbf{x}_{\mathbf{s}}\right) 
	&=
	\sum_{
		\mathbf{m} \geq \mathbf{0}_{\vert \mathbf{s} \vert}
	}
	\frac{1}{\prod_{c\in\mathbf{s}} m_c!}
	\leftidx{^b}{}f_i^{\vert \mathbf{m} \vert}
	\left(x;
	\mathbf{m},
	\mathbf{w}_{\mathbf{s}},
	\mathbf{x}_{\mathbf{s}}
	\right)\\
	&=
	f_i^0(x)
	+
	n!
	\sum_{
		\substack{
			\mathbf{m} \geq \mathbf{0}_{\vert \mathbf{s} \vert}\\
			\vert \mathbf{m} \vert = n
		}
	}
	\prod_{c\in\mathbf{s}}
	\frac{(w_c x_c)^{m_c}}{m_c!}
	\\
	&=
	f_i^0(x)
	+
	\left(
	\sum_{c\in\mathbf{s}} w_c x_c
	\right)^{n},
	\end{align*}
	which is exactly the same oracle component as in \cref{lemma:power_n}.\\
	As a sanity check we can easily verify from \cref{eq:component_from_basis1_inf} that the coupling components $ \{f_i^{k}\}_{k \geq 0} $ match the previously calculated ones. In particular, for $ \vert \mathbf{s} \vert > 0$, 
	\begin{align*}
	f_i^{\vert \mathbf{s} \vert}
	\left(x;
	\mathbf{w}_{\mathbf{s}},
	\mathbf{x}_{\mathbf{s}}
	\right)
	= 
	\sum_{
		\mathbf{m} \geq \mathbf{1}_{\vert \mathbf{s} \vert}
	}
	\frac{1}{\prod_{c\in\mathbf{s}} m_c!}
	\leftidx{^b}{}f_i^{\vert \mathbf{m} \vert}
	\left(x;
	\mathbf{m},
	\mathbf{w}_{\mathbf{s}},
	\mathbf{x}_{\mathbf{s}}
	\right)
	=
	n!
	\sum_{
		\substack{
			\mathbf{m} \geq \mathbf{1}_{\vert \mathbf{s} \vert}\\
			\vert \mathbf{m} \vert = n
		}
	}
	\prod_{c\in\mathbf{s}}
	\frac{(w_c x_c)^{m_c}}{m_c!}.
	\end{align*}
	Finally, note that \cref{lemma:locally_maximal_order_same_func} is verified for $ \vert \mathbf{s} \vert = n $. That is, $ f_i^{n} = \leftidx{^b}{}f_i^{n} $.
	\label{lemma:power_n_basis}
	\hfill$ \square $
\end{exmp}
We now extend \cref{lemma:power_n_basis} to the polynomial case.
\begin{exmp}
	Consider a single-type finite order oracle component $ \hat{f}_i\in \hat{\mathcal{F}}_i^{<\infty} $ with basis components $ \{\leftidx{^b}{}f_i^{k}\}_{k \geq 0} $ such that, for some fixed $ N>0 $
	\begin{align*}
	\leftidx{^b}{}f_i^{\vert \mathbf{s} \vert}
	\left(x;
	\mathbf{w}_{\mathbf{s}},
	\mathbf{x}_{\mathbf{s}}
	\right) 
	&=
	\begin{cases}
	a_{\vert \mathbf{s} \vert }\vert \mathbf{s} \vert! 	\prod_{c\in\mathbf{s}}
	(w_c x_c) & 0< \vert \mathbf{s} \vert \leq N,\\
	f_i^0(x) & \vert \mathbf{s} \vert = 0,\\
	0 & \text{otherwise}.
	\end{cases}
	\end{align*}
	It is clear that $ \{\leftidx{^b}{}f_i^{k}\}_{k \geq 0} $ satisfy \cref{thm:basis_f_permutation_invariance,thm:basis_f_additive} of \cref{defi:decomp_pure_basis}.\\ From \cref{lemma:power_n_basis,lemma:additivity_basis}, we conclude that the corresponding oracle component is given by
	\begin{align*}
	\hat{f}_i
	\left(x; \mathbf{w}_{\mathbf{s}},\mathbf{x}_{\mathbf{s}}
	\right) 
	=
	f_i^0(x)
	+
	\sum_{n=1}^{N}
	a_{n}
	\left(
	\sum_{c\in\mathbf{s}} w_c x_c
	\right)^{n},
	\end{align*}
	which is exactly the same oracle component as in \cref{cor:polynomial}. Note that we could also obtain this directly through \cref{thm:oracle_basis_relation_basis_to_oracle_to}.
	\label{cor:polynomial_basis}
	\hfill$ \square $
\end{exmp}
We now extend the previous result for multi-type networks.
\begin{exmp}
	Consider a multi-type finite order oracle component $ \hat{f}_i\in \hat{\mathcal{F}}_i^{<\infty} $ with basis components $ \{\leftidx{^b}{}f_i^{\mathbf{k}}\}_{\mathbf{k} \geq \mathbf{0}_{\vert \tset \vert}} $ such that, for $ \{a_{\mathbf{n}}\}_{\mathbf{n} > \mathbf{0}_{\vert \tset \vert}} $ with finite support,
	\begin{align*}
	\leftidx{^b}{}f_i^{\mathbf{k}}
	\left(x;
	\mathbf{w}_{\mathbf{s}},
	\mathbf{x}_{\mathbf{s}}
	\right) 
	&=
	\begin{cases}
	a_{\mathbf{k}} \prod_{j\in\tset}k_j! \prod_{c\in\mathbf{s}_j}
	(w_c x_c) & \mathbf{k} > \mathbf{0}_{\vert \tset \vert},\\
	f_i^{\mathbf{0}}(x) & 
	\mathbf{k} = \mathbf{0}_{\vert \tset \vert}.
	\end{cases}
	\end{align*}
	It is clear that $ \{\leftidx{^b}{}f_i^{\mathbf{k}}\}_{\mathbf{k} \geq \mathbf{0}_{\vert \tset \vert}} $ satisfy \cref{thm:basis_f_permutation_invariance,thm:basis_f_additive} of \cref{defi:decomp_pure_basis}. The corresponding oracle component is given by
	\begin{align*}
	\hat{f}_i\left(x; \mathbf{w}_{\mathbf{s}},\mathbf{x}_{\mathbf{s}}\right) 
	&=
	\sum_{
		\mathbf{m} \geq \mathbf{0}_{\vert \mathbf{s} \vert}
	}
	\frac{1}{\prod_{c\in\mathbf{s}} m_c!}
	\leftidx{^b}{}f_i^{\mathcal{K}(\mathbf{m}\mathbf{s})}
	\left(x;
	\mathbf{m},
	\mathbf{w}_{\mathbf{s}},
	\mathbf{x}_{\mathbf{s}}
	\right)
	\\
	&=
	f_i^{\mathbf{0}}(x)
	+
	\sum_{\mathbf{n} > \mathbf{0}_{\vert \tset \vert}}
	\sum_{
		\substack{
			\mathbf{m} \geq \mathbf{0}_{\vert \mathbf{s} \vert}
			\\
			\vert \mathbf{m}_1 \vert = n_1\\
			\ldots\\
			\vert \mathbf{m}_{\vert \tset\vert} \vert = n_{\vert \tset\vert}
		}
	}
	a_{\mathbf{n}}
	\left(
	\prod_{j \in \tset}
	n_j!
	\right)
	\prod_{c\in\mathbf{s}}
	\frac{(w_c x_c)^{m_c}}{m_c!}
	\\
	&=
	f_i^{\mathbf{0}}(x)
	+
	\sum_{\mathbf{n} > \mathbf{0}_{\vert \tset \vert}}
	a_{\mathbf{n}}
	\prod_{j\in\tset}
	\sum_{
		\substack{
			\mathbf{m}_j \geq \mathbf{0}_{\vert \mathbf{s}_j \vert}
			\\
			\vert \mathbf{m}_j \vert = n_j
		}
	}
	n_j!
	\prod_{c\in\mathbf{s}_j}
	\frac{(w_c x_c)^{m_c}}{m_c!}
	\\
	&=
	f_i^{\mathbf{0}}(x)
	+
	\sum_{\mathbf{n} > \mathbf{0}_{\vert \tset \vert}}
	a_{\mathbf{n}}
	\prod_{j\in \tset}
	\left(
	\sum_{c\in\mathbf{s}_j} w_c x_c
	\right)^{n_j},
	\end{align*}
	which is exactly the same oracle component as in \cref{lemma:power_poly_multi}.
	\hfill$ \square $
\end{exmp}
We now consider a slightly more complicated type of basis components.
\begin{exmp}
Consider a single-type finite order oracle component $ \hat{f}_i\in \hat{\mathcal{F}}_i^{<\infty} $ with basis components $ \{\leftidx{^b}{}f_i^{k}\}_{k \geq 0} $ such that, for some fixed $ n,k $ with $ n\geq k >0 $
\begin{align*}
\leftidx{^b}{}f_i^{\vert \mathbf{s} \vert}
\left(x;
\mathbf{w}_{\mathbf{s}},
\mathbf{x}_{\mathbf{s}}
\right) 
= 
\begin{cases}
(n-k)! k! (\prod_{c\in\mathbf{s}} w_c) e_k(\mathbf{x}_{\mathbf{s}}) & \vert \mathbf{s} \vert = n,\\
f_i^0(x) & \vert \mathbf{s} \vert  = 0,\\
0 & \text{otherwise}.
\end{cases}
\end{align*}
where $ e_k $ denotes what is called elementary symmetric polynomials. With the multi-index notation this can be written as
\begin{align*}
e_k(\mathbf{x}_{\mathbf{s}})
=
\sum_{
	\substack{
		\mathbf{q} \geq \mathbf{0}_{\vert \mathbf{s} \vert}\\
		\mathbf{q} \leq \mathbf{1}_{\vert \mathbf{s} \vert}\\
		\vert \mathbf{q} \vert = k
	}
}
\prod_{c\in\mathbf{s}}
x_c^{q_c}.
\end{align*}
It is clear that $ \{\leftidx{^b}{}f_i^{k}\}_{k \geq 0} $ satisfy \cref{thm:basis_f_permutation_invariance,thm:basis_f_additive} of \cref{defi:decomp_pure_basis}. We show that the oracle components $ \{f_i^{k}\}_{k \geq 0} $ can be found to be
\begin{align*}
\hat{f}_i
\left(x;
\mathbf{w}_{\mathbf{s}}
,
\mathbf{x}_{\mathbf{s}}
\right)
&=
f_i^{\mathbf{0}}(x)
+
\left(
\sum_{c\in\mathbf{s}}
w_c
\right)^{n-k}
\left(
\sum_{c\in\mathbf{s}}
w_c x_c
\right)^{k}.
\end{align*}
\label{lemma:power_n_basis_symetric}
To prove this, firstly note that
\begin{align*}
e_k(\mathbf{m}\mathbf{x}_{\mathbf{s}})
=
\sum_{
	\substack{
		\mathbf{q} \geq \mathbf{0}_{\vert \mathbf{s} \vert}\\
		\mathbf{q} \leq \mathbf{m}\\
		\vert \mathbf{q} \vert = k
	}
}
\prod_{c\in\mathbf{s}}
x_c^{q_c}
\binom{m_c}{q_c}.
\end{align*}
Using \cref{thm:oracle_basis_relation_basis_to_oracle_to},
\begin{align*}
\hat{f}_i
\left(x;
\mathbf{w}_{\mathbf{s}}
,
\mathbf{x}_{\mathbf{s}}
\right)
&= 
\sum_{
	\mathbf{m} \geq \mathbf{0}_{\vert \mathbf{s} \vert}
}
\frac{1}{\prod_{c\in\mathbf{s}} m_c!}
\leftidx{^b}{}f_i^{\vert \mathbf{m} \vert}
\left(x;
\mathbf{m},
\mathbf{w}_{\mathbf{s}},
\mathbf{x}_{\mathbf{s}}
\right)
\\
&=
f_i^{\mathbf{0}}(x)
+
\sum_{
	\substack{
		\mathbf{m} \geq \mathbf{0}_{\vert \mathbf{s} \vert}\\
		\vert \mathbf{m} \vert = n
	}
}
\frac{1}{\prod_{c\in\mathbf{s}} m_c!}
\leftidx{^b}{}f_i^{\vert \mathbf{m} \vert}
\left(x;
\mathbf{m},
\mathbf{w}_{\mathbf{s}},
\mathbf{x}_{\mathbf{s}}
\right)
\\
&=
f_i^{\mathbf{0}}(x)
+
(n-k)! k!
\sum_{
	\substack{
		\mathbf{m} \geq \mathbf{0}_{\vert \mathbf{s} \vert}\\
		\vert \mathbf{m} \vert = n
	}
}
\frac{1}{\prod_{c\in\mathbf{s}} m_c!}
\left(
\prod_{c\in\mathbf{s}} w_c^{m_c}
\right)
e_k(\mathbf{m}\mathbf{x}_{\mathbf{s}})
\\
&=
f_i^{\mathbf{0}}(x)
+
(n-k)! k!
\sum_{
	\substack{
		\mathbf{m} \geq \mathbf{0}_{\vert \mathbf{s} \vert}\\
		\vert \mathbf{m} \vert = n
	}
}
\left(
\prod_{c\in\mathbf{s}} 
\frac{w_c^{m_c}}{m_c!}
\right)
\sum_{
	\substack{
		\mathbf{q} \geq \mathbf{0}_{\vert \mathbf{s} \vert}\\
		\mathbf{q} \leq \mathbf{m}\\
		\vert \mathbf{q} \vert = k
	}
}
\prod_{c\in\mathbf{s}}
x_c^{q_c}
\binom{m_c}{q_c}
\\
&=
f_i^{\mathbf{0}}(x)
+
(n-k)! k!
\sum_{
	\substack{
		\mathbf{m} \geq \mathbf{0}_{\vert \mathbf{s} \vert}\\
		\vert \mathbf{m} \vert = n
	}
}
\sum_{
	\substack{
		\mathbf{q} \geq \mathbf{0}_{\vert \mathbf{s} \vert}\\
		\mathbf{q} \leq \mathbf{m}\\
		\vert \mathbf{q} \vert = k
	}
}
\prod_{c\in\mathbf{s}}
\frac{w_c^{m_c-q_c} (w_c x_c)^{q_c}}{(m_c-q_c)!q_c!}.
\end{align*}
Define $ p_c:=m_c-q_c $. Then, we can write this in terms of $ \mathbf{p} $ as
\begin{align*}
\hat{f}_i
\left(x;
\mathbf{w}_{\mathbf{s}}
,
\mathbf{x}_{\mathbf{s}}
\right)
&=
f_i^{\mathbf{0}}(x)
+
(n-k)! k!
\sum_{
	\substack{
		\mathbf{p} \geq \mathbf{0}_{\vert \mathbf{s} \vert}\\
		\mathbf{q} \geq \mathbf{0}_{\vert \mathbf{s} \vert}\\
		\vert \mathbf{p} \vert = n-k\\
		\vert \mathbf{q} \vert = k
	}
}
\prod_{c\in\mathbf{s}}
\frac{w_c^{p_c} (w_c x_c)^{q_c}}{p_c!q_c!}
\\
&=
f_i^{\mathbf{0}}(x)
+
\left[
(n-k)!
\sum_{
	\substack{
		\mathbf{p} \geq \mathbf{0}_{\vert \mathbf{s} \vert}\\
		\vert \mathbf{p} \vert = n-k
	}
}
\prod_{c\in\mathbf{s}}
\frac{w_c^{p_c}}{p_c!}
\right]
\left[
k!
\sum_{
	\substack{
		\mathbf{q} \geq \mathbf{0}_{\vert \mathbf{s} \vert}\\
		\vert \mathbf{q} \vert = k
	}
}
\prod_{c\in\mathbf{s}}
\frac{(w_c x_c)^{q_c}}{q_c!}
\right]
\\
&=
f_i^{\mathbf{0}}(x)
+
\left(
\sum_{c\in\mathbf{s}}
w_c
\right)^{n-k}
\left(
\sum_{c\in\mathbf{s}}
w_c x_c
\right)^{k},
\end{align*}
which completes our proof.
\hfill$ \square $
\end{exmp}	
For completeness sake, we present in \cref{thm:oracle_basis_relation_oracle_to basis} the inverse result of \cref{thm:oracle_basis_relation_basis_to_oracle_to}. That is, we express the basis components in terms of the oracle components.
In this result, we use a generalization of Stirling numbers called $ r $-Stirling numbers, which are defined in \cite{broder1984r}.
\begin{defi}
	The unsigned r-Stirling numbers of the first kind, $ \stirstg{r}{n}{k} $, with $ r,n,k \geq 0$, are given by the recurrence relation
	\begin{align*}
	\stirstg{r}{n}{k} = (n-1) \stirstg{r}{n-1}{k} + \stirstg{r}{n-1}{k-1}
	,\quad n>r, k>0,
	\end{align*}
	together with the boundary conditions
	\begin{align*}
	\stirstg{r}{r}{k} &= \delta_{r,k},\\
	\stirstg{r}{n}{k} &= 0 \quad n < r,\\
	\stirstg{r}{n}{0} &= 0 \quad n > r.
	\end{align*}
	\label{defi:stirling_1_generalized}
	\hfill$ \square $
\end{defi}
\begin{remark}
	Note that $ \stirstg{0}{n}{k} = \stirst{n}{k} $. Moreover, $ \stirstg{1}{n}{k} = \stirst{n}{k} $ when $ n>0 $.
	\hfill$ \square $
\end{remark}
We denote by $ \hat{\mathcal{F}}_i^{\leq \mathbf{K}} $ the subset of all $ \hat{f}_i \in \hat{\mathcal{F}}_i $ such that all of their non-zero coupling components are inside the subset $ \{f_i^{\mathbf{k}}\}_{\mathbf{k} \leq \mathbf{K}} $. From \cref{lemma:coupling_additivity}, this forms a subspace.
\begin{theorem}
	For every finite order $ \hat{f}_i \in \hat{\mathcal{F}}_i^{<\infty} $, we can express the set of basis components $ \{\leftidx{^b}{}f_i^{\mathbf{k}}\}_{\mathbf{k} \geq \mathbf{0}_{\vert\tset\vert}} $, which have the properties as in \cref{defi:decomp_pure_basis}, in terms of its oracle components $ \hat{f}_i $.	In particular, for any $ \mathbf{K}\geq \mathbf{0}_{\vert\tset\vert} $ such that $ \hat{f}_i \in \hat{\mathcal{F}}_i^{\leq \mathbf{K}} $, we have that
\begin{align}
\leftidx{^b}{}f_i^{\mathcal{K}(\mathbf{s})}\left(x;\mathbf{w}_{\mathbf{s}},\mathbf{x}_{\mathbf{s}}\right)
=
(-1)^{\vert\mathbf{s}\vert}
\sum_{\overline{\mathbf{s}}\subseteq\mathbf{s}}
\sum_{
	\substack{
		\mathbf{M} \geq \mathbf{1}_{\vert \overline{\mathbf{s}} \vert}\\
		\mathcal{K}(\mathbf{M}\overline{\mathbf{s}}) \leq \mathbf{K}
	}
}
\frac{(-1)^{\vert\mathbf{M}\vert}}
{\prod_{c\in\overline{\mathbf{s}}} M_c}
\left[
\prod_{j \in \tset}
C(
K_j,
\vert \mathbf{M}_j \vert,
\vert  \mathbf{s}_j\setminus\overline{\mathbf{s}}_j \vert
)
\right]
\hat{f}_i
\left(x;
\mathbf{M},
\mathbf{w}_{\overline{\mathbf{s}}},
\mathbf{x}_{\overline{\mathbf{s}}}
\right)
\label{eq:basis1_inf_from_oracle}
\end{align}
where $ C(K,M,r) $, with $ K\geq M \geq 0 $ and $ r\geq 0 $, is defined as
\begin{align}
C(K, M, r)
&:=
\frac{r!}
{(K - M)!}
\stirstg{M +1 }{K+1}{r + M + 1}.
\label{eq:oracle_basis_relation_oracle_to basis_coefs}
\end{align}
\label{thm:oracle_basis_relation_oracle_to basis}
\hfill$ \square $
\end{theorem}
In order to prove this, we require \cref{thm:sum_prod_comb_stirstg}, which is proved in \cref{apend:stirling_generalized}.
	\begin{lemma}
	For $ \mathbf{M} \geq \mathbf{1}_{k} $, with $ n\geq \vert \mathbf{M} \vert $ and $ k,r\geq0 $, we have that
	\begin{align}
	\sum_{
		\substack{
			\mathbf{m} \geq \mathbf{M} \\
			\mathbf{p} \geq \mathbf{1}_{r}\\
			\vert \mathbf{m} \vert
			+
			\vert\mathbf{p}\vert \leq n
		}
	}
	\prod_{i=1}^{k}
	\binom{m_i -1}{M_i -1}
	\prod_{j=1}^r
	\frac{1}{p_j}
	=
	\frac{r!}{(n - \vert \mathbf{M} \vert)!}
	\stirstg{\vert \mathbf{M} \vert +1 }{n+1}{r + \vert \mathbf{M} \vert + 1}.
	\end{align}
	\label{thm:sum_prod_comb_stirstg}
	\hfill$ \square $
\end{lemma}
\begin{proof}[Proof of \cref{thm:oracle_basis_relation_oracle_to basis}]
We first note that \cref{eq:decomp_explicit} can be generalized into 
\begin{align*}
f_i^{\mathcal{K}(\mathbf{m}\mathbf{s})}\left(x;
\mathbf{m},
\mathbf{w}_{\mathbf{s}},
\mathbf{x}_{\mathbf{s}}
\right)
=
\sum_{
	\substack{
	\overline{\mathbf{m}} \geq \mathbf{0}_{\vert \mathbf{s} \vert}\\
	\overline{\mathbf{m}} \leq \mathbf{m}
	}
}
\left[
\prod_{c\in\mathbf{s}}
\binom{m_c}{\overline{m}_c}
\right]
(-1)^{\vert\mathbf{m}\vert-\vert\overline{\mathbf{m}}\vert}
\hat{f}_i
\left(x;
\overline{\mathbf{m}},
\mathbf{w}_{\mathbf{s}},
\mathbf{x}_{\mathbf{s}}
\right).
\end{align*}
Plugging in this result into \cref{eq:basis_from_component1_inf}, we obtain
\begin{align*}
\leftidx{^b}{}f_i^{\mathcal{K}(\mathbf{s})}\left(x;\mathbf{w}_{\mathbf{s}},\mathbf{x}_{\mathbf{s}}\right)
&= 
\sum_{
	\substack{
		\mathbf{m} \geq \mathbf{1}_{\vert \mathbf{s} \vert}\\
		\mathcal{K}(\mathbf{m}\mathbf{s}) \leq \mathbf{K}
	}
}
\frac{(-1)^{\vert\mathbf{m}\vert-\vert\mathbf{s}\vert}}{\prod_{c\in\mathbf{s}} m_c}
\sum_{
	\substack{
		\overline{\mathbf{m}} \geq \mathbf{0}_{\vert \mathbf{s} \vert}\\
		\overline{\mathbf{m}} \leq \mathbf{m}
	}
}
\left[
\prod_{c\in\mathbf{s}}
\binom{m_c}{\overline{m}_c}
\right]
(-1)^{\vert\mathbf{m}\vert-\vert\overline{\mathbf{m}}\vert}
\hat{f}_i
\left(x;
\overline{\mathbf{m}},
\mathbf{w}_{\mathbf{s}},
\mathbf{x}_{\mathbf{s}}
\right)
\\
&=
(-1)^{\vert\mathbf{s}\vert}
\sum_{
	\substack{
		\mathbf{m} \geq \mathbf{1}_{\vert \mathbf{s} \vert}\\
		\mathcal{K}(\mathbf{m}\mathbf{s}) \leq \mathbf{K}
	}
}
\sum_{
	\substack{
		\overline{\mathbf{m}} \geq \mathbf{0}_{\vert \mathbf{s} \vert}\\
		\overline{\mathbf{m}} \leq \mathbf{m}
	}
}
\left[
\prod_{c\in\mathbf{s}}
\frac{
\binom{m_c}{\overline{m}_c}}
{m_c}
\right]
(-1)^{\vert\overline{\mathbf{m}}\vert}
\hat{f}_i
\left(x;
\overline{\mathbf{m}},
\mathbf{w}_{\mathbf{s}},
\mathbf{x}_{\mathbf{s}}
\right).
\end{align*}
Rearranging the two sums we get
\begin{align*}
\leftidx{^b}{}f_i^{\mathcal{K}(\mathbf{s})}\left(x;\mathbf{w}_{\mathbf{s}},\mathbf{x}_{\mathbf{s}}\right)
&=
(-1)^{\vert\mathbf{s}\vert}
\sum_{
	\substack{
		\overline{\mathbf{m}} \geq \mathbf{0}_{\vert \mathbf{s} \vert}\\
		\mathcal{K}(\overline{\mathbf{m}}\mathbf{s}) \leq \mathbf{K}
	}
}
(-1)^{\vert\overline{\mathbf{m}}\vert}
\left[
\sum_{
	\substack{
		\mathbf{m} \geq \overline{\mathbf{m}}\\
		\mathbf{m} \geq \mathbf{1}_{\vert \mathbf{s} \vert}\\
		\mathcal{K}(\mathbf{m}\mathbf{s}) \leq \mathbf{K}
	}
}
\prod_{c\in\mathbf{s}}
\frac{
	\binom{m_c}{\overline{m}_c}}
{m_c}
\right]
\hat{f}_i
\left(x;
\overline{\mathbf{m}},
\mathbf{w}_{\mathbf{s}},
\mathbf{x}_{\mathbf{s}}
\right).
\end{align*}
We now break the first sum into
\begin{align*}
\sum_{
	\substack{
		\overline{\mathbf{m}} \geq \mathbf{0}_{\vert \mathbf{s} \vert}\\
		\mathcal{K}(\overline{\mathbf{m}}\mathbf{s}) \leq \mathbf{K}
	}
}
=
\sum_{\overline{\mathbf{s}}\subseteq\mathbf{s}}
\sum_{
	\substack{
		\mathbf{M} \geq \mathbf{1}_{\vert \overline{\mathbf{s}} \vert}\\
		\mathcal{K}(\mathbf{M}\overline{\mathbf{s}}) \leq \mathbf{K}
	}
},
\end{align*}
that is, we correspond a given $ \overline{\mathbf{m}} \geq \mathbf{0}_{\vert \mathbf{s}\vert} $ to the subset $ \overline{\mathbf{s}}\subseteq\mathbf{s} $, of its non-zero entries. Therefore, we have that $ \overline{\mathbf{m}}_{\overline{\mathbf{s}}} = \mathbf{M} $ and $ \overline{\mathbf{m}}_{\mathbf{s}\setminus\overline{\mathbf{s}}} = \mathbf{0}_{\vert\mathbf{s}\setminus\overline{\mathbf{s}}\vert}  $. Note that according to this split, the condition $ \mathbf{m} \geq \overline{\mathbf{m}} $ becomes $ \mathbf{m}_{\overline{\mathbf{s}}} \geq \mathbf{M} $ and $ \mathbf{m}_{\mathbf{s}\setminus\overline{\mathbf{s}}} \geq \mathbf{0}_{\vert\mathbf{s}\setminus\overline{\mathbf{s}}\vert} $. On the other hand, the condition $ \mathbf{m} \geq \mathbf{1}_{\vert \mathbf{s}\vert} $ becomes $ \mathbf{m}_{\overline{\mathbf{s}}} \geq \mathbf{1}_{\vert \overline{\mathbf{s}}\vert} $ and $ \mathbf{m}_{\mathbf{s}\setminus\overline{\mathbf{s}}} \geq \mathbf{1}_{\vert \mathbf{s}\setminus\overline{\mathbf{s}}\vert}  $. Out of the resulting four conditions, the non-redundant ones are clearly $ \mathbf{m}_{\overline{\mathbf{s}}} \geq \mathbf{M} $ and $  \mathbf{m}_{\mathbf{s}\setminus\overline{\mathbf{s}}} \geq \mathbf{1}_{\vert \mathbf{s}\setminus\overline{\mathbf{s}}\vert} $. This results in
\begin{align*}
&\leftidx{^b}{}f_i^{\mathcal{K}(\mathbf{s})}\left(x;\mathbf{w}_{\mathbf{s}},\mathbf{x}_{\mathbf{s}}\right)
\\
&=
(-1)^{\vert\mathbf{s}\vert}
\sum_{\overline{\mathbf{s}}\subseteq\mathbf{s}}
\sum_{
	\substack{
		\mathbf{M} \geq \mathbf{1}_{\vert \overline{\mathbf{s}} \vert}\\
		\mathcal{K}(\mathbf{M}\overline{\mathbf{s}}) \leq \mathbf{K}
	}
}
(-1)^{\vert\mathbf{M}\vert}
\left[
\sum_{
	\substack{
		\mathbf{m}_{\overline{\mathbf{s}}} \geq \mathbf{M} \\
		 \mathbf{m}_{\mathbf{s}\setminus\overline{\mathbf{s}}} \geq \mathbf{1}_{\vert \mathbf{s}\setminus\overline{\mathbf{s}}\vert}\\
		\mathcal{K}(\mathbf{m}\mathbf{s}) \leq \mathbf{K}
	}
}
\prod_{c\in\overline{\mathbf{s}}}
\frac{\binom{m_c}{M_c}}{m_c}
\prod_{d\in\mathbf{s}\setminus\overline{\mathbf{s}}}
\frac{1}{m_d}
\right]
\hat{f}_i
\left(x;
\mathbf{M},
\mathbf{w}_{\overline{\mathbf{s}}},
\mathbf{x}_{\overline{\mathbf{s}}}
\right)
\\
&=
(-1)^{\vert\mathbf{s}\vert}
\sum_{\overline{\mathbf{s}}\subseteq\mathbf{s}}
\sum_{
	\substack{
		\mathbf{M} \geq \mathbf{1}_{\vert \overline{\mathbf{s}} \vert}\\
		\mathcal{K}(\mathbf{M}\overline{\mathbf{s}}) \leq \mathbf{K}
	}
}
\frac{(-1)^{\vert\mathbf{M}\vert}}
{\prod_{c\in\overline{\mathbf{s}}} M_c}
\left[
\sum_{
	\substack{
		\mathbf{m}_{\overline{\mathbf{s}}} \geq \mathbf{M} \\
		\mathbf{m}_{\mathbf{s}\setminus\overline{\mathbf{s}}} \geq \mathbf{1}_{\vert \mathbf{s}\setminus\overline{\mathbf{s}}\vert}\\
		\mathcal{K}(\mathbf{m}\mathbf{s}) \leq \mathbf{K}
	}
}
\prod_{c\in\overline{\mathbf{s}}}
\binom{m_c -1}{M_c -1}
\prod_{d\in\mathbf{s}\setminus\overline{\mathbf{s}}}
\frac{1}{m_d}
\right]
\hat{f}_i
\left(x;
\mathbf{M},
\mathbf{w}_{\overline{\mathbf{s}}},
\mathbf{x}_{\overline{\mathbf{s}}}
\right).
\end{align*}
Note that $ \frac{\binom{m_c}{M_c}}{m_c} = \frac{\binom{m_c -1 }{M_c -1}}{M_c} $, whenever $ M_c \geq 1 $.\\
We now show that the expression in brackets gives us $ \prod_{j \in \tset}
C(
K_j,
\vert \mathbf{M}_j \vert,
\vert  \mathbf{s}_j\setminus\overline{\mathbf{s}}_j \vert
) $, as given by \cref{eq:oracle_basis_relation_oracle_to basis_coefs}.
We rearrange it by breaking all the multi-indices according to the typing of their cells. That is,
\begin{align*}
\prod_{j \in \tset}
\sum_{
	\substack{
		\mathbf{m}_{\overline{\mathbf{s}}_j} \geq \mathbf{M}_j \\
		\mathbf{m}_{\mathbf{s}_j\setminus\overline{\mathbf{s}}_j} \geq \mathbf{1}_{\vert \mathbf{s}_j\setminus\overline{\mathbf{s}}_j\vert}\\
		\vert \mathbf{m}_{\overline{\mathbf{s}}_j} \vert
		+
		\vert\mathbf{m}_{\mathbf{s}_j\setminus\overline{\mathbf{s}}_j}\vert \leq K_j
	}
}
\prod_{c\in\overline{\mathbf{s}}_j}
\binom{m_c -1}{M_c -1}
\prod_{d\in\mathbf{s}_j\setminus\overline{\mathbf{s}}_j}
\frac{1}{m_d}.
\end{align*}
from \cref{thm:sum_prod_comb_stirstg}, the result is proven.
\end{proof}
\begin{remark}
	Note that if $ \hat{f}_i \in \hat{\mathcal{F}}_i^{\leq \mathbf{K}_1} \cap \hat{\mathcal{F}}_i^{\leq \mathbf{K}_2} $, with $ \mathbf{K}_1 \neq \mathbf{K}_2 $, the associated coefficients $C(
	K_j,
	\vert \mathbf{M}_j \vert,
	\vert  \mathbf{s}_j\setminus\overline{\mathbf{s}}_j \vert
	) $
	will be different when we apply \cref{eq:basis1_inf_from_oracle} to $ \mathbf{K}_1 $ and $\mathbf{K}_2 $.\\
	The fact that there are multiple valid formulas that express $ \{\leftidx{^b}{}f_i^{\mathbf{k}}\}_{\mathbf{k} \geq \mathbf{0}_{\vert\tset\vert}} $ as a function of $ \hat{f}_i $ might seem unexpected. One way to convince ourselves that this is reasonable, is to consider a simple case like $ \hat{f}_i \in \hat{\mathcal{F}}_i^{\leq \mathbf{0}}$. In this case, we have that $  
	\leftidx{^b}{}f_i^{\mathbf{0}}
	\left(x	\right)
	=
	\hat{f}_i
	\left(x;
	\mathbf{w}_{\mathbf{s}},
	\mathbf{x}_{\mathbf{s}}
	\right)
	$ and it is easy to get creative and find multiple formulas that are all valid under the (very strict) assumption that $ \hat{f}_i \in \hat{\mathcal{F}}_i^{\leq \mathbf{0}}$.
	\hfill$ \square $
\end{remark}
\begin{remark}
	Note that the coefficients $C(
	K_j,
	\vert \mathbf{M}_j \vert,
	\vert  \mathbf{s}_j\setminus\overline{\mathbf{s}}_j \vert
	) $ diverge as $ \mathbf{K} \to \infty $. In particular, $ C(K_j, 1, 0) = \stirstg{2}{K_j+1}{2}/(K_j-1)! = K_j!/(K_j-1)! = K_j$ for $ K_j \geq 1 $. Therefore, the limit case of \cref{thm:oracle_basis_relation_oracle_to basis} does not give us a universal formula that works for all $ \hat{f}_i\in \hat{\mathcal{F}}_i^{<\infty} $.
	\hfill$ \square $
\end{remark}
\subsection{Infinite coupling order}
All the results of \cref{subsec:finite_order} fall under the assumption that the oracle functions are in $ \hat{\mathcal{F}}_i^{<\infty} $. That is, they have finite order. There are, however, plentiful useful functions that lie outside this subspace, such as the exponential function in \cref{cor:exponential}. Our goal is to create a useful extension of this theory that applies to at least some important functions, such as the exponential and the trigonometric functions.
The first idea that comes to mind is to simply allow the family of basis components $ \{\leftidx{^b}{}f_i^{\mathbf{k}}\}_{\mathbf{k} \geq \mathbf{0}_{\vert\tset\vert}} $ to have infinite support. There is an important issue with such an approach. When dealing with infinite sums we are actually talking about limits on a sequence of partial sums. For this to be well-defined we need to be clear about the meaning of infinite sums of the type $ \sum_{\mathbf{m} \geq \mathbf{1}_{\vert \mathbf{s} \vert}} a_{\mathbf{m}}$. There are plenty of possible definitions, with some of the more obvious ones being 
\begin{align*}
\lim\limits_{N\to\infty} 
\sum_{n=0}^{N}
\sum_{
	\substack{
		\mathbf{m} \geq \mathbf{1}_{\vert \mathbf{s} \vert}\\
		\vert \mathbf{m} \vert = n
	}
}
a_{\mathbf{m}}
,\quad \text{or} \quad
\lim\limits_{N\to\infty} 
\sum_{n=0}^{N}
\sum_{
	\substack{
		\mathbf{m} \geq \mathbf{1}_{\vert \mathbf{s} \vert}\\
		\max(\mathbf{m})  = n
	}
}
a_{\mathbf{m}}.
\end{align*}
However, there is no clear reason for why one definition would be preferable to the other. If we chose one of them and developed our theory based on that, we would only be restricting ourselves to that choice. A different (and better) approach is to simply choose to give up on a $ \{\leftidx{^b}{}f_i^{\mathbf{k}}\}_{\mathbf{k} \geq \mathbf{0}_{\vert\tset\vert}} $ representation for oracle components outside $ \hat{\mathcal{F}}_i^{<\infty} $ and use the following result instead.
\begin{lemma}
	Consider the related set $ \yset_i $ to be a Hausdorff vector space. Then, for every sequence $ (\leftidx{^N}{}\hat{f}_i)_{N\in\mathbb{N}} $, with $ \leftidx{^N}{}\hat{f}_i \in \hat{\mathcal{F}}_i^{<\infty} $ for all $ N\in\mathbb{N} $ such that $ \hat{f}_i := \lim\limits_{N\to\infty} 
	\leftidx{^N}{}\hat{f}_i$ converges pointwise, we have that $ \hat{f}_i \in \hat{\mathcal{F}}_i$.
	\label{cor:build_oracle_from_basis_sequence}
	\hfill$ \square $
\end{lemma}
\begin{proof}
	This is direct from the fact that $ \hat{\mathcal{F}}_i $ is sequentially closed in the topology of pointwise convergence (\cref{lemma:F_closed}) and $ \hat{\mathcal{F}}_i^{<\infty} \subset \hat{\mathcal{F}}_i $. We assume $ \yset_i $ is a vector space so that $ \hat{\mathcal{F}}_i^{<\infty} $ can be defined.
\end{proof}
This provides us with a framework that allows us to build a set of oracle components with infinite order, in particular the ones in $ scl({\hat{\mathcal{F}}_i^{<\infty}}) $, the sequential closure of $ \hat{\mathcal{F}}_i^{<\infty} $. We illustrate this with the following example.
\begin{exmp}
	Consider the sequence of oracle components with finite coupling order \\$ (\leftidx{^N}{}\hat{f}_i)_{N\in\mathbb{N}} $ such that the basis components of $ \leftidx{^N}{}\hat{f}_i \in \hat{\mathcal{F}}_i^{<\infty} $ are given according to
	\begin{align*}
	\leftidx{^{bN}}{}f_i^{\vert \mathbf{s} \vert}
	\left(x;
	\mathbf{w}_{\mathbf{s}},
	\mathbf{x}_{\mathbf{s}}
	\right) 
	&=
	\begin{cases}
	a_{\vert \mathbf{s} \vert }\vert \mathbf{s} \vert! 
	\prod_{c\in\mathbf{s}}
	(w_c x_c) & 0< \vert \mathbf{s} \vert \leq N,\\
	f_i^0(x) & \vert \mathbf{s} \vert = 0,\\
	0 & \text{otherwise}.
	\end{cases}
	\end{align*}
	From \cref{cor:polynomial_basis}, we know that this corresponds to the oracle component
	\begin{align*}
	\leftidx{^N}{}\hat{f}_i
	\left(x; \mathbf{w}_{\mathbf{s}},\mathbf{x}_{\mathbf{s}}
	\right) 
	=
	f_i^0(x)
	+
	\sum_{n=1}^{N}
	a_{n}
	\left(
	\sum_{c\in\mathbf{s}} w_c x_c
	\right)^{n}.
	\end{align*}
	Then, if the infinite series given by $ F(x) = 	\sum_{n=1}^{\infty}
	a_{n}
	x^{n} $	converges for all $ x $, we know from \cref{cor:build_oracle_from_basis_sequence} that $ \hat{f}_i := \lim\limits_{N\to\infty} 
	\leftidx{^N}{}\hat{f}_i $ is given by
	\begin{align*}
	\hat{f}_i
	\left(x; \mathbf{w}_{\mathbf{s}},\mathbf{x}_{\mathbf{s}}
	\right) 
	=
	f_i^0(x)
	+
	F
	\left(
	\sum_{c\in\mathbf{s}} w_c x_c
	\right)
	\end{align*}
	and it is a valid oracle component.
	\label{exmp:taylor_infinite_case}
	\hfill$ \square $
\end{exmp}
\begin{remark}
	Note that this example covers functions such as $ F(x) = \exp(x)-1 $, $ F(x) = \sin(x) $ and $ F(x) = \cos(x)-1 $.
	\hfill$ \square $
\end{remark}
\begin{lemma}
	The set $scl( \hat{\mathcal{F}}_i^{<\infty}) $ is a vector space.
	\hfill$ \square $
\end{lemma}
\begin{proof}
	Consider $ \hat{f}_i,\hat{g}_i \in scl( \hat{\mathcal{F}}_i^{<\infty})$. From assumption, there are sequences $ (\leftidx{^N}{}\hat{f}_i)_{N\in\mathbb{N}} $, $ (\leftidx{^N}{}\hat{g}_i)_{N\in\mathbb{N}} $ with $ \leftidx{^N}{}\hat{f}_i,\leftidx{^N}{}\hat{g}_i \in \hat{\mathcal{F}}_i^{<\infty} $ such that $ \lim\limits_{N\to\infty} 
	\leftidx{^N}{}\hat{f}_i = \hat{f}_i $ and $ \lim\limits_{N\to\infty} 
	\leftidx{^N}{}\hat{g}_i = \hat{g}_i $. Then, the elements of the sequence $ (\alpha \leftidx{^N}{}\hat{f}_i + \leftidx{^N}{}\hat{g}_i)_{N\in\mathbb{N}} $ are also in $ \hat{\mathcal{F}}_i^{<\infty} $ and the sequence converges into $ \alpha \hat{f}_i + \hat{g}_i $. Therefore $ \alpha \hat{f}_i + \hat{g}_i \in scl(\hat{\mathcal{F}}_i^{<\infty}) $ and $ scl(\hat{\mathcal{F}}_i^{<\infty})$ is a vector space.
\end{proof}
\begin{remark}
	Note that $ \hat{\mathcal{F}}_i^{<\infty} \subseteq scl(\hat{\mathcal{F}}_i^{<\infty}) \subseteq \hat{\mathcal{F}}_i $.
	\hfill$ \square $
\end{remark}
\section{Conclusion}
This paper makes CCN theory more useful for practical application thanks to the two decompositions here derived. These decompositions allow us to verify and model dynamical systems (or some other first-order property of a network, such as a measurement function) in a systematic way. Multiple examples illustrating their use are provided.

\bibliographystyle{siamplain} 
\bibliography{references}

\begin{thebibliography}{10}

\bibitem{aguiar2022network}
{\sc M.~Aguiar, C.~Bick, and A.~Dias}, {\em Network dynamics with higher-order
  interactions: Coupled cell hypernetworks for identical cells and synchrony},
  arXiv preprint arXiv:2201.09379,  (2022).

\bibitem{aguiar2018synchronization}
{\sc M.~A. Aguiar and A.~P.~S. Dias}, {\em Synchronization and equitable
  partitions in weighted networks}, Chaos: An Interdisciplinary Journal of
  Nonlinear Science, 28 (2018), p.~073105.

\bibitem{aguiar2017patterns}
{\sc M.~A. Aguiar, A.~P.~S. Dias, and F.~Ferreira}, {\em Patterns of synchrony
  for feed-forward and auto-regulation feed-forward neural networks}, Chaos: An
  Interdisciplinary Journal of Nonlinear Science, 27 (2017), p.~013103.

\bibitem{arenas2008synchronization}
{\sc A.~Arenas, A.~D{\'\i}az-Guilera, J.~Kurths, Y.~Moreno, and C.~Zhou}, {\em
  Synchronization in complex networks}, Physics reports, 469 (2008),
  pp.~93--153.

\bibitem{ashwin2016hopf}
{\sc P.~Ashwin and A.~Rodrigues}, {\em Hopf normal form with sn symmetry and
  reduction to systems of nonlinearly coupled phase oscillators}, Physica D:
  Nonlinear Phenomena, 325 (2016), pp.~14--24.

\bibitem{battiston2021physics}
{\sc F.~Battiston, E.~Amico, A.~Barrat, G.~Bianconi, G.~Ferraz~de Arruda,
  B.~Franceschiello, I.~Iacopini, S.~K{\'e}fi, V.~Latora, Y.~Moreno, et~al.},
  {\em The physics of higher-order interactions in complex systems}, Nature
  Physics, 17 (2021), pp.~1093--1098.

\bibitem{battiston2020networks}
{\sc F.~Battiston, G.~Cencetti, I.~Iacopini, V.~Latora, M.~Lucas, A.~Patania,
  J.-G. Young, and G.~Petri}, {\em Networks beyond pairwise interactions:
  structure and dynamics}, Physics Reports, 874 (2020), pp.~1--92.

\bibitem{bick2016chaos}
{\sc C.~Bick, P.~Ashwin, and A.~Rodrigues}, {\em Chaos in generically coupled
  phase oscillator networks with nonpairwise interactions}, Chaos: An
  Interdisciplinary Journal of Nonlinear Science, 26 (2016), p.~094814.

\bibitem{bick2021higher}
{\sc C.~Bick, E.~Gross, H.~A. Harrington, and M.~T. Schaub}, {\em What are
  higher-order networks?}, arXiv preprint arXiv:2104.11329,  (2021).

\bibitem{broder1984r}
{\sc A.~Z. Broder}, {\em The r-stirling numbers}, Discrete Mathematics, 49
  (1984), pp.~241--259.

\bibitem{comtet2012advanced}
{\sc L.~Comtet}, {\em Advanced Combinatorics: The art of finite and infinite
  expansions}, Springer Science \& Business Media, 2012.

\bibitem{dorfler2014synchronization}
{\sc F.~D{\"o}rfler and F.~Bullo}, {\em Synchronization in complex networks of
  phase oscillators: A survey}, Automatica, 50 (2014), pp.~1539--1564.

\bibitem{golubitsky2006nonlinear}
{\sc M.~Golubitsky and I.~Stewart}, {\em Nonlinear dynamics of networks: the
  groupoid formalism}, Bulletin of the american mathematical society, 43
  (2006), pp.~305--364.

\bibitem{golubitsky2005patterns}
{\sc M.~Golubitsky, I.~Stewart, and A.~T{\"o}r{\"o}k}, {\em Patterns of
  synchrony in coupled cell networks with multiple arrows}, SIAM Journal on
  Applied Dynamical Systems, 4 (2005), pp.~78--100.

\bibitem{kuo2010decompositions}
{\sc F.~Kuo, I.~Sloan, G.~Wasilkowski, and H.~Wo{\'z}niakowski}, {\em On
  decompositions of multivariate functions}, Mathematics of computation, 79
  (2010), pp.~953--966.

\bibitem{memmesheimer2012non}
{\sc R.-M. Memmesheimer and M.~Timme}, {\em Non-additive coupling enables
  propagation of synchronous spiking activity in purely random networks}, PLoS
  computational biology, 8 (2012), p.~e1002384.

\bibitem{nijholt2022dynamical}
{\sc E.~Nijholt and L.~DeVille}, {\em Dynamical systems defined on simplicial
  complexes: symmetries, conjugacies, and invariant subspaces}, arXiv preprint
  arXiv:2204.08350,  (2022).

\bibitem{rodrigues2016kuramoto}
{\sc F.~A. Rodrigues, T.~K.~D. Peron, P.~Ji, and J.~Kurths}, {\em The kuramoto
  model in complex networks}, Physics Reports, 610 (2016), pp.~1--98.

\bibitem{sequeira2021commutative}
{\sc P.~M. Sequeira, A.~P. Aguiar, and J.~Hespanha}, {\em Commutative monoid
  formalism for weighted coupled cell networks and invariant synchrony
  patterns}, SIAM Journal on Applied Dynamical Systems, 20 (2021),
  pp.~1485--1513.

\bibitem{stewart2003symmetry}
{\sc I.~Stewart, M.~Golubitsky, and M.~Pivato}, {\em Symmetry groupoids and
  patterns of synchrony in coupled cell networks}, SIAM Journal on Applied
  Dynamical Systems, 2 (2003), pp.~609--646.

\end{thebibliography}


\begin{thebibliography}{1}

\bibitem{broder1984r_sup}
{\sc A.~Z. Broder}, {\em The r-stirling numbers}, Discrete Mathematics, 49
  (1984), pp.~241--259.

\end{thebibliography}

\end{document}

\maketitle
	
	This supplement presents intermediate results that are required for some proofs in the main text.
	In particular, in \cref{apend:decomposition_finite_intermediate_results} we derive \cref{lemma:stirling_sums_multi_2,lemma:stirling_sums_alternated_multi_2,lemma:stirling_sums_multi_1,lemma:stirling_sums_alternated_multi_1,lemma:basis_comp_multi_expansion,lemma:coupl_comp_multi_expansion,lemma:comb_sum_coupling_to_basis_major}, which are used to prove \cref{thm:decomposition_finite}. \Cref{apend:stirling_generalized} derives \cref{thm:sum_prod_comb_stirstg} which is used to prove \cref{thm:oracle_basis_relation_oracle_to basis}.
	\appendix
	\section{Intermediate results used in \cref{thm:decomposition_finite}}
	\label{apend:decomposition_finite_intermediate_results}
	\begin{lemma}
		For every $ n\in\mathbb{Z} $, $ k\in\mathbb{N} $, we have that
		\begin{align}
		\sum_{
			\substack{
				\mathbf{m} \geq \mathbf{1}_{k}\\
				\vert \mathbf{m} \vert = n
			}
		}
		\frac{n}{\prod_{i=1}^{k}m_i}
		=
		\sum_{
			\substack{
				\mathbf{m} \geq \mathbf{1}_{k}\\
				\vert \mathbf{m} \vert = n-1
			}
		}
		\frac{n-1}{\prod_{i=1}^{k}m_i}
		+
		\sum_{
			\substack{
				\mathbf{m} \geq \mathbf{1}_{k-1}\\
				\vert \mathbf{m} \vert = n-1
			}
		}
		\frac{k}{\prod_{i=1}^{k-1}m_i}.
		\label{eq:s1_my_express_recursive}
		\end{align}
		\label{lemma:s1_my_express_recursive}
		\hfill$ \square $
	\end{lemma}
	\begin{proof}
		Using the fact that $ \vert \mathbf{m} \vert = n $, we can rewrite the left hand side into
		\begin{align*}
		\sum_{
			\substack{
				\mathbf{m} \geq \mathbf{1}_{k}\\
				\vert \mathbf{m} \vert = n
			}
		}
		\frac{n}{\prod_{i=1}^{k}m_i}
		=
		\sum_{
			\substack{
				\mathbf{m} \geq \mathbf{1}_{k}\\
				\vert \mathbf{m} \vert = n
			}
		}
		\sum_{j = 1}^{k}
		\frac{m_j}{\prod_{i=1}^{k}m_i}
		=
		\sum_{j = 1}^{k}
		\sum_{
			\substack{
				\mathbf{m} \geq \mathbf{1}_{k}\\
				\vert \mathbf{m} \vert = n
			}
		}
		\frac{1}{\prod_{
				\substack{
					i=1\\
					i\neq j}
			}^{k}m_i}.
		\end{align*}
		Note that $ m_j \geq 1 $, therefore, the quotients are always well-defined. We perform a change of variables by removing the entry $ j $ of the multi-index $ \mathbf{m} $. The other conditions in the sum have to be adjusted accordingly, in particular, $ \vert \mathbf{m} \vert = n $ becomes $ \vert \mathbf{m} \vert \leq n-1 $, that is,
		\begin{align*}
		\sum_{j = 1}^{k}
		\sum_{
			\substack{
				\mathbf{m} \geq \mathbf{1}_{k}\\
				\vert \mathbf{m} \vert = n
			}
		}
		\frac{1}{\prod_{
				\substack{
					i=1\\
					i\neq j}
			}^{k}m_i
		}
		=
		\sum_{j = 1}^{k}
		\sum_{
			\substack{
				\mathbf{m} \geq \mathbf{1}_{k-1}\\
				\vert \mathbf{m} \vert \leq n-1
			}
		}
		\frac{1}{\prod_{i=1}^{k-1}m_i
		}
		=
		\sum_{
			\substack{
				\mathbf{m} \geq \mathbf{1}_{k-1}\\
				\vert \mathbf{m} \vert \leq n-1
			}
		}
		\frac{k}{\prod_{i=1}^{k-1}m_i
		}.
		\end{align*}
		In summary, we have proven that
		\begin{align*}
		\sum_{
			\substack{
				\mathbf{m} \geq \mathbf{1}_{k}\\
				\vert \mathbf{m} \vert = n
			}
		}
		\frac{n}{\prod_{i=1}^{k}m_i}
		=
		\sum_{
			\substack{
				\mathbf{m} \geq \mathbf{1}_{k-1}\\
				\vert \mathbf{m} \vert \leq n-1
			}
		}
		\frac{k}{\prod_{i=1}^{k-1}m_i
		}.
		\end{align*}
		This expression is valid for every $ n\in\mathbb{Z} $, therefore, changing variable $ n $ into $ n-1 $, one obtains
		\begin{align*}
		\sum_{
			\substack{
				\mathbf{m} \geq \mathbf{1}_{k}\\
				\vert \mathbf{m} \vert = n-1
			}
		}
		\frac{n-1}{\prod_{i=1}^{k}m_i}
		=
		\sum_{
			\substack{
				\mathbf{m} \geq \mathbf{1}_{k-1}\\
				\vert \mathbf{m} \vert \leq n-2
			}
		}
		\frac{k}{\prod_{i=1}^{k-1}m_i
		}.
		\end{align*}
		These two equations can be merged in the following way
		\begin{align*}
		\sum_{
			\substack{
				\mathbf{m} \geq \mathbf{1}_{k}\\
				\vert \mathbf{m} \vert = n
			}
		}
		\frac{n}{\prod_{i=1}^{k}m_i}
		&=
		\sum_{
			\substack{
				\mathbf{m} \geq \mathbf{1}_{k-1}\\
				\vert \mathbf{m} \vert \leq n-1
			}
		}
		\frac{k}{\prod_{i=1}^{k-1}m_i
		}
		\\
		&=
		\sum_{
			\substack{
				\mathbf{m} \geq \mathbf{1}_{k-1}\\
				\vert \mathbf{m} \vert \leq n-2
			}
		}
		\frac{k}{\prod_{i=1}^{k-1}m_i
		}
		+
		\sum_{
			\substack{
				\mathbf{m} \geq \mathbf{1}_{k-1}\\
				\vert \mathbf{m} \vert = n-1
			}
		}
		\frac{k}{\prod_{i=1}^{k-1}m_i
		}
		\\
		&=
		\sum_{
			\substack{
				\mathbf{m} \geq \mathbf{1}_{k}\\
				\vert \mathbf{m} \vert = n-1
			}
		}
		\frac{n-1}{\prod_{i=1}^{k}m_i}
		+
		\sum_{
			\substack{
				\mathbf{m} \geq \mathbf{1}_{k-1}\\
				\vert \mathbf{m} \vert = n-1
			}
		}
		\frac{k}{\prod_{i=1}^{k-1}m_i
		},
		\end{align*}
		which concludes the proof.
	\end{proof}
	\begin{lemma}
		For every $ n\in\mathbb{Z} $, $ k\in\mathbb{N} $, we have that
		\begin{align}
		\sum_{
			\substack{
				\mathbf{m} \geq \mathbf{1}_{k}\\
				\vert \mathbf{m} \vert = n
			}
		}
		\frac{n}{\prod_{i=1}^{k}m_i!}
		=
		\sum_{
			\substack{
				\mathbf{m} \geq \mathbf{1}_{k}\\
				\vert \mathbf{m} \vert = n-1
			}
		}
		\frac{k}{\prod_{i=1}^{k}m_i!}
		+
		\sum_{
			\substack{
				\mathbf{m} \geq \mathbf{1}_{k-1}\\
				\vert \mathbf{m} \vert = n-1
			}
		}
		\frac{k}{\prod_{i=1}^{k-1}m_i!}.
		\label{eq:s2_my_express_recursive}
		\end{align}
		\label{lemma:s2_my_express_recursive}
		\hfill$ \square $
	\end{lemma}
	\begin{proof}
		Using the fact that $ \vert \mathbf{m} \vert = n $, we can rewrite the left hand side into
		\begin{align*}
		\sum_{
			\substack{
				\mathbf{m} \geq \mathbf{1}_{k}\\
				\vert \mathbf{m} \vert = n
			}
		}
		\frac{n}{\prod_{i=1}^{k}m_i!}
		=
		\sum_{
			\substack{
				\mathbf{m} \geq \mathbf{1}_{k}\\
				\vert \mathbf{m} \vert = n
			}
		}
		\sum_{j = 1}^{k}
		\frac{m_j}{\prod_{i=1}^{k}m_i!}
		=
		\sum_{j = 1}^{k}
		\sum_{
			\substack{
				\mathbf{m} \geq \mathbf{1}_{k}\\
				\vert \mathbf{m} \vert = n
			}
		}
		\frac{1}{
			(m_j-1)!
			\prod_{
				\substack{
					i=1\\
					i\neq j}
			}^{k}m_i!}.
		\end{align*}
		Note that $ m_j \geq 1 $, therefore, the quotients are always well-defined. We now split the multi-index $ \mathbf{m} $ according to whenever $ m_j \geq 2 $ or $ m_j = 1 $, that is,
		\begin{align*}
		\sum_{j = 1}^{k}
		\left(
		\sum_{
			\substack{
				\mathbf{m} \geq \mathbf{1}_{k}\\
				m_j\geq 2\\
				\vert \mathbf{m} \vert = n
			}
		}
		\frac{1}{
			(m_j-1)!
			\prod_{
				\substack{
					i=1\\
					i\neq j}
			}^{k}m_i!}
		+\sum_{
			\substack{
				\mathbf{m} \geq \mathbf{1}_{k}\\
				m_j = 1\\
				\vert \mathbf{m} \vert = n
			}
		}
		\frac{1}{
			\prod_{
				\substack{
					i=1\\
					i\neq j}
			}^{k}m_i!}
		\right).
		\end{align*}
		On the first inner sum, we perform a change of variables so that $ m_j-1 $ becomes $ m_j $. This means that the condition $ m_j \geq 2 $ becomes $ m_j \geq 1 $. Therefore, we can compress the conditions $ \mathbf{m} \geq \mathbf{1}_{k} $ and $ m_j \geq 2 $ in the old coordinates into just $ \mathbf{m} \geq \mathbf{1}_{k} $ in the new ones.\\
		On the second inner sum we perform a change of variables by removing the entry $ j $ of the multi-index $ \mathbf{m} $. For both sums, the conditions $ \vert \mathbf{m} \vert = n $ become $ \vert \mathbf{m} \vert = n-1 $. That is,
		\begin{align*}
		\sum_{
			\substack{
				\mathbf{m} \geq \mathbf{1}_{k}\\
				\vert \mathbf{m} \vert = n
			}
		}
		\frac{n}{\prod_{i=1}^{k}m_i!}
		&=
		\sum_{j = 1}^{k}
		\left(
		\sum_{
			\substack{
				\mathbf{m} \geq \mathbf{1}_{k}\\
				\vert \mathbf{m} \vert = n-1
			}
		}
		\frac{1}{
			\prod_{i=1}^{k}m_i!}
		+
		\sum_{
			\substack{
				\mathbf{m} \geq \mathbf{1}_{k-1}\\
				\vert \mathbf{m} \vert = n-1
			}
		}
		\frac{1}{
			\prod_{i=1}^{k-1}m_i!
		}
		\right)
		\\
		&=
		\sum_{
			\substack{
				\mathbf{m} \geq \mathbf{1}_{k}\\
				\vert \mathbf{m} \vert = n-1
			}
		}
		\frac{k}{
			\prod_{i=1}^{k}m_i!}
		+
		\sum_{
			\substack{
				\mathbf{m} \geq \mathbf{1}_{k-1}\\
				\vert \mathbf{m} \vert = n-1
			}
		}
		\frac{k}{
			\prod_{i=1}^{k-1}m_i!
		},
		\end{align*}
		which concludes the proof.
	\end{proof}
	We are now ready to introduce our new formulas for the Stirling numbers of the first and second kinds.
	\begin{theorem}
		The unsigned Stirling numbers of the first kind, $ \stirst{n}{k} $, with $ n,k \geq 0$ are given by
		\begin{align}
		\stirst{n}{k}
		=
		\frac{n!}{k!}
		\sum_{
			\substack{
				\mathbf{m} \geq \mathbf{1}_{k}\\
				\vert \mathbf{m} \vert = n
			}
		}
		\frac{1}{\prod_{i=1}^{k}m_i}. 
		\label{eq:stirling_sums_1}
		\end{align}
		\label{lemma:stirling_sums_1}
		\hfill$ \square $
	\end{theorem}
	\begin{proof}
		We have to prove that the right hand side of \cref{eq:stirling_sums_1} has the properties of \cref{defi:stirling_1}.\\
		Consider $ n,k = 0 $, the initial condition is satisfied since the only valid argument of the sum is the 0-tuple. For both $ n>0 $, $ k = 0 $ and $ n=0 $, $ k > 0 $ there are no valid arguments in the sum, which results in zero and those initial conditions are also satisfied.\\
		For the remaining values, $ n,k > 0 $, we have to show that they follow the recurrence relation, that is,
		\begin{align*}
		\frac{n!}{k!}
		\sum_{
			\substack{
				\mathbf{m} \geq \mathbf{1}_{k}\\
				\vert \mathbf{m} \vert = n
			}
		}
		\frac{1}{\prod_{i=1}^{k}m_i}
		=
		(n-1)
		\frac{(n-1)!}{k!}
		\sum_{
			\substack{
				\mathbf{m} \geq \mathbf{1}_{k}\\
				\vert \mathbf{m} \vert = n-1
			}
		}
		\frac{1}{\prod_{i=1}^{k}m_i}
		+
		\frac{(n-1)!}{(k-1)!}
		\sum_{
			\substack{
				\mathbf{m} \geq \mathbf{1}_{k-1}\\
				\vert \mathbf{m} \vert = n-1
			}
		}
		\frac{1}{\prod_{i=1}^{k-1}m_i}. 
		\end{align*}
		Multiplying both sides by $ \frac{k!}{(n-1)!} $ we obtain
		\begin{align*}
		\sum_{
			\substack{
				\mathbf{m} \geq \mathbf{1}_{k}\\
				\vert \mathbf{m} \vert = n
			}
		}
		\frac{n}{\prod_{i=1}^{k}m_i}
		=
		\sum_{
			\substack{
				\mathbf{m} \geq \mathbf{1}_{k}\\
				\vert \mathbf{m} \vert = n-1
			}
		}
		\frac{n-1}{\prod_{i=1}^{k}m_i}
		+
		\sum_{
			\substack{
				\mathbf{m} \geq \mathbf{1}_{k-1}\\
				\vert \mathbf{m} \vert = n-1
			}
		}
		\frac{k}{\prod_{i=1}^{k-1}m_i} ,
		\end{align*}
		which is true from \cref{lemma:s1_my_express_recursive}.
	\end{proof}
	\begin{theorem}
		The Stirling numbers of the second kind, $ \stirnd{n}{k} $, with $ n,k \geq 0$ are given by
		\begin{align}
		\stirnd{n}{k}
		=
		\frac{n!}{k!}
		\sum_{
			\substack{
				\mathbf{m} \geq \mathbf{1}_{k}\\
				\vert \mathbf{m} \vert = n
			}
		}
		\frac{1}{\prod_{i=1}^{k}m_i!} .
		\label{eq:stirling_sums_2}
		\end{align}
		\label{lemma:stirling_sums_2}
		\hfill$ \square $
	\end{theorem}
	\begin{proof}
		We have to prove that the right hand side of \cref{eq:stirling_sums_2} has the properties of \cref{defi:stirling_2}.\\
		Consider $ n,k = 0 $, the initial condition is satisfied since the only valid argument of the sum is the 0-tuple. For both $ n>0 $, $ k = 0 $ and $ n=0 $, $ k > 0 $ there are no valid arguments in the sum, which results in zero and those initial conditions are also satisfied.\\
		For the remaining values, $ n,k > 0 $, we have to show that they follow the recurrence relation, that is,
		\begin{align*}
		\frac{n!}{k!}
		\sum_{
			\substack{
				\mathbf{m} \geq \mathbf{1}_{k}\\
				\vert \mathbf{m} \vert = n
			}
		}
		\frac{1}{\prod_{i=1}^{k}m_i!}
		=
		k
		\frac{(n-1)!}{k!}
		\sum_{
			\substack{
				\mathbf{m} \geq \mathbf{1}_{k}\\
				\vert \mathbf{m} \vert = n-1
			}
		}
		\frac{1}{\prod_{i=1}^{k}m_i!}
		+
		\frac{(n-1)!}{(k-1)!}
		\sum_{
			\substack{
				\mathbf{m} \geq \mathbf{1}_{k-1}\\
				\vert \mathbf{m} \vert = n-1
			}
		}
		\frac{1}{\prod_{i=1}^{k-1}m_i!}. 
		\end{align*}
		Multiplying both sides by $ \frac{k!}{(n-1)!} $ we obtain
		\begin{align*}
		\sum_{
			\substack{
				\mathbf{m} \geq \mathbf{1}_{k}\\
				\vert \mathbf{m} \vert = n
			}
		}
		\frac{n}{\prod_{i=1}^{k}m_i!}
		=
		\sum_{
			\substack{
				\mathbf{m} \geq \mathbf{1}_{k}\\
				\vert \mathbf{m} \vert = n-1
			}
		}
		\frac{k}{\prod_{i=1}^{k}m_i!}
		+
		\sum_{
			\substack{
				\mathbf{m} \geq \mathbf{1}_{k-1}\\
				\vert \mathbf{m} \vert = n-1
			}
		}
		\frac{k}{\prod_{i=1}^{k-1}m_i!}, 
		\end{align*}
		which is true from \cref{lemma:s2_my_express_recursive}.
	\end{proof}
	\begin{proof}[Proof of \cref{lemma:stirling_sums_multi_1}]
		We can break the multi-index $ \overline{\mathbf{m}} $ with $ \vert \mathbf{m} \vert $ elements into the set of multi-indexes $ (\overline{\mathbf{m}}^i) $, with $ 1\leq i\leq k $, such that each $ \overline{\mathbf{m}}^i $ has $ m_i $ elements and is associated with $ M_i $. In particular, this means that $ \overline{\mathbf{m}}\mathbf{m} = \mathbf{M} $ becomes $ \vert \overline{\mathbf{m}}^i \vert =  M_i $, for every $ i $ with $ 1\leq i\leq k $. Using this, the left hand side of \cref{eq:stirling_sums_multi_1} can be written as the product
		\begin{align*}
		\prod_{i=1}^{k}
		\left(
		\sum_{ 
			\substack{
				\overline{\mathbf{m}}^i \geq \mathbf{1}_{m_i}
				\\
				\vert \overline{\mathbf{m}}^i \vert = M_i
			}
		}
		\frac{1}{\prod_{j=1}^{m_i} \overline{m}^i_j}
		\right).
		\end{align*}
		The result comes directly from applying \cref{lemma:stirling_sums_1}.
	\end{proof}
	\begin{proof}[Proof of \cref{lemma:stirling_sums_multi_2}]
		Using the exact same approach as in the proof of \cref{lemma:stirling_sums_multi_1}, the left hand side of \cref{eq:stirling_sums_multi_2} can be written as the product
		\begin{align*}
		\prod_{i=1}^{k}
		\left(
		\sum_{ 
			\substack{
				\overline{\mathbf{m}}^i \geq \mathbf{1}_{m_i}
				\\
				\vert \overline{\mathbf{m}}^i \vert = M_i
			}
		}
		\frac{1}{\prod_{j=1}^{m_i} \overline{m}^i_j !}
		\right).
		\end{align*}
		The result comes directly from applying \cref{lemma:stirling_sums_2}.
	\end{proof}
	\begin{lemma}
		For $ n \geq 0 $, we have that
		\begin{align}
		\sum_{k\geq 1} (-1)^k  \stirst{n}{k} 
		=
		\begin{cases}
		-1 & \text{if $n=1$},\\
		0 & \text{otherwise}.
		\end{cases}
		\end{align}
		\label{lemma:stirling_alternated_1}
		\hfill$ \square $
	\end{lemma}
	\begin{proof}
		The first cases can easily be verified from inspection.
		Note that the sum is finite since $ \stirst{n}{k} = 0 $ when $ k>n $. For $ n=0 $ this is a zero sum and for $ n=1 $ we have $ (-1)\stirst{1}{1} = -1 $.\\
		The remaining terms are proven by induction. Assume the sum to be zero for $ n > 1 $. We expand $ \stirst{n+1}{k} $ according to its recurrence relation
		\begin{align*}
		\sum_{k\geq 1} (-1)^k \stirst{n+1}{k}  
		&=
		\sum_{k\geq 1} (-1)^k 
		\left[
		n\stirst{n}{k} 
		+
		\stirst{n}{k-1} 
		\right]
		\\
		&=
		n
		\sum_{k\geq 1} (-1)^k
		\stirst{n}{k} 
		+ 
		\sum_{k\geq 1} (-1)^k
		\stirst{n}{k-1} 
		\\
		&=
		0.
		\end{align*}
		The base case of the induction process, $ n=2 $, is simply $ \stirst{2}{2} - \stirst{2}{1} = 1 -1 = 0 $.
	\end{proof}
	\begin{lemma}
		For $ n \geq 0 $, we have that
		\begin{align}
		\sum_{k\geq 1} (-1)^k
		(k-1)!
		\stirnd{n}{k} 
		=
		\begin{cases}
		-1 & \text{if $n=1$},\\
		0 & \text{otherwise}.
		\end{cases}
		\label{eq:stirling_alternated_2}
		\end{align}
		\label{lemma:stirling_alternated_2}
		\hfill$ \square $
	\end{lemma}
	\begin{proof}
		Firstly, note that for $ n=0 $, the sum is trivially $ 0 $.
		Consider now $ n>0 $. Then, from \cref{defi:stirling_2}
		\begin{align*}
		\sum_{k\geq 1} (-1)^k  (k-1)!\stirnd{n}{k} 
		=
		\sum_{k\geq 1} (-1)^k  k!\stirnd{n-1}{k}
		+
		\sum_{k\geq 1} (-1)^k  (k-1)!\stirnd{n-1}{k-1}. 
		\end{align*}
		We perform the change of variables $ \overline{k} = k-1 $ on the second sum of the right hand side, which gives
		\begin{align*}
		\sum_{k\geq 1} (-1)^k  (k-1)!\stirnd{n}{k} 
		&=
		\sum_{k\geq 1} (-1)^k  k!\stirnd{n-1}{k}
		+
		\sum_{\overline{k}\geq 0} (-1)^{\overline{k}+1}  \overline{k}!\stirnd{n-1}{\overline{k}} 
		\\
		&=
		-\stirnd{n-1}{0} ,
		\end{align*}
		which is $ -1 $ for $ n = 1 $ and $ 0 $ for $ n>1 $.
	\end{proof}
	\begin{proof}[Proof of \cref{lemma:stirling_sums_alternated_multi_1}]
		We can write the left hand side of \cref{eq:stirling_sums_alternated_multi_1} as the product
		\begin{align*}
		\prod_{i=1}^{k}
		\sum_{m_i \geq 1}
		(-1)^{m_i}
		\stirst{M_i}{m_i}.
		\end{align*}
		Note that although the outer sum of \cref{eq:stirling_sums_alternated_multi_1} looks infinite, it only has a finite number of non-zero elements. Therefore, there are no convergence issues when we do this rearrangement.\\
		Consider $ k>0 $. If $ \mathbf{M} \neq \mathbf{1}_{k} $, then there will be at least one $ i $ with $ 1\leq i\leq k $ such that $ M_i \neq 1 $. From \cref{lemma:stirling_alternated_1}, that term will be zero, which means that the whole product is zero. For $ \mathbf{M} = \mathbf{1}_{k} $ the result is immediate. In the case $ k=0 $ we have on the left hand side a sum over one valid index (the 0-tuple) of an empty product, which results in $ 1 = (-1)^0 $.
	\end{proof}
	\begin{proof}[Proof of \cref{lemma:stirling_sums_alternated_multi_2}]
		We can write the left hand side of \cref{eq:stirling_sums_alternated_multi_2} as the product
		\begin{align*}
		\prod_{i = 1}^{k}
		\sum_{m_i \geq 1}
		(-1)^{m_i}
		(m_i -1)!
		\stirnd{M_i}{m_i}.
		\end{align*}
		Using the exact same approach as in the proof of \cref{lemma:stirling_sums_alternated_multi_1}, the result is straightforward from \cref{lemma:stirling_alternated_2}.
	\end{proof}
	\begin{proof}[Proof of \cref{lemma:basis_comp_multi_expansion}]
		The proof is by induction. Assume this to be satisfied for $ m_{12} = a-1 \geq 0$. Then, for the case $ m_{12} =a $, the left hand side of \cref{eq:basis_component_multiplicity_expansion} can be written as
		\begin{align*}
		\leftidx{^b}{}f_i^{\mathbf{k}}
		\left(x;
		\begin{bmatrix}
		a-1\\
		1\\
		\overline{\mathbf{m}}
		\end{bmatrix},
		\begin{bmatrix}
		w_{j_1} \| w_{j_2}\\
		w_{j_1} \| w_{j_2}\\
		\mathbf{w}_{\overline{\mathbf{s}}}
		\end{bmatrix},
		\begin{bmatrix}
		x_{j_{12}} \\
		x_{j_{12}} \\
		\mathbf{x}_{\overline{\mathbf{s}}}
		\end{bmatrix}
		\right).
		\end{align*}
		From assumption, this can be expanded into
		\begin{align*}
		\sum_{
			\substack{
				m_1,m_2 \geq 0\\
				m_1 + m_2 = a-1	
			}
		}
		\binom{a-1}{m_1,m_2}
		\leftidx{^b}{}f_i^{\mathbf{k}}
		\left(x;
		\begin{bmatrix}
		m_{1}\\
		m_{2}\\
		1\\
		\overline{\mathbf{m}}
		\end{bmatrix}
		,
		\begin{bmatrix}
		w_{j_1}\\
		w_{j_2}\\
		w_{j_1} \| w_{j_2}\\ 
		\mathbf{w}_{\overline{\mathbf{s}}}
		\end{bmatrix}
		,
		\begin{bmatrix}
		x_{j_{12}} \\ 
		x_{j_{12}} \\
		x_{j_{12}} \\
		\mathbf{x}_{\overline{\mathbf{s}}}
		\end{bmatrix}
		\right).
		\end{align*}
		Using \cref{thm:basis_f_additive} of \cref{defi:decomp_pure_basis} in order to expand over the weight $ w_{j_1} \| w_{j_2} $, we get
		\begin{align*}
		\sum_{
			\substack{
				m_1,m_2 \geq 0\\
				m_1 + m_2 = a-1	
			}
		}
		\binom{a-1}{m_1,m_2}
		&\left[
		\leftidx{^b}{}f_i^{\mathbf{k}}
		\left(x;
		\begin{bmatrix}
		m_{1}+1\\
		m_{2}\\
		\overline{\mathbf{m}}
		\end{bmatrix}
		,
		\begin{bmatrix}
		w_{j_1}\\
		w_{j_2}\\
		\mathbf{w}_{\overline{\mathbf{s}}}
		\end{bmatrix}
		,
		\begin{bmatrix}
		x_{j_{12}} \\ 
		x_{j_{12}} \\
		\mathbf{x}_{\overline{\mathbf{s}}}
		\end{bmatrix}
		\right)
		\right.
		\\
		&\quad+
		\left.
		\leftidx{^b}{}f_i^{\mathbf{k}}
		\left(x;
		\begin{bmatrix}
		m_{1}\\
		m_{2}+1\\
		\overline{\mathbf{m}}
		\end{bmatrix}
		,
		\begin{bmatrix}
		w_{j_1}\\
		w_{j_2}\\
		\mathbf{w}_{\overline{\mathbf{s}}}
		\end{bmatrix}
		,
		\begin{bmatrix}
		x_{j_{12}} \\ 
		x_{j_{12}} \\
		\mathbf{x}_{\overline{\mathbf{s}}}
		\end{bmatrix}
		\right)
		\right].
		\end{align*}
		We split this sum into two according to the two terms. Furthermore, we change variables such that $ m_1 +1 $ becomes $ m_1 $ on the first sum and $ m_2+1 $ becomes $ m_2 $ on the second sum. This gives us
		\begin{align*}
		\left[
		\sum_{
			\substack{
				m_1 \geq 1\\
				m_2 \geq 0\\
				m_1 + m_2 = a	
			}
		}
		\binom{a-1}{m_1-1,m_2}
		+
		\sum_{
			\substack{
				m_1 \geq 0\\
				m_2 \geq 1\\
				m_1 + m_2 = a	
			}
		}
		\binom{a-1}{m_1,m_2-1}
		\right]
		\leftidx{^b}{}f_i^{\mathbf{k}}
		\left(x;
		\begin{bmatrix}
		m_{1}\\
		m_{2}\\
		\overline{\mathbf{m}}
		\end{bmatrix}
		,
		\begin{bmatrix}
		w_{j_1}\\
		w_{j_2}\\
		\mathbf{w}_{\overline{\mathbf{s}}}
		\end{bmatrix}
		,
		\begin{bmatrix}
		x_{j_{12}} \\ 
		x_{j_{12}} \\
		\mathbf{x}_{\overline{\mathbf{s}}}
		\end{bmatrix}
		\right).
		\end{align*}
		These two sums can be unified into one by appending cases that correspond to zero terms until their index sets match. In particular, in the first sum we can freely append the cases $ m_1 = 0$ and in the second sum we can append the cases $ m_2 = 0 $. Then, we end up with a single sum over the index set $ m_1,m_2 \geq 0 $, with $ m_1+m_2 = a $.\\
		Then, from the fact that $ \binom{n}{m_1,m_2} = \binom{n-1}{m_1-1,m_2} + \binom{n-1}{m_1,m_2-1} $, our expression simplifies into the right hand side of \cref{eq:basis_component_multiplicity_expansion} and the result applies to $ m_{12}=a $.\\		
		In base case $ m_{12} = 0$, the only valid term in the sum is the one indexed with $ m_1,m_2 = 0 $. It is clear that for this case the equality is satisfied, which concludes the proof.
	\end{proof}
	\begin{proof}[Proof of \cref{lemma:coupl_comp_multi_expansion}]
		The proof is by induction. Assume this to be satisfied for $m_{12} = a-1 \geq 0 $. Then, for the case $ m_{12} = a $, the left hand side of	\cref{eq:coupl_component_multiplicity_expansion} can be written as
		\begin{align*}
		f_i^{
			\overline{\mathbf{k}}
			+ a \, 1_j
		}
		\left(x;
		\begin{bmatrix}
		a-1\\
		1\\
		\overline{\mathbf{m}}
		\end{bmatrix}
		,
		\begin{bmatrix}
		w_{j_1} \| w_{j_2}\\
		w_{j_1} \| w_{j_2}\\ 
		\mathbf{w}_{\overline{\mathbf{s}}}
		\end{bmatrix}
		,
		\begin{bmatrix}
		x_{j_{12}}\\
		x_{j_{12}}\\
		\mathbf{x}_{\overline{\mathbf{s}}}
		\end{bmatrix}
		\right).
		\end{align*}
		From assumption, this can be expanded into
		\begin{align*}
		\sum_{ 
			\substack{
				m_1,m_2 \geq 0\\
				m_1,m_2 \leq a-1\\
				m_1 + m_2 \geq a-1
			}
		}
		B(m_1,m_2,a-1)
		f_i^{
			\overline{\mathbf{k}}
			+ (m_{1}+m_{2}+1) \, 1_j 
		}
		\left(x;
		\begin{bmatrix}
		m_{1}\\
		m_{2}\\
		1\\
		\overline{\mathbf{m}}
		\end{bmatrix}
		,
		\begin{bmatrix}
		w_{j_1} \\
		w_{j_2}\\
		w_{j_1} \| w_{j_2}\\ 
		\mathbf{w}_{\overline{\mathbf{s}}}
		\end{bmatrix}
		,
		\begin{bmatrix}
		x_{j_{12}} \\ 
		x_{j_{12}} \\
		x_{j_{12}} \\ 
		\mathbf{x}_{\overline{\mathbf{s}}}
		\end{bmatrix}
		\right).
		\end{align*}
		Using \cref{thm:decomp_f_dependence} of \cref{thm:decomp_f} in order to expand over the weight $ w_{j_1} \| w_{j_2} $, we get
		\begin{align*}
		\sum_{ 
			\substack{
				m_1,m_2 \geq 0\\
				m_1,m_2 \leq a-1\\
				m_1 + m_2 \geq a-1
			}
		}
		B(m_1,m_2,a-1)
		&\left[
		f_i^{
			\overline{\mathbf{k}}
			+ (m_{1}+m_{2}+1) \, 1_j 
		}
		\left(x;
		\begin{bmatrix}
		m_{1}+1\\
		m_{2}\\
		\overline{\mathbf{m}}
		\end{bmatrix}
		,
		\begin{bmatrix}
		w_{j_1} \\
		w_{j_2}\\
		\mathbf{w}_{\overline{\mathbf{s}}}
		\end{bmatrix}
		,
		\begin{bmatrix}
		x_{j_{12}} \\ 
		x_{j_{12}} \\
		\mathbf{x}_{\overline{\mathbf{s}}}
		\end{bmatrix}
		\right)
		\right.
		\\
		&\quad+
		f_i^{
			\overline{\mathbf{k}}
			+ (m_{1}+m_{2}+1) \, 1_j 
		}
		\left(x;
		\begin{bmatrix}
		m_{1}\\
		m_{2}+1\\
		\overline{\mathbf{m}}
		\end{bmatrix}
		,
		\begin{bmatrix}
		w_{j_1} \\
		w_{j_2}\\
		\mathbf{w}_{\overline{\mathbf{s}}}
		\end{bmatrix}
		,
		\begin{bmatrix}
		x_{j_{12}} \\ 
		x_{j_{12}} \\
		\mathbf{x}_{\overline{\mathbf{s}}}
		\end{bmatrix}
		\right)
		\\
		&\quad+
		\left.
		f_i^{
			\overline{\mathbf{k}}
			+ (m_{1}+m_{2}+2) \, 1_j 
		}
		\left(x;
		\begin{bmatrix}
		m_{1}+1\\
		m_{2}+1\\
		\overline{\mathbf{m}}
		\end{bmatrix}
		,
		\begin{bmatrix}
		w_{j_1} \\
		w_{j_2}\\
		\mathbf{w}_{\overline{\mathbf{s}}}
		\end{bmatrix}
		,
		\begin{bmatrix}
		x_{j_{12}} \\ 
		x_{j_{12}} \\
		\mathbf{x}_{\overline{\mathbf{s}}}
		\end{bmatrix}
		\right)
		\right].
		\end{align*}
		We split this sum into three according to the three terms. Furthermore, we change variables such that $ m_1 +1 $ becomes $ m_1 $ on the first sum, $ m_2+1 $ becomes $ m_2 $ on the second sum and we apply both changes on the third sum. This gives us
		\begin{align*}
		&
		\left[
		\sum_{ 
			\substack{
				m_1 \geq 1\\
				m_2 \geq 0\\
				m_1 \leq a\\
				m_2 \leq a-1\\
				m_1 + m_2 \geq a
			}
		}
		B(m_1-1,m_2,a-1)
		+
		\sum_{ 
			\substack{
				m_1 \geq 0\\
				m_2 \geq 1\\
				m_1 \leq a-1\\
				m_2 \leq a\\
				m_1 + m_2 \geq a
			}
		}
		B(m_1,m_2-1,a-1)
		\right.
		\\
		&+
		\left.
		\sum_{ 
			\substack{
				m_1,m_2 \geq 1\\
				m_1,m_2 \leq a\\
				m_1 + m_2 \geq a+1
			}
		}
		B(m_1-1,m_2-1,a-1)
		\right]
		f_i^{
			\overline{\mathbf{k}}
			+ (m_{1}+m_{2}) \, 1_j 
		}
		\left(x;
		\begin{bmatrix}
		m_{1}\\
		m_{2}\\
		\overline{\mathbf{m}}
		\end{bmatrix}
		,
		\begin{bmatrix}
		w_{j_1} \\
		w_{j_2}\\
		\mathbf{w}_{\overline{\mathbf{s}}}
		\end{bmatrix}
		,
		\begin{bmatrix}
		x_{j_{12}} \\ 
		x_{j_{12}} \\
		\mathbf{x}_{\overline{\mathbf{s}}}
		\end{bmatrix}
		\right).
		\end{align*}
		These three sums can be unified into one by appending cases that correspond to zero terms until their index sets match. In particular, in the first sum we can freely append the cases $ m_1 = 0$ and the cases $ m_2 =a $. Similarly, in the second sum, we can append the cases $ m_1 = a $ and the cases $ m_2 = 0 $. Finally, in the last sum, we can append the cases $ m_1 = 0 $, the cases $ m_2 =0 $ and the cases $ m_1+m_2 = a $. Then, we end up with a single sum over the index set $ m_1,m_2 \geq 0 $, with $ m_1,m_2 \leq a$, and $ m_1+m_2 \geq a  $.
		Note that
		\begin{align*}
		B(m_1-1,m_2,a-1)
		+
		B(m_1,m_2-1,a-1)
		+
		B(m_1-1,m_2-1,a-1),
		\end{align*}
		which is equal to
		\begin{align*}
		\binom{a-1}{
			a - m_1,
			a - m_2-1,
			m_1 + m_2 - a}
		&+
		\binom{a-1}{
			a - m_1-1,
			a - m_2,
			m_1 + m_2 - a}
		\\
		&+
		\binom{a-1}{
			a - m_1,
			a - m_2,
			m_1 + m_2 - a-1},
		\end{align*}
		gives us
		\begin{align*}
		\binom{a}{a - m_1,a - m_2,m_1 + m_2 - a} = B(m_1,m_2,a),
		\end{align*}
		from the fact that $ \binom{n}{m_1,m_2,m_3} = \binom{n-1}{m_1-1,m_2,m_3} + \binom{n-1}{m_1,m_2-1,m_3} +
		\binom{n-1}{m_1,m_2,m_3-1} $. Therefore, our expression simplifies into the right hand side of \cref{eq:coupl_component_multiplicity_expansion} and the result applies to $ m_{12} = a $.\\		
		In base case $ m_{12} = 0$, the only valid term in the sum is the one indexed with $ m_1,m_2 = 0 $. It is clear that for this case the equality is satisfied, which concludes the proof.
	\end{proof}
	\begin{lemma}
		For every $ m_1,m_2\in\mathbb{N} $, we have that
		\begin{align}
		\sum_{n = 0}^{m_1} (-1)^{n}\binom{m_1}{n} \binom{n +m_2-1}{m_1-1}
		=0.
		\label{eq:comb_sum_coupling_to_basis_minor}
		\end{align}
		\label{lemma:comb_sum_coupling_to_basis_minor}
		\hfill$ \square $
	\end{lemma}
	\begin{proof}
		The proof is by induction on $ m_1 $. Assume \cref{eq:comb_sum_coupling_to_basis_minor} is valid for a particular $ m_1\geq 1 $. Using $ \binom{n-1}{k-1} = \binom{n}{k} - \binom{n-1}{k} $, we expand \cref{eq:comb_sum_coupling_to_basis_minor} into
		\begin{align*}
		\sum_{n = 0}^{m_1} (-1)^{n}\binom{m_1}{n} \left[\binom{n +m_2}{m_1}-\binom{n +m_2-1}{m_1}\right]
		&=0
		\\
		\sum_{n = 0}^{m_1} (-1)^{n}\binom{m_1}{n} \binom{n +m_2}{m_1}
		-
		\sum_{n = 0}^{m_1} (-1)^{n}\binom{m_1}{n} \binom{n +m_2-1}{m_1}
		&=0.
		\end{align*}
		On the first sum we change variables so that $ n+1 $ becomes $ n $, which gives us
		\begin{align*}
		\sum_{n = 1}^{m_1+1} (-1)^{n-1}\binom{m_1}{n-1} \binom{n+m_2-1}{m_1}
		-
		\sum_{n = 0}^{m_1} (-1)^{n}\binom{m_1}{n} \binom{n +m_2-1}{m_1}
		&=0.
		\end{align*}
		We can append the case $ n = 0 $ on the first sum and the case $ n=m_1+1 $ on the second, which correspond to zero terms. Then, we can merge the two sums back into
		\begin{align*}
		-
		\sum_{n = 0}^{m_1+1} (-1)^{n}
		\left[
		\binom{m_1}{n-1}
		+
		\binom{m_1}{n}
		\right]
		\binom{n+m_2-1}{m_1}
		&=0
		\\
		-
		\sum_{n = 0}^{m_1+1} (-1)^{n}
		\binom{m_1+1}{n}
		\binom{n+m_2-1}{m_1}
		&=0.
		\end{align*}
		That is, \cref{eq:comb_sum_coupling_to_basis_minor} is also valid for $ m_1+1 $.
		In the base case $ m_1 = 1 $, the sum gives us $ 	1-1=0 $, which concludes de proof.
	\end{proof}
	\begin{proof}[Proof of \cref{lemma:comb_sum_coupling_to_basis_major}]
		For the case $ m_1 = 0,m_2 = 0 $ the sum is empty, therefore the result is $ 0 $. Consider now the case $ m_1\geq 1, m_2 = 0 $. Then, the sum consists of only the term indexed with $ n=m_1 $, which is equal to $ \frac{(-1)^{m_1}}{m_1} $. Note that we only need to study the cases with $ m_1 \leq m_2 $, since the other ones can be trivially deduced thanks to the symmetry of this expression with regard to $ m_1 $ and $ m_2 $. Therefore, to study the remaining cases $ m_1, m_2\geq 1$, we will now consider the case $ 1 \leq m_1 \leq m_2 $, without loss of generality. We multiply the expression by $ \frac{m_1!}{m_1 (m_1-1)!} $, where we have that $ m_1! = m_1 (m_1-1)! $ since we know from assumption that $ m_1 > 0 $. This gives us
		\begin{align*}
		\frac{1}{m_1}
		\sum_{n \geq m_2}^{m_1 + m_2}
		(-1)^n
		\frac{m_1!}{n (m_1-1)!}  \binom{n}{n-m_1,n-m_2, m_1+m_2-n}.
		\end{align*}
		We change variables such that $ n $ becomes $ n+m_2 $
		\begin{align*}
		\frac{1}{m_1}
		\sum_{n \geq 0}^{m_1}
		(-1)^{n+m_2}
		\frac{m_1!}{(n+m_2) (m_1-1)!}  \binom{n+m_2}{n+m_2-m_1,n, m_1-n}.
		\end{align*}
		This can be further simplified as follows
		\begin{align*}
		&\frac{(-1)^{m_2}}{m_1}
		\sum_{n \geq 0}^{m_1}
		(-1)^{n}
		\frac{m_1!}{(n+m_2)(m_1-1)!}
		\frac{(n+m_2)(n+m_2-1)!}{(n+m_2-m_1)! n!(m_1-n)!}
		\\
		&=
		\frac{(-1)^{m_2}}{m_1}
		\sum_{n \geq 0}^{m_1}
		(-1)^{n}
		\frac{m_1!}{ n!(m_1-n)!}
		\frac{(n+m_2-1)!}{(m_1-1)!(n+m_2-m_1)!}
		\\
		&=
		\frac{(-1)^{m_2}}{m_1}
		\sum_{n \geq 0}^{m_1}
		(-1)^{n}
		\binom{m_1}{n}
		\binom{n+m_2-1}{m_1-1},
		\end{align*}
		which is $ 0 $ from \cref{lemma:comb_sum_coupling_to_basis_minor}.
	\end{proof}
	\section{Intermediate results used in \cref{thm:oracle_basis_relation_oracle_to basis}} \label{apend:stirling_generalized}
	\begin{lemma}[Theorem 3 of \cite{broder1984r_sup}]
		The r-Stirling numbers of the first kind are related according to the cross-recurrence formula
		\begin{align}
		\stirstg{r}{n}{k}
		=
		r\stirstg{r+1}{n}{k+1}
		+
		\stirstg{r+1}{n}{k},
		\label{eq:stirstg1_cross_rec}
		\end{align}
		for all $ n>r\geq 0 $ and $ k\geq 0 $.
		\label{lemma:stirstg1_cross_rec}
		\hfill$ \square $
	\end{lemma}
	\begin{remark}
		Note that if one takes the original form of Theorem 3 of \cite{broder1984r_sup} and manipulates it until the present form is reached, the domain obtained would be $ n>r>0 $. This can be easily extended for the case $ r=0 $ as well, since we have that $ \stirstg{0}{n}{k} = \stirstg{1}{n}{k} $ for $ n>0 $. This case was missed in the original formula because it corresponded to a singularity.
	\end{remark}
	\begin{lemma}
		For $ n>r>0$ and $ k>0 $, we have that
		\begin{align}
		\stirstg{r}{n}{k}
		=
		(n-r) \stirstg{r}{n-1}{k}
		+
		\stirstg{r-1}{n-1}{k-1}.
		\label{eq:stirstg1_cross_rec2}
		\end{align}
		\label{lemma:stirstg1_cross_rec2}
		\hfill$ \square $
	\end{lemma}
	\begin{proof}
		We take the recurrence relation in \cref{defi:stirling_1_generalized} and we add and subtract the term $ (r-1)\stirstg{r}{n-1}{k} $ to it. That is,
		\begin{align*}
		\stirstg{r}{n}{k} =
		(n-1) \stirstg{r}{n-1}{k}
		-
		(r-1)\stirstg{r}{n-1}{k}
		+
		(r-1)\stirstg{r}{n-1}{k} 
		+
		\stirstg{r}{n-1}{k-1}.
		\end{align*}
		The first two terms of the right hand side simplify into $ (n-r) \stirstg{r}{n-1}{k}$ while the last two, according to \cref{lemma:stirstg1_cross_rec}, simplify into $ \stirstg{r-1}{n-1}{k-1} $ whenever $ n>r>0 $ and $ k>0 $, which concludes the proof.
	\end{proof}
	\begin{lemma}
		For $ r,N,k\geq0 $, we have that,
		\begin{align}
		\sum_{n = 0}^{N}
		\frac{\stirstg{r}{n+r}{k+r}}{n!}
		=
		\frac{\stirstg{r+1}{N+r+1}{k+r+1}}{N!}
		\label{eq:sum_frac_stirstg_factorial_1}.
		\end{align}
		\label{lemma:sum_frac_stirstg_factorial_1}
		\hfill$ \square $
	\end{lemma}
	\begin{proof}
		The proof is by induction. Assume \cref{eq:sum_frac_stirstg_factorial_1} to be satisfied for some $ N=  a-1 \geq 0$. Then,
		\begin{align*}
		\sum_{n = 0}^{a}
		\frac{\stirstg{r}{n+r}{k+r}}{n!}
		&=
		\sum_{n = 0}^{a-1}
		\frac{\stirstg{r}{n+r}{k+r}}{n!}
		+
		\frac{\stirstg{r}{a+r}{k+r}}{a!}
		\\
		&=
		\frac{\stirstg{r+1}{a+r}{k+r+1}}{(a-1)!}
		+
		\frac{\stirstg{r}{a+r}{k+r}}{a!}
		\\
		&=
		\frac{1}{a!}
		\left[
		a\stirstg{r+1}{a+r}{k+r+1}
		+
		\stirstg{r}{a+r}{k+r}
		\right].
		\end{align*}
		Consider the change of variables $ \overline{n}:= a+r+1 $, $ \overline{r} = r+1 $ and $ \overline{k} = k+r+1 $. Then, this becomes
		\begin{align*}
		\frac{1}{a!}
		\left[
		(\overline{n}-\overline{r})\stirstg{\overline{r}}{\overline{n}-1}{\overline{k}}
		+
		\stirstg{\overline{r}-1}{\overline{n}-1}{\overline{k}-1}
		\right].
		\end{align*}
		Note that the prerequisites for applying \cref{lemma:stirstg1_cross_rec2} are satisfied. That is, $ \overline{n}>\overline{r} $, $ \overline{r}>0 $ and $ \overline{k}>0 $ correspond to $ a+r+1 > r+1 $, $ r+1>0 $ and $ k+r+1>0 $ respectively. Therefore, this simplifies into
		\begin{align*}
		\frac{\stirstg{\overline{r}}{\overline{n}}{\overline{k}}}{a!}
		=
		\frac{\stirstg{r+1}{a+r+1}{k+r+1}}{a!},
		\end{align*}
		which concludes the induction step. For the base case $ N = 0 $, we have that $ \stirstg{r}{r}{k+r}
		=
		\stirstg{r+1}{r+1}{k+r+1} $, which is always satisfied since $ \delta_{r,k+r} = \delta_{r+1,k+r+1}$.
	\end{proof}
	\begin{lemma}
		For $ n\geq r\geq 0$ and $ k\geq 0 $, we have that
		\begin{align}
		\sum_{p=k}^{n-r} \binom{n-p}{r} \frac{\stirst{p}{k}}{p!}
		=
		\frac{\stirstg{r+1}{n+1}{k + r +1}}{(n-r)!}
		\label{eq:stirstg_frac_sum_1}.
		\end{align}
		\label{lemma:stirstg_frac_sum_1}
		\hfill$ \square $
	\end{lemma}
	\begin{proof}
		Consider $ n=r $. Then, the expression becomes
		\begin{align*}
		\sum_{p=k}^{0} \binom{n-p}{n} \frac{\stirst{p}{k}}{p!}
		=
		\frac{\stirstg{n+1}{n+1}{k + n +1}}{0!}.
		\end{align*}
		If $ k=0 $, the left hand side is $ \binom{n}{n}\frac{\stirst{0}{0}}{0!} =1 $. If $ k>0 $, the sum is empty so it is $ 0 $. The generalized Stirling number on the right simplifies into $ \delta_{n+1,k+n+1} $, which is one if $ k=0 $ and zero if $ k>0 $. Therefore, equality is achieved for all $ k\geq 0 $.\\
		Consider now $ r =0 $. Then, we have
		\begin{align*}
		\sum_{p=k}^{n} \binom{n-p}{0} \frac{\stirst{p}{k}}{p!}
		=
		\frac{\stirstg{1}{n+1}{k +1}}{n!}.
		\end{align*}
		If $ n\geq k $, this reduces to \cref{lemma:sum_frac_stirstg_factorial_1} (note that the missing indexes of the sum correspond to zero terms). If $ n<k $ then the left hand side is an empty sum and the Stirling number on the right is zero.
		The remaining cases that we have to prove are $ n>r>0 $, $ k\geq 0 $, which we prove by induction. Assume \cref{eq:stirstg_frac_sum_1} is satisfied for all $ (n,r,k) $ such that $ n=a-1 $ with $ n\geq r \geq 0 $ and $ k\geq 0 $. We now prove that it is satisfied for the cases $ (a,r,k) $ with $ a>r>0 $ and $ k\geq 0 $.
		Note that we have $ a-p\geq 1 $ for all $ p $ in the sum due to the fact that $ r>0 $ from assumption. Therefore, we can split 		\cref{eq:stirstg_frac_sum_1} into
		\begin{align*}
		&\sum_{p=k}^{a-r}
		\left[
		\binom{a-p-1}{r-1}
		+
		\binom{a-p-1}{r}
		\right] 
		\frac{\stirst{p}{k}}{p!}
		\\
		&=
		\sum_{p=k}^{(a-1)-(r-1)}
		\binom{(a-1)-p}{r-1}
		\frac{\stirst{p}{k}}{p!}
		+
		\sum_{p=k}^{(a-1)-r}
		\binom{(a-1)-p}{r}
		\frac{\stirst{p}{k}}{p!}.
		\end{align*}
		Note that the cases $ (a-1,r-1,k) $ and $ (a-1,r,k) $ satisfy the assumption. That is, $ a-1\geq r-1 \geq 0$ and $ a-1\geq r \geq 0 $ are true if $ a > r > 0 $. We can apply \cref{eq:stirstg_frac_sum_1} to those cases and obtain
		\begin{align*}
		\frac{1}{(a-r)!}
		\left[
		\stirstg{r}{a}{k+r}+(a-r)\stirstg{r+1}{a}{k+r+1}
		\right].
		\end{align*} 
		 From \cref{lemma:stirstg1_cross_rec2}, this simplifies into the right hand side of what we want to prove as long as $ a+1>r+1 $, $ r+1>0 $ and $ k+r+1>0 $, which are all satisfied under the current assumptions.\\
		 The base case $ n=0 $ has $ r=0 $, since $ n\geq r \geq 0 $. This was already covered by the previous cases $ n=r $ and $ r=0 $.
	\end{proof}
	\begin{lemma}
		For $ n\geq r\geq 0$ and $ k\geq 0 $, we have that
		\begin{align}
		\sum_{
			\substack{
				\mathbf{m} \geq \mathbf{1}_{k}\\
				\vert \mathbf{m} \vert \leq n - r
			}
		}
		\frac{1}{\prod_{i=1}^{k}m_i}
		\binom{n-\vert \mathbf{m} \vert}{r}
		=
		\frac{k!}{(n-r)!
		}
		\stirstg{r+1}{n+1}{ k + r +1}.
		\end{align}
		\label{lemma:stirstg_frac_sum_2}.
		\hfill$ \square $
	\end{lemma}
	\begin{proof}
		Split the sum in the left hand side into the two sums
		\begin{align*}
		\sum_{p=k}^{n-r}
		\sum_{
			\substack{
				\mathbf{m} \geq \mathbf{1}_{k}\\
				\vert \mathbf{m} \vert = p
			}
		}
		\frac{1}{\prod_{i=1}^{k}m_i}
		\binom{n-\vert \mathbf{m} \vert}{r}
		&=
		\sum_{p=k}^{n-r}
		\sum_{
			\substack{
				\mathbf{m} \geq \mathbf{1}_{k}\\
				\vert \mathbf{m} \vert = p
			}
		}
		\frac{1}{\prod_{i=1}^{k}m_i}
		\binom{n-p}{r}
		\\
		&=
		\sum_{p=k}^{n-r}
		\binom{n-p}{r}
		\sum_{
			\substack{
				\mathbf{m} \geq \mathbf{1}_{k}\\
				\vert \mathbf{m} \vert = p
			}
		}
		\frac{1}{\prod_{i=1}^{k}m_i}.
		\end{align*}
		From \cref{lemma:stirling_sums_1}, this is
		\begin{align*}
		k!
		\sum_{p=k}^{n-r}
		\binom{n-p}{r}
		\frac{\stirst{p}{k}}{p!}.
		\end{align*}
		the result is now straightforward from \cref{lemma:stirstg_frac_sum_1}.
	\end{proof}
	\begin{lemma}
		For $ n\geq k \geq 0 $, we have that,
		\begin{align}
		\sum_{
			\substack{
				\mathbf{m} \geq \mathbf{1}_{k}\\
				\vert \mathbf{m} \vert \leq n
			}
		}
		1
		=
		\binom{n}{k}.
		\label{eq:sums_1_abs_leq_n}
		\end{align}	
		\label{corol:sums_1_abs_leq_n}
		\hfill$ \square $
	\end{lemma}
	\begin{proof}
		The proof is by induction. Assume that the result applies for a given $ n\geq 0 $ and all $ k $ such that $ 0 \leq k\leq n $. We can split the following sum
		\begin{align*}
		\sum_{
			\substack{
				\mathbf{m} \geq \mathbf{1}_{k}\\
				\vert \mathbf{m} \vert \leq n+1
			}
		}
		1
		&=
		\sum_{
			\substack{
				\mathbf{m} \geq \mathbf{1}_{k}\\
				\vert \mathbf{m} \vert \leq n
			}
		}
		1
		+
		\sum_{
			\substack{
				\mathbf{m} \geq \mathbf{1}_{k}\\
				\vert \mathbf{m} \vert = n+1
			}
		}
		1.
		\end{align*}
		If $ k\geq1 $, we can rearrange the last sum so that we obtain
		\begin{align*}
		\sum_{
			\substack{
				\mathbf{m} \geq \mathbf{1}_{k}\\
				\vert \mathbf{m} \vert \leq n+1
			}
		}
		1
		=
		\sum_{
			\substack{
				\mathbf{m} \geq \mathbf{1}_{k}\\
				\vert \mathbf{m} \vert \leq n
			}
		}
		1
		+
		\sum_{
			\substack{
				\mathbf{m} \geq \mathbf{1}_{k-1}\\
				\vert \mathbf{m} \vert \leq n
			}
		}
		1.
		\end{align*}
		From assumption, whenever $ k\leq n $, this simplifies into $ \binom{n}{k} + \binom{n}{k-1} = \binom{n+1}{k} $. To complete the induction step we now prove the remaining cases $ k=0 $ and $ k=n+1 $. For the first one, the only valid index is the 0-tuple so the sum is always $ 1 = \binom{n+1}{0} $. For the second one the only valid index is the (n+1)-tuple of all ones so the sum is always $ 1 =\binom{n+1}{n+1} $. Therefore, the result applies to $ n+1 $ and all $ k $ such that $ 0\leq k\leq n+1 $.\\
		In the base case $ n,k=0 $, the only valid index is again the 0-tuple and the sum gives us $ 1 =\binom{0}{0} $.
	\end{proof}
	\begin{lemma}
		For $ \mathbf{M} \geq \mathbf{1}_{k} $, with $ n\geq \vert \mathbf{M} \vert $ and $ k\geq 0 $, we have that
		\begin{align}
		\sum_{
			\substack{
				\mathbf{m} \geq \mathbf{M} \\
				\vert \mathbf{m} \vert \leq n
			}
		}
		\prod_{i=1}^{k}
		\binom{m_i-1}{M_i-1}
		=
		\binom{n}{\vert \mathbf{M}\vert}.
		\label{eq:pascal_prod_sum}
		\end{align}
		\label{lemma:pascal_prod_sum}
		\hfill$ \square $
	\end{lemma}
	\begin{proof}
		Using \cref{corol:sums_1_abs_leq_n}, the left hand side can be expanded into
		\begin{align*}
				\sum_{
			\substack{
				\mathbf{m} \geq \mathbf{M} \\
				\vert \mathbf{m} \vert \leq n
			}
		}
		\prod_{i=1}^{k}
		\sum_{
			\substack{
				\overline{\mathbf{m}}_i \geq \mathbf{1}_{M_i-1}\\
				\vert \overline{\mathbf{m}}_i \vert \leq m_i-1
			}
		}
		1.
		\end{align*}
		The inner sum can be rearranged such that we get
		\begin{align*}
		\sum_{
			\substack{
				\mathbf{m} \geq \mathbf{M} \\
				\vert \mathbf{m} \vert \leq n
			}
		}
		\prod_{i=1}^{k}
		\sum_{
			\substack{
				\overline{\mathbf{m}}_i \geq \mathbf{1}_{M_i}\\
				\vert \overline{\mathbf{m}}_i \vert = m_i
			}
		}
		1.
		\end{align*}
		Distributing the product over the inner sum gives us
		\begin{align*}
		\sum_{
			\substack{
				\mathbf{m} \geq \mathbf{M} \\
				\vert \mathbf{m} \vert \leq n
			}
		}
		\sum_{
			\substack{
				\overline{\mathbf{m}}_1 \geq \mathbf{1}_{M_1}\\
				\ldots\\
				\overline{\mathbf{m}}_k \geq \mathbf{1}_{M_k}\\
				\vert \overline{\mathbf{m}}_1 \vert = m_1 \\
				\ldots\\
				\vert \overline{\mathbf{m}}_k \vert = m_k
			}
		}
		1.
		\end{align*}
		Note that we can completely remove the dependence on $ \mathbf{m} $ in this expression. In particular, note that for every $ i $ such that $ 1 \leq i\leq k $, we have that $ \overline{\mathbf{m}}_i \geq \mathbf{1}_{M_i} $. Then, $ \vert \overline{\mathbf{m}}_i \vert \geq M_i $. Since we also have that $ \vert \overline{\mathbf{m}}_i \vert = m_i$, this implies that $ m_i \geq M_i $ for all $ i $. Therefore, the expression $ \mathbf{m} \geq \mathbf{M} $ is redundant. Moreover, we can replace $ \vert \mathbf{m} \vert \leq n $ by $ \vert \overline{\mathbf{m}}_1 \vert + \ldots + \vert \overline{\mathbf{m}}_k \vert \leq n $. Defining $ \overline{\mathbf{m}} $ as the concatenation of $ \overline{\mathbf{m}}_1 , \ldots , \overline{\mathbf{m}}_k$, the expression simplifies into
		\begin{align*}
		\sum_{
			\substack{
				\overline{\mathbf{m}}\geq \mathbf{1}_{\vert \mathbf{M}\vert}\\
				\vert \overline{\mathbf{m}} \vert \leq n
			}
		}
		1.
		\end{align*} 
		Since $ n\geq \vert \mathbf{M} \vert $ from assumption and $ \vert \mathbf{M}\vert \geq k \geq 0 $, the result is now immediate from applying \cref{corol:sums_1_abs_leq_n} again.
	\end{proof}
	\begin{proof}[Proof of \cref{thm:sum_prod_comb_stirstg}]
		We first split the sum of the left hand side and reorganize it as
		\begin{align*}
		\sum_{
			\substack{
				\mathbf{p} \geq \mathbf{1}_{ r}\\
				\vert \mathbf{p} \vert \leq n - \vert \mathbf{M} \vert
			}
		}
		\frac{1}{\prod_{j=1}^{r}p_j}
		\left[
		\sum_{
			\substack{
				\mathbf{m} \geq \mathbf{M} \\
				\vert \mathbf{m} \vert \leq n - \vert \mathbf{p} \vert
			}
		}
		\prod_{i=1}^{k}
		\binom{m_i-1}{M_i-1}
		\right].
		\end{align*}
		Using \cref{lemma:pascal_prod_sum}, this simplifies into
		\begin{align*}
		\sum_{
			\substack{
				\mathbf{p} \geq \mathbf{1}_{r}\\
				\vert \mathbf{p} \vert \leq n - \vert \mathbf{M} \vert
			}
		}
		\frac{1}{\prod_{j=1}^{r}p_j}
		\binom{n-\vert \mathbf{p} \vert}{\vert \mathbf{M} \vert}.
		\end{align*}
		The result now follows from \cref{lemma:stirstg_frac_sum_2}.
	\end{proof}

	\bibliographystyle{siamplain}
	\bibliography{references_sup}